\documentclass[10pt]{amsart}  

\usepackage[utf8]{inputenc}

\usepackage[left=1in,right=1in,top=1in,bottom=1.17in]{geometry}

\usepackage{graphicx,amssymb,amsmath,amsthm,wrapfig}
\usepackage{caption}
\usepackage{subcaption}
\usepackage{indentfirst}
\usepackage{cancel}
\usepackage{lipsum}
\usepackage{epstopdf}
\usepackage{epsfig}
\usepackage{comment}
\usepackage{enumerate,color}
\usepackage{empheq}
\usepackage{hyperref}

\everymath{\displaystyle}
\theoremstyle{plain}
\newtheorem{theorem}{Theorem}[section]
\newtheorem{corollary}[theorem]{Corollary}
\newtheorem{lemma}[theorem]{Lemma}

\newtheorem{proposition}[theorem]{Proposition}
\newtheorem{remark}[theorem]{Remark}

\numberwithin{equation}{section}

\newcommand{\RNum}[1]{\uppercase\expandafter{\romannumeral #1\relax}}

\title[Stability of composite waves for the Navier--Stokes--Poisson system]{Asymptotic stability of composite waves of shock profile and rarefaction for the Navier--Stokes--Poisson system}
\author[Wanyong Shim]{Wanyong Shim}

\address{(Wanyong Shim) Department of Mathematical Sciences, Korea Advanced Institute of Science and Technology, Daejeon, 34141, Korea}
\email{wanyong.shim@gmail.com}

\thanks{\textbf{Acknowledgment. }This work was supported by Samsung Science and Technology Foundation under Project Number SSTF-BA2102-01. The author thanks Professor Moon-Jin Kang for suggesting this problem.}

\begin{document}

\begin{abstract}
We study the stability of composite waves consisting of a shock profile and a rarefaction wave for the one-dimensional isothermal Navier--Stokes--Poisson (NSP) system, which describes the ion dynamics in a collision-dominated plasma. More precisely, we prove that if the initial data are sufficiently close in the $H^2$ norm to the Riemann data corresponding to a solution consisting of a shock and a rarefaction wave of the associated quasi-neutral Euler system, then the solution to the Cauchy problem for the NSP system converges, up to a dynamical shift, to a superposition of the corresponding shock profile and the rarefaction wave as time tends to infinity. Our proof is based on the method of $a$-contraction with shifts, which has recently been applied to the Navier--Stokes equations to establish the asymptotic stability of composite waves. To adapt this method to the NSP system, we employ a modulated relative functional introduced in our previous work on the stability of single shock profiles. \\

\end{abstract}

\maketitle

\tableofcontents

\section{Introduction}

\subsection{The Navier--Stokes--Poisson system}

We consider the one-dimensional compressible Navier--Stokes--Poisson (NSP) system, which serves as a model for the dynamics of ions in an isothermal plasma in the collision-dominated regime \cite{GGKS}. In Lagrangian mass coordinates, the NSP system is written as
\begin{subequations} \label{NSP}
\begin{align}
& \label{NSP11} v_t - u_x = 0,\\
& \label{NSP22} u_t + p(v)_x   =  \left( \frac{\nu u_x}{v} \right)_x - \frac{\phi_x}{v}, \\
& \label{NSP33} - \lambda^2 \left( \frac{\phi_x}{v} \right)_x = 1 - ve^{\phi}
\end{align}
\end{subequations}
for $t>0$ and $x \in \mathbb{R}$. Here $v=\tfrac{1}{n}$ is the specific volume for $n>0$, the density of ions, and $\phi$ is the electric potential. The function $p(v)$ denotes the pressure given by $p(v) = K v^{-1}$. The constants $K>0$, $\nu>0$ and $\lambda>0$ represent the absolute temperature, viscosity coefficient and Debye length, respectively. For simplicity, we normalize by setting $K=1$, $\nu = 1$, and $\lambda = 1$, as all the results in this paper hold for any positive constants $K$, $\nu$, and $\lambda$. In the Poisson equation \eqref{NSP33}, we have assumed that the electron density $n_e$ is determined by the Boltzmann relation, $n_e = e^\phi$, which is justified by the physical observation that electrons reach the equilibrium state much faster than ions for varying potential in a plasma \cite{Ch}.

This paper is concerned with the large-time behavior of solutions to the NSP system. In particular, we study the asymptotic stability of composite waves consisting of a shock profile and a rarefaction wave. For this purpose, we consider the associated Cauchy problem for \eqref{NSP} with the initial data:
\begin{equation} \label{ic}
(v,u)(0,x) = (v_0,u_0)(x), \quad \lim_{x \rightarrow \pm \infty} (v_0,u_0)(x) = (v_\pm, u_\pm)
\end{equation}
for constant end states $(v_\pm, u_\pm)$ to be specified below. The far-field behavior of $\phi$ is prescribed as
\begin{equation} \label{qnc}
\lim_{x \rightarrow \pm \infty} \phi(t,x) = \phi_\pm,
\end{equation}
where the constants $\phi_\pm$ are given by the quasi-neutral condition \( \phi_\pm = -\ln{v_\pm} \).

We note that, as in the Navier--Stokes equations, the time-asymptotic behavior of \eqref{NSP}–\eqref{ic} is expected to be governed by the solution of a Riemann problem for the associated hyperbolic system. Here it corresponds to the quasi-neutral Euler system:
\begin{subequations} \label{Euler}
\begin{align}
& \label{Euler1} v_t - u_x = 0,\\
& \label{Euler2} u_t + \tilde{p}(v)_x  = 0, \\
& \label{Euler3} \phi = - \ln{v},
\end{align}
\end{subequations}
where the modified pressure $\tilde{p}$ is defined by
\begin{equation} \label{p(v)}
\tilde{p}(v) := p(v) + \frac{1}{v}.
\end{equation}
This system can be obtained by taking the formal limits $\lambda \to 0$ and $\nu \to 0$ in \eqref{NSP}.

\subsection{Composite waves of shock profile and rarefaction wave} \label{Sec_1.2}

To determine the far-field states $(v_\pm,u_\pm)$ in \eqref{ic}, we introduce the Riemann problem given by the quasi-neutral Euler system \eqref{Euler} with the initial data
\begin{equation} \label{Eic}
(v,u)(0,x) = \begin{cases}
(v_-,u_-), & x<0, \\
(v_+,u_+), & x>0.
\end{cases}
\end{equation}
In this paper, we consider the end states $(v_\pm,u_\pm)$ such that the Riemann problem for \eqref{Euler} yields a superposition of a rarefaction wave and a shock. Without loss of generality, we focus on the case where the Riemann solution consists of a $1$-rarefaction and a $2$-shock. In what follows, we describe this by recalling the theory of Riemann problems for the Euler system \eqref{Euler1}--\eqref{Euler2} with \eqref{Eic}; see \cite{Smoller}.

The Euler system \eqref{Euler1}--\eqref{Euler2} can be rewritten as
\begin{equation*}
\begin{pmatrix}
v \\ u
\end{pmatrix}_t + \begin{pmatrix}
0 & -1 \\
\tilde{p}'(v) & 0
\end{pmatrix} \begin{pmatrix}
v \\ u
\end{pmatrix}_x = 0,
\end{equation*}
which is strictly hyperbolic, meaning that the coefficient matrix has real and distinct eigenvalues
\begin{equation} \label{eig}
\lambda_1 = - \sqrt{- \tilde{p}'(v)} < 0, \quad \lambda_2 = \sqrt{-\tilde{p}'(v)} > 0.
\end{equation}
Then, for any left-end state $(v_-,u_-)$, the $1$-rarefaction curve $R_1$ and $2$-shock curve $S_2$ are defined by
\begin{equation*}
\begin{split}
R_1(v_-,u_-) & := \bigg\{ (v,u): v > v_-, \ u = u_- - \int_{v_-}^{v} \lambda_1(s) \, ds  \bigg\}, \\
S_2(v_-,u_-) & := \left\{ (v,u): v > v_-, \ u = u_- - \sqrt{(v-v_-)\left(\tilde{p}(v_-) - \tilde{p}(v) \right)}  \right\}.
\end{split}
\end{equation*}
To ensure that the Riemann solution of \eqref{Euler1}--\eqref{Euler2} with \eqref{Eic} consists of a $1$-rarefaction wave and a $2$-shock, we restrict the right-end state $(v_+,u_+)$ to lie in the region between the two curves $R_1$ and $S_2$. Specifically, for given $(v_-,u_-) \in \mathbb{R}_+ \times \mathbb{R}$ and $\delta_0 \in \mathbb{R}_+$, we define a subset of the $(v,u)$-plane as
\begin{equation*}
\begin{split}
\Gamma_{\delta_0}(v_-,u_-) := \bigg\{ & (v,u)\in \mathbb{R}_+\times \mathbb{R}: 
\ v-v_-\in (0,\delta_0), \\
&\ -\sqrt{(v-v_-)\left(\tilde{p}(v_-) - \tilde{p}(v) \right)} 
   < u-u_- 
   < - \int_{v_-}^v \lambda_1(s) \, ds 
\bigg\},
\end{split}
\end{equation*}
where $\mathbb{R}_+ := (0, +\infty)$. The state \((v_+, u_+)\) is chosen to satisfy
\begin{equation} \label{Gammain}
(v_+,u_+) \in \Gamma_{\delta_0}(v_-,u_-)
\end{equation}
for sufficiently small $\delta_0$, noting that this implies $|u_+-u_-| \sim |v_+-v_-| < \delta_0$.

For far-field data satisfying \eqref{Gammain}, the associated Riemann solution of \eqref{Euler} and its elementary wave components are defined as follows. There exists a unique intermediate state $(v_m,u_m)$ such that $(v_m,u_m) \in R_1(v_-,u_-)$ and $(v_+,u_+) \in S_2(v_m,u_m)$. The solution $(v,u,\phi)$ of the Riemann problem \eqref{Euler}--\eqref{Eic} is then given by a superposition of a \textit{$1$-rarefaction} wave and a \textit{$2$-shock} wave:
\begin{equation} \label{Eshra}
(v,u,\phi)(t,x) = (v^r,u^r,\phi^r)(t,x) + (v^s,u^s,\phi^s)(t,x) - (v_m,u_m,\phi_m).
\end{equation}
The $1$-rarefaction wave $(v^r,u^r,\phi^r)$ is defined as a self-similar solution determined by the eigenvalue $\lambda_1$ and the Riemann invariant $\textstyle z_1(v,u) = u + \int^v \lambda_1(s) \, ds$:
\begin{equation} \label{def_ra}
\lambda_1 \left( v^r(t,x) \right) = \begin{cases}
\lambda_1 (v_-), & x < \lambda_1(v_-) t, \\
\tfrac{x}{t}, & \lambda_1(v_-)t \leq x \leq \lambda_1(v_m)t, \\
\lambda_1(v_m), & x > \lambda_1(v_m) t,
\end{cases}
\end{equation}
together with
\begin{equation} \label{def_ra1}
z_1 \left( v^r(t,x), u^r(t,x) \right) = z_1 (v_-,u_-) = z_1 (v_m,u_m).
\end{equation}
The $2$-shock wave $(v^s,u^s,\phi^s)$ is defined by
\begin{equation} \label{sh_def}
(v^s,u^s)(t,x) = \begin{cases}
(v_m,u_m), & x < \sigma t, \\
(v_+,u_+), & x > \sigma t,
\end{cases}
\end{equation}
where the shock speed $\sigma$ is given by the Rankine--Hugoniot jump condition
\begin{equation} \label{RH}
\begin{split}
& - \sigma (v_+ - v_m) - (u_+ - u_m ) = 0, \\
& - \sigma (u_+ - u_m) + \left( \tilde{p}(v_+) - \tilde{p}(v_m) \right) = 0.
\end{split}
\end{equation}
Note that the electric potentials $\phi^r$, $\phi^s$, and $\phi_m$ are subsequently determined from $v^r$, $v^s$, and $v_m$, respectively, via the quasi-neutral relation \eqref{Euler3}.

We now construct an asymptotic profile (up to a spatial shift) for the NSP system that corresponds to the Riemann solution \eqref{Eshra} of the quasi-neutral Euler system \eqref{Euler}. To this end, we introduce a viscous-electrostatic counterpart of the $2$-shock wave $(v^s,u^s,\phi^s)$. Consider the Cauchy problem for the NSP system \eqref{NSP} with the initial data
\begin{equation*}
(v,u)(0,x) = (v_0,u_0)(x)
\end{equation*}
with
\begin{equation*}
\lim_{x \to - \infty}{(v_0,u_0)(x)} = (v_m,u_m), \quad \lim_{x \to +\infty}{(v_0,u_0)(x)} = (v_+,u_+).
\end{equation*}
This problem admits a smooth traveling wave $(\bar{v}^S,\bar{u}^S,\bar{\phi}^S)(x-\sigma t)$, called the $2$-shock profile, defined by the governing ODEs:
\begin{equation} \label{SHODE}
\begin{split}
& -\sigma \left( \bar{v}^S \right)' - \left( \bar{u}^S \right)' =0, \\
& - \sigma \left( \bar{u}^S \right)' + p(\bar{v}^S)' = \bigg( \frac{\left(\bar{u}^S\right)'}{\bar{v}^S} \bigg)' -  \frac{\left(\bar{\phi}^S\right)'}{\bar{v}^S}, \\
& - \bigg( \frac{\left(\bar{\phi}^S \right)'}{\bar{v}^S} \bigg)' = 1 - \bar{v}^S e^{\bar{\phi}^S},
\end{split}
\end{equation}
where $'$ denotes $\textstyle \frac{d}{d\xi}$ with $\xi=x-\sigma t$. The existence and properties of this profile will be presented in Section~\ref{Sec_2.2}. Our goal is to establish the asymptotic stability of the following composite wave, consisting of the $1$-rarefaction wave of \eqref{Euler} and the $2$-shock profile of \eqref{NSP}:
\begin{equation} \label{composite}
(v^r,u^r,\phi^r)\big(\frac{x}{t}\big) + (\bar{v}^S,\bar{u}^S,\bar{\phi}^S)(x-\sigma t) - (v_m,u_m,\phi_m).
\end{equation}

\subsection{Literature review}

Before stating our main result, we briefly review related works on the large-time behavior of solutions to viscous compressible fluid models with prescribed constant far-field states. We begin by recalling several notable results on the stability of viscous shock profiles, rarefaction waves, and their superpositions for the one-dimensional Navier--Stokes (NS) equations.

For shock profiles, the problem was first studied by Matsumura and Nishihara \cite{MN}, who proved time-asymptotic stability under perturbations with zero integral, referred to as the \emph{zero mass} condition, by using the anti-derivative method (see also \cite{Go, KM}). This restriction was later removed by Liu \cite{L}, Liu and Zeng \cite{LZ}, and Szepessy and Xin \cite{SX} through pointwise estimate methods. Howard and Zumbrun \cite{HZ} and Mascia and Zumbrun \cite{MZ, MZ1} established linear and nonlinear stability using pointwise semigroup techniques. Kang and Vasseur \cite{KV3} obtained $L^2$ stability by using the $a$-contraction method with shifts (see also \cite{KV4, KV5, KVW2, KVW3}). The results in \cite{KV3, MZ1} do not require the zero mass condition.

For rarefaction waves, Matsumura and Nishihara proved time-asymptotic stability for the compressible and isentropic NS equations in \cite{MN1, MN2}, using direct energy methods. Such direct energy methods are incompatible with the anti-derivative method used for viscous shocks. Due to this incompatibility, the asymptotic stability of composite waves consisting of a viscous shock and a rarefaction wave was first established only recently by Kang, Vasseur, and Wang \cite{KVW2}. Their analysis handled the different types of waves within a relative entropy framework, employing the theory of $a$-contraction with shifts. See also the notable work \cite{HKK} for the case of a composition of two viscous shocks.

We now turn to the one-dimensional NSP system. Duan, Liu, and Zhang \cite{DLZ} proved asymptotic stability of shock profiles under the zero mass constraint, and Duan and Liu also established the stability of rarefaction waves in \cite{DL}. In addition, Kang, Kwon, and the present author \cite{KKSh} proved the stability of shock profiles up to a dynamical shift using the $a$-contraction method. Building upon these works, we establish here the asymptotic stability of composite waves consisting of a shock profile and a rarefaction wave for the NSP system. For further results, we refer to \cite{LMY, Zh} for shocks in different settings, and to \cite{ZZZ} for rarefactions under large initial perturbations.

\subsection{Main result}

For completeness, we recall the definition of the set \( \Gamma_{\delta_0}(v_-, u_-) \), which appears in the main theorem:
\begin{equation} \label{Gamma}
\begin{split}
\Gamma_{\delta_0}(v_-,u_-) = \bigg\{ & (v,u)\in \mathbb{R}_+\times \mathbb{R}: 
v-v_-\in (0,\delta_0), \\
& \ -\sqrt{(v-v_-)\left(\tilde{p}(v_-) - \tilde{p}(v) \right)} 
   < u-u_- 
   < \int_{v_-}^v \sqrt{-\tilde{p}'(s)} \, ds 
\bigg\}.
\end{split}
\end{equation}
The main result is stated below.

\begin{theorem} \label{Main}
Given any left-end state $(v_-,u_-) \in \mathbb{R}_+ \times \mathbb{R}$, there exist positive constants $\delta_0$ and $\varepsilon_0$ such that the following statement holds.

For any $(v_+,u_+) \in \Gamma_{\delta_0}(v_-,u_-)$, let $(v_m, u_m, \phi_m)$ be the unique state such that
\begin{equation*}
(v_m, u_m) \in R_1(v_-, u_-), \quad (v_+, u_+) \in S_2(v_m, u_m), \quad \phi_m = - \ln{v_m}.
\end{equation*}
Let \( \textstyle (v^r, u^r, \phi^r)(\frac{x}{t}) \) denote the $1$-rarefaction wave of \eqref{Euler} with the end states $(v_-,u_-,- \ln{v_-})$ and $(v_m,u_m,\phi_m)$, and \( (\bar{v}^S, \bar{u}^S, \bar{\phi}^S)(x - \sigma t) \) the $2$-shock profile of \eqref{SHODE} with the end states $(v_m,u_m,\phi_m)$ and $(v_+,u_+,-\ln{v_+})$. Suppose that the initial data $(v_0,u_0)$ satisfies
\begin{equation*}
\sum_{\pm}{\lVert (v_0-v_\pm,u_0-u_\pm) \rVert_{L^2(\mathbb{R}_\pm)}} + \lVert (v_{0x},u_{0x}) \rVert_{H^1(\mathbb{R})} < \varepsilon_0,
\end{equation*}
where $\mathbb{R}_- := -\mathbb{R}_+ = (-\infty,0)$. Then the Cauchy problem \eqref{NSP}--\eqref{qnc} admits a unique global-in-time strong solution $(v,u,\phi)$.\footnote{On any finite time interval, the solution has the regularity stated in Proposition~\ref{LocalE}, in the moving coordinate $\xi=x-\sigma t$.} Moreover, there exists a shift function $X(t)$ such that the solution $(v,u,\phi)(t,x)$ and the shift $X(t)$ satisfy
\begin{equation} \label{asympt}
\begin{split}
\lim_{t \to +\infty} \, \Big\lVert (v,u,\phi) (t,\cdot) - \Big( & v^r\big( \frac{\cdot}{t} \big) + \bar{v}^S ( \cdot - \sigma t - X(t) ) - v_m, \\
& u^r\big( \frac{\cdot}{t} \big) + \bar{u}^S ( \cdot - \sigma t - X(t) ) - u_m, \\
& \phi^r\big( \frac{\cdot}{t} \big) + \bar{\phi}^S ( \cdot - \sigma t - X(t) ) - \phi_m \Big) \Big\rVert_{L^\infty(\mathbb{R})} = 0
\end{split}
\end{equation}
and
\begin{equation} \label{limX}
\lim_{t \to +\infty}{|\dot{X}(t)|} = 0.
\end{equation}
\end{theorem}

\begin{remark}
From \eqref{limX}, we see that the shift function $X(t)$ grows at most sublinearly as $t\to +\infty$, i.e.,
\begin{equation*}
\lim_{t \to +\infty}{\frac{X(t)}{t}} =0.
\end{equation*}
This implies that the shifted shock profile $(\bar{v}^S,\bar{u}^S,\bar{\phi}^S)(x-\sigma t - X(t))$ approaches the original profile $(\bar{v}^S,\bar{u}^S,\bar{\phi}^S)\allowbreak (x-\sigma t)$ as \( t \to +\infty \), in the sense that the shift becomes asymptotically negligible.
\end{remark}

\vspace{1em}
\textit{Plan of the paper}. In Section~\ref{Sec_2}, we provide some preliminaries, including the construction of an approximate composite wave and the basic properties of its components. Section~\ref{Sec_3} states the a priori estimate for perturbations around the approximate composite wave and briefly outlines the proof of Theorem~\ref{Main}, which is given in the Appendices. The proof of the a priori estimate is provided in the subsequent sections. In Section~\ref{Sec_4}, we establish an energy estimate using the method of $a$-contraction with shifts. Sections~\ref{Sec_5} and \ref{Sec_6} are devoted to closing the estimate by deriving elliptic and higher-order estimates. The Appendices contain basic elliptic estimates related to the Poisson equation, the deferred proofs of Lemmas~\ref{uxi}, \ref{uxixi}, and \ref{vxixi}, as well as the complete proof of Theorem~\ref{Main}.


\section{Preliminaries} \label{Sec_2}

In this section, we construct an approximate composite wave consisting of a smooth approximation of the $1$-rarefaction wave and the $2$-shock profile shifted by a function $X(t)$. We also reformulate the NSP system into divergence form in the moving frame $(t, x - \sigma t)$.

\subsection{Smooth approximation of rarefaction wave} \label{Sec_2.1}

Following the approach in \cite{DL, MN1}, we construct a smooth approximation to the $1$-rarefaction wave using the solution $w(t,x)$ to the inviscid Burgers equation:
\begin{subequations} \label{Bur}
\begin{align}
& \label{eqB} w_t + ww_x = 0, \\
& \label{init_B} w(0,x) = w_0(x) := \frac{w_m+w_-}{2} + \frac{w_m-w_-}{2} \tanh{x},
\end{align}
\end{subequations}
where $w_- := \lambda_1(v_-)$ and $w_m := \lambda_1(v_m)$ with $\lambda_1(v) = - \sqrt{-\tilde{p}'(v)}$. We define the approximation $(\bar{v}^R,\bar{u}^R,\bar{\phi}^R)$ to the $1$-rarefaction $(v^r,u^r,\phi^r)$, defined in Section~\ref{Sec_1.2}, as
\begin{equation} \label{def_ara}
\begin{split}
& \bar{v}^R(t,x) = \lambda_1^{-1} \big( w(1+t,x) \big), \\
& z_1 (\bar{v}^R,\bar{u}^R) (t,x) = z_1(v_-,u_-) = z_1(v_m,u_m), \\
& \bar{\phi}^R = - \ln{\bar{v}^R} \quad \text{(by the quasi-neutral condition \eqref{Euler3})},
\end{split}
\end{equation}
where $w(t,x)$ is the smooth solution to \eqref{Bur} and $\textstyle z_1(v,u)=u+\int^v \lambda_1(s) \, ds$. This approximate rarefaction wave satisfies the quasi-neutral Euler system:
\begin{equation} \label{eq:Ra}
\begin{split}
& \bar{v}^R_t - \bar{u}^R_x = 0, \\
& \bar{u}^R_t + \tilde{p}(\bar{v}^R)_x = 0, \\
& \bar{\phi}^R = - \ln{\bar{v}^R},
\end{split}
\end{equation}
and
\begin{equation*}
\lim_{x \to - \infty}{(\bar{v}^R,\bar{u}^R)(0,x)} = (v_-,u_-), \quad \lim_{x \to + \infty}{(\bar{v}^R,\bar{u}^R)(0,x)} = (v_m,u_m).
\end{equation*}

The lemma below follows from Lemmas~2.1 and 2.2 in \cite{MN1}.

\begin{lemma} \label{rarefaction}
Let $\delta_R$ denote the rarefaction wave strength as $\delta_R := \lvert v_m - v_- \rvert \sim \lvert u_m - u_- \rvert$. The approximate smooth solution $(\bar{v}^R,\bar{u}^R)$ given by \eqref{def_ara} and \eqref{Bur} satisfies the following properties:
\begin{enumerate}
\item $\bar{u}^R_x = \bar{v}^R w_x > 0 $, $\bar{v}^R_x = (\bar{v}^R/\sqrt{2}) \bar{u}^R_x >0$, and $\bar{\phi}^R_x = - (\bar{v}^R)^{-1} \bar{v}^R_x <0$ for all $x\in\mathbb{R}$ and $t \geq 0$.
\item For any $p \in [1,+\infty]$, there exists a constant $C>0$ such that for all $t \geq 0$,
\begin{equation*}
\begin{split}
& \lVert \partial_x (\bar{v}^R, \bar{u}^R) \rVert_{L^p(\mathbb{R})} \leq C \min{\left\{ \delta_R, \delta_R^{1/p} (1+t)^{-1+1/p} \right\} }, \\
& \lVert \partial_x^2 (\bar{v}^R,\bar{u}^R) \rVert_{L^p(\mathbb{R})} \leq C \min{\left\{ \delta_R, (1+t)^{-1} \right\}}, \\
& \lVert \partial_x^3 (\bar{v}^R,\bar{u}^R) \rVert_{L^p(\mathbb{R})} \leq C \min{\left\{ \delta_R, (1+t)^{-1} \right\}}, \\
& \lvert \bar{u}^R_{xx} \rvert \leq C \lvert \bar{u}^R_x \rvert, \quad x \in \mathbb{R}.
\end{split}
\end{equation*}
\item For $x \geq \lambda_1(v_m)(1+t)$, $t \geq 0$, it holds that
\begin{equation*}
\begin{split}
& \lvert (\bar{v}^R,\bar{u}^R)(t,x) - (v_m,u_m) \rvert \leq C \delta_R e^{-2 \lvert x - \lambda_1(v_m) (1+t) \rvert }, \\
& \lvert (\bar{v}^R_x,\bar{u}^R_x)(t,x) \rvert \leq C \delta_R e^{-2 \lvert x -  \lambda_1(v_m) (1+t) \rvert }.
\end{split}
\end{equation*}
\item For $x \leq \lambda_1(v_-)t$, $t \geq 0$, it holds that
\begin{equation*}
\begin{split}
& \lvert (\bar{v}^R,\bar{u}^R)(t,x) - (v_-,u_-) \rvert \leq C \delta_R e^{-2 \lvert x - \lambda_1(v_-)t \rvert}, \\
& \lvert (\bar{v}^R_x,\bar{u}^R_x)(t,x) \rvert \leq C \delta_R e^{-2 \lvert x- \lambda_1(v_-) t \rvert}.
\end{split}
\end{equation*}
\item $\lim_{t \to \infty} \sup_{x \in \mathbb{R}} \big\lvert (\bar{v}^R,\bar{u}^R)(t,x) - (v^r,u^r)(x/t) \big\rvert = 0$.
\end{enumerate}
\end{lemma}

\subsection{Shock profile} \label{Sec_2.2}

We now turn to the $2$-shock profile which is defined by the ODE system \eqref{SHODE} together with the far-field condition
\begin{equation} \label{ffc}
\lim_{\xi \rightarrow - \infty} (\bar{v}^S,\bar{u}^S,\bar{\phi}^S) (\xi) = (v_m,u_m,\phi_m), \quad \lim_{\xi \rightarrow + \infty} (\bar{v}^S,\bar{u}^S,\bar{\phi}^S) (\xi) = (v_+,u_+,\phi_+).
\end{equation}
Note that $(v_+,u_+)$ lies on the $2$-shock curve $S_2(v_m,u_m)$ if and only if the end states satisfy the Rankine--Hugoniot condition \eqref{RH} and the Lax entropy condition
\begin{equation} \label{Laxc2}
\lambda_2(v_+) < \sigma < \lambda_2(v_m).
\end{equation}
The existence and uniqueness of the $2$-shock profile are stated in the following lemma.
\begin{lemma} [\cite{DLZ}, Theorem~1.1] \label{Prop.1.1}
For any $(v_m,u_m )\in \mathbb{R}_+ \times \mathbb{R}$, there exists a constant $\delta_1$ such that, if the end state $(v_+,u_+) \in S_2(v_m,u_m)$ satisfies
\begin{equation*}
\lvert u_+ - u_m \rvert \sim \lvert v_+ - v_m \rvert =: \delta_S < \delta_1,
\end{equation*}
then \eqref{SHODE} admits a unique (up to a shift) solution $(\bar{v}^S,\bar{u}^S,\bar{\phi}^S)(\xi)$ satisfying
\begin{equation} \label{vup'}
\sigma \bar{v}^S_\xi = - \bar{u}^S_\xi > 0, \quad \underline{C} \bar{u}^S_\xi \leq \bar{\phi}^S_\xi \leq \overline{C} \bar{u}^S_\xi
\end{equation}
for some positive constants $\underline{C}, \overline{C}$. Moreover, the unique solution satisfying $ \textstyle \bar{v}^S(0) = \frac{v_m + v_+}{2}$ verifies the derivative bounds
\begin{equation} \label{shderiv}
\begin{cases}
\displaystyle \bigg\lvert \frac{d^k}{d\xi^k} (\bar{v}^S-v_+, \bar{u}^S - u_+, \bar{\phi}^S - \phi_+ ) \bigg\rvert \leq C_k \delta_S^{k+1} e^{-\theta \delta_S \lvert \xi \rvert}, & \xi > 0 \\
\displaystyle \bigg\lvert \frac{d^k}{d\xi^k} (\bar{v}^S - v_m, \bar{u}^S - u_m, \bar{\phi}^S - \phi_m ) \bigg\rvert \leq C_k \delta_S^{k+1} e^{-\theta \delta_S \lvert \xi \rvert}, & \xi < 0
\end{cases}
\end{equation}
for $ k \in \mathbb{N} \cup \{0\}$, where $C_k>0$ and $\theta>0$ are generic constants.
\end{lemma}

\subsection{Construction of approximate composite wave}

\subsubsection{Reformulation in a moving frame}

We begin by rewriting the NSP system \eqref{NSP} in divergence form using the Poisson equation \eqref{NSP33}. This reformulation is crucial in establishing a priori estimates, in particular the zeroth-order estimate in Section~\ref{Sec_4} and Section~\ref{Sec_6.1}. From \eqref{NSP33}, we derive the following identity:
\begin{equation} \label{trf}
\frac{\phi_x}{v} = \bigg[ \frac{1}{v} + \frac{\lambda^2}{v} \left( \frac{\phi_x}{v} \right)_x - \frac{\lambda^2}{2} \left( \frac{\phi_x}{v} \right)^2 \bigg]_x.
\end{equation}
Substituting this into \eqref{NSP22}, we rewrite the system \eqref{NSP} in the moving coordinates $(t,\xi)$, with $\xi = x -\sigma t$, as
\begin{subequations} \label{NSP''}
\begin{align}
& \label{NSP1''} v_t - \sigma v_\xi - u_\xi = 0, \\
& \label{NSP2''}u_t - \sigma u_\xi + \tilde{p}(v)_\xi = \left( \frac{u_\xi}{v} \right)_\xi + \Phi(v,\phi)_\xi, \\
& \label{NSP3''} - \left( \frac{\phi_\xi}{v} \right)_\xi = 1 - v e^\phi.
\end{align}
\end{subequations}
where $\tilde{p}$ is as in \eqref{p(v)} and the electric force $\Phi$ arising from non-neutral plasma density is given by
\begin{equation} \label{Phi}
\Phi (v,\phi) := \frac{1}{2} \left( \frac{\phi_\xi}{v} \right)^2 - \frac{1}{v}  \left( \frac{\phi_\xi}{v} \right)_\xi.
\end{equation}

\subsubsection{Approximate composite wave}

As Lemma~\ref{Prop.1.1} indicates, the shock profile is translation invariant as a solution to \eqref{SHODE} with \eqref{ffc}. Motivated by this invariance, we consider a translated shock profile $(\bar{v}^S,\bar{u}^S,\bar{\phi}^S)(\xi-X(t))$, where the shift function $X(t)$ will be determined later. For notational convenience, we write
\begin{equation*}
(\bar{v}^S,\bar{u}^S,\bar{\phi}^S) = (\bar{v}^S,\bar{u}^S,\bar{\phi}^S)(\xi-X(t))
\end{equation*}
throughout the rest of the paper. The approximate composite wave $(\bar{v},\bar{u},\bar{\phi})$ is then constructed by superposing the approximate rarefaction wave $\left( \bar{v}^R, \bar{u}^R, \bar{\phi}^R \right)$ and the shifted shock profile $ \left( \bar{v}^S, \bar{u}^S, \bar{\phi}^S \right)$ in the moving frame:
\begin{equation} \label{ACW}
\begin{split}
(\bar{v},\bar{u},\bar{\phi})(t,\xi) & := ( \bar{v}^R, \bar{u}^R, \bar{\phi}^R ) (t, \xi + \sigma t)  + ( \bar{v}^S, \bar{u}^S, \bar{\phi}^S ) (\xi - X(t)) - (v_m,u_m,\phi_m).
\end{split}
\end{equation}
We refer to this profile as \textit{approximate composite wave} since the rarefaction component has been replaced by its smooth approximation.

We conclude this section by presenting the governing equations in $(t,\xi)$ for the approximate rarefaction wave, the shifted shock profile, and the approximate composite wave. By \eqref{eq:Ra}, we have
\begin{subequations} \label{eq:ra}
\begin{align}
& \label{ra1} \bar{v}^R_t - \sigma \bar{v}^R_\xi - \bar{u}^R_\xi = 0, \\
& \label{ra2} \bar{u}^R_t - \sigma \bar{u}^R_\xi + \tilde{p}(\bar{v}^R)_\xi = 0, \\
& \label{ra3} \bar{\phi}^R = - \ln{\bar{v}^R}.
\end{align}
\end{subequations}
Similarly, by \eqref{SHODE}, the shifted shock profile $(\bar{v}^S, \bar{u}^S, \bar{\phi}^S)$ satisfies
\begin{subequations} \label{eq:sh'}
\begin{align}
& \label{sh1'} \bar{v}^S_t - \sigma \bar{v}^S_\xi - \bar{u}^S_\xi + \dot{X}(t) \bar{v}^S_\xi = 0, \\
& \label{sh2'} \bar{u}^S_t - \sigma \bar{u}^S_\xi + \tilde{p}(\bar{v}^S)_\xi + \dot{X}(t) \bar{u}^S_\xi = \bigg( \frac{\bar{u}^S_\xi}{\bar{v}^S} \bigg)_\xi + \Phi(\bar{v}^S,\bar{\phi}^S)_\xi, \\
& \label{sh3'} - \bigg( \frac{\bar{\phi}^S_\xi}{\bar{v}^S} \bigg)_\xi = 1 - \bar{v}^S e^{\bar{\phi}^S}.
\end{align}
\end{subequations}
The approximate composite wave $(\bar{v}, \bar{u}, \bar{\phi})$, defined in \eqref{ACW}, then satisfies
\begin{subequations}
\begin{align}
& \label{CW1} \bar{v}_t - \sigma \bar{v}_\xi - \bar{u}_\xi + \dot{X}(t) \bar{v}^S_\xi = 0, \\
& \label{CW2} \bar{u}_t - \sigma \bar{u}_\xi + \tilde{p}(\bar{v})_\xi + \dot{X}(t) \bar{u}^S_\xi = \left( \frac{\bar{u}_\xi}{\bar{v}} \right)_\xi + \Phi(\bar{v},\bar{\phi})_\xi + F_1 + F_2 + F_3.
\end{align}
\end{subequations}
Here, the correction terms $F_1$, $F_2$, and $F_3$ arise from the structural differences between the equations \eqref{eq:ra} and \eqref{eq:sh'}. They are given by
\begin{equation} \label{F1F2_def}
\begin{split}
F_1 & := \bigg( \frac{\bar{u}^S_\xi}{\bar{v}^S} \bigg)_\xi -  \left( \frac{\bar{u}_\xi}{\bar{v}} \right)_\xi, \quad 
F_2 := \tilde{p}(\bar{v})_\xi - \tilde{p}(\bar{v}^R)_\xi - \tilde{p}(\bar{v}^S)_\xi,
\end{split}
\end{equation}
and
\begin{equation} \label{F3_def}
F_3 := \Phi(\bar{v}^S, \bar{\phi}^S)_\xi - \Phi(\bar{v}, \bar{\phi})_\xi.
\end{equation}

\section{A priori estimate and proof of Theorem~\ref{Main}} \label{Sec_3}

In this section, we state an a priori estimate for $H^2$-perturbations around the approximate composite wave and outline the main ideas of its proof. We then use this estimate to prove Theorem~\ref{Main}.

\subsection{Existence of local-in-time solution}

We first present the local existence of solutions to the NSP system \eqref{NSP}.

\begin{proposition} \label{LocalE}
Let $\underline{v}$, $\underline{u}$, and $\underline{\phi}$ be smooth monotone functions such that
\begin{equation*}
\underline{v}(\xi) = v_\pm, \quad \underline{u}(\xi) = u_\pm, \quad \underline{\phi}=\phi_\pm \quad \text{for } \pm \xi \geq 1.
\end{equation*}
For any constants $M_0$, $M_1$, $\underline{\kappa}_0$, $\overline{\kappa}_0$, $\underline{\kappa}_1$, and $\overline{\kappa}_1$ with
\begin{equation*}
0 < M_0 < M_1 \quad \text{and} \quad 0 < \underline{\kappa}_1 < \underline{\kappa}_0 < \overline{\kappa}_0 < \overline{\kappa}_1,
\end{equation*}
there exists a finite time $T_0>0$ such that if the initial data $(v_0,u_0)$ satisfies
\begin{equation*}
\lVert v_0 - \underline{v} \rVert_{H^2(\mathbb{R})} + \lVert u_0-\underline{u} \rVert_{H^2(\mathbb{R})} \leq M_0 \quad \text{and} \quad \underline{\kappa}_0 \leq v_0(x) \leq \overline{\kappa}_0 \quad \text{for all } \xi \in \mathbb{R},
\end{equation*}
the Cauchy problem \eqref{NSP''}--\eqref{ic}, with \eqref{qnc}, admits a unique solution $(v,u,\phi)$ on $[0,T_0]$ satisfying
\begin{equation*}
\begin{split}
&v-\underline{v} \in C (0,T_0;H^2(\mathbb{R})), \\
&u-\underline{u} \in C (0,T_0;H^2(\mathbb{R})) \cap L^2(0,T_0;H^3(\mathbb{R})), \\
&\phi-\underline{\phi} \in C(0,T_0;H^3(\mathbb{R})),
\end{split}
\end{equation*}
with
\begin{equation*}
\lVert (v-\underline{v},u-\underline{u}) \rVert_{L^\infty(0,T_0;H^2(\mathbb{R}))} + \lVert \phi-\underline{\phi} \rVert_{L^\infty(0,T_0;H^3(\mathbb{R}))} \leq M_1
\end{equation*}
and
\begin{equation*}
\underline{\kappa}_1 \leq v(t,\xi) \leq \overline{\kappa}_1 \quad \text{for all } (t,\xi) \in [0,T_0]\times \mathbb{R}.
\end{equation*}
\end{proposition}

The proof of the local existence follows from the standard iteration method (see \cite{MN, Sol}) and is omitted here for brevity.

\subsection{Construction of shift}

We define the shift function $X:\mathbb{R}_+ \to \mathbb{R}$ as the solution to the ODE:
\begin{equation} \label{XODE}
\begin{split}
\dot{X}(t) & = - \frac{M}{\delta_S} \bigg( \int_\mathbb{R} a^X(\xi) \partial_\xi \left( \bar{u}^S - \partial_\xi \left( \ln{\bar{v}^S} \right) \right) (\xi-X) \frac{\tilde{p}(v) - \tilde{p}(\bar{v})}{\sigma} \, d\xi \\
& \qquad \qquad - \int_\mathbb{R} a^X(\xi) \partial_\xi \tilde{p}\big(\bar{v}^S(\xi-X)\big)(v-\bar{v}) \, d\xi \bigg)
\end{split}
\end{equation}
with $X(0) = 0$, where the function $a^X$ is defined in \eqref{a} and $\textstyle M = \frac{\sqrt{2}v_+}{v_m}$. This choice is used in Section~\ref{Sec_4.6.1} to control the leading-order terms associated with the shock profile. The existence of the shift $X(t)$ is ensured by the standard existence theorem for ODEs, as shown in \cite{KVW2} for the Navier--Stokes equations.

\begin{remark}
We note that the electric potential $\phi$ does not appear explicitly in the defining ODE \eqref{XODE} for $X(t)$. This reflects the fact that, in the method of $a$-contraction with shifts, $X(t)$ is essentially determined by the hyperbolic structure of the system; see the brief discussion given in Section~\ref{Sec_4.6.1}. In the NSP system, the contribution of the electric force term in the momentum equation \eqref{NSP22} to the hyperbolic part is already absorbed into the definition of the modified pressure $\tilde{p}$ through the term $\tfrac{1}{v}$ (see \eqref{trf} and \eqref{p(v)}). Consequently, $\phi$ does not appear explicitly in \eqref{XODE}.
\end{remark}

\subsection{A priori estimate}

The a priori estimate is stated as follows.

\begin{proposition} \label{Apriori}
Let $T>0$ be a positive constant. Suppose that $(v,u,\phi)$ is the solution to \eqref{NSP''} with \eqref{ic} on $[0,T]$, and that $(\bar{v},\bar{u},\bar{\phi})$ is the approximate composite wave defined in \eqref{ACW}, with the shift $X(t)$ given by \eqref{XODE}. Then there exist positive constants $\delta_0$ and $\varepsilon_1$ such that the following holds.

Assume that the solution $(v,u,\phi)$ satisfies
\begin{equation} \label{AssH2}
\sup_{0 \leq t \leq T}{ \big( \lVert (v-\bar{v},u-\bar{u})(t,\cdot) \rVert_{H^2(\mathbb{R})} + \lVert \phi-\bar{\phi}(t,\cdot) \rVert_{H^3(\mathbb{R})} \big) }  \leq \varepsilon_1
\end{equation}
and the amplitudes $\delta_R=|v_m-v_-|$ and $\delta_S=|v_+-v_m|$, defined in Lemmas~\ref{rarefaction} and \ref{Prop.1.1}, satisfy $\delta_R + \delta_S < \delta_0$. Then it holds that
\begin{equation} \label{apriori}
\begin{split}
& \lVert (v-\bar{v},u-\bar{u},\phi-\bar{\phi}) (t,\cdot) \rVert_{H^2(\mathbb{R})}^2 + \delta_S \int_0^t |\dot{X}(\tau)|^2 \, d\tau \\
& \qquad + \int_0^t  \left(  G_2 + G_3 + G^S + G^R + D + \| (v - \bar{v})_{\xi\xi} \|_{L^2(\mathbb{R})}^2 + \| (u - \bar{u})_\xi \|_{H^2(\mathbb{R})}^2\right)(\tau) \, d\tau \\
& \quad \leq C \lVert (v-\bar{v},u-\bar{u})(0,\cdot) \rVert_{H^2(\mathbb{R})}^2 + C \delta_R^{1/3}
\end{split}
\end{equation}
for all $t \in [0,T]$, with a constant $C>0$ independent of $T$, where
\begin{equation*}
\begin{array}{ll}
G_2 := \int_\mathbb{R} |(\phi - \bar{\phi})_{\xi\xi}|^2 \, d\xi, &
G_3 := \int_\mathbb{R} |(\phi - \bar{\phi})_{\xi\xi\xi}|^2 \, d\xi, \\
G^S := \int_\mathbb{R} \bar{v}^S_\xi \lvert \tilde{p}(v) - \tilde{p}(\bar{v}) \rvert^2 \, d\xi, &
G^R := \int_\mathbb{R} \bar{v}^R_\xi |v-\bar{v}|^2 \, d\xi, \\
D := \int_\mathbb{R} | \left( \tilde{p}(v) - \tilde{p}(\bar{v}) \right)_\xi |^2 \, d\xi.
\end{array}
\end{equation*}
\end{proposition}

\subsection{Proof of Theorem~\ref{Main}}

Together with Propositions~\ref{LocalE}--\ref{Apriori} and Lemma~\ref{rarefaction}, we use a continuation argument to prove the global-in-time existence of solutions near the composite wave. The time-asymptotic behavior then follows from Proposition~\ref{Apriori} and Lemma~\ref{rarefaction}. Since the proof is standard and similar to those in \cite{HK, HKK, KKSh, KVW2}, we present it in the Appendices.

\subsection{Main ideas for the proof of Proposition~\ref{Apriori}} \label{Sec_3.5}

We begin by rewriting the NSP system using the effective velocity \( h := u - (\ln v)_\xi \), as considered in \cite{HKK, KVW2}. The introduction of this variable serves two purposes. First, it enables a consistent formulation with \cite{KVW2}, allowing us to directly refer to certain intermediate results without reproducing identical computations. Second, it allows us to separate the energy estimates into a main part and the remaining parts including elliptic and higher-order estimates. As observed in the stability analysis of single shock profiles in NSP (see \cite{KKSh}; see also \cite{DLZ}), the estimate for the first derivative of the perturbation $(v-\bar{v})$ requires a delicate analysis. When expressed in terms of $h$, however, this is incorporated into the main zeroth-order estimates; see Section~\ref{ModifedD} for the treatment of the associated term.\footnote{The Poisson coupling modifies the dissipation entering the leading-order estimate associated with the shock component. We retain the coupled structure through the modified dissipation $\mathcal{D}_z$ and apply the sharp Poincar\'e-type inequality to the corresponding variable $z$ in the proof of Lemma~\ref{R1R2est}.}

Building on the formulation in terms of \( (v,h,\phi) \)-variables, we establish the energy estimates of zeroth order by employing the method of $a$-contraction with shifts. This method is used to control the main terms from the isothermal Navier--Stokes part of \eqref{NSPh}. However, to control the electric force term in \eqref{NSPh2} within the relative entropy framework, we apply the method to a modulated relative functional introduced in our previous work \cite{KKSh}. This functional, \eqref{relentropy}, is obtained by modulating the relative entropy of the Navier--Stokes equations.

After deriving the zeroth-order energy estimate in the $(v,h,\phi)$-variables, we obtain elliptic estimates for $(\phi - \bar{\phi})$, including time and space derivatives, which are required to close the estimate. These are combined with the $L^2$-estimate for $(u - \bar{u})$ to obtain the zeroth-order estimate in the $(v,u,\phi)$-variables. Finally, we complete the a priori estimate by performing higher-order estimates.

\section{\texorpdfstring{Energy estimates for the system of $(v,h)$-variables}{Energy estimates for the system of (v,h)-variables}} \label{Sec_4}

As mentioned in Section~\ref{Sec_3.5}, we introduce the effective velocity $h(t,x) := u - (\ln{v})_\xi$ so that the NSP system \eqref{NSP''} can be rewritten as
\begin{subequations} \label{NSPh}
\begin{align}
& \label{NSPh1} v_t - \sigma v_\xi - h_\xi = (\ln{v})_{\xi\xi}, \\
& \label{NSPh2} h_t - \sigma h_\xi + \tilde{p}(v)_\xi = \Phi(v,\phi)_\xi, \\
& \label{NSPh3} - \left( \frac{\phi_\xi}{v} \right)_\xi = 1 - ve^{\phi}.
\end{align}
\end{subequations}
For the shock profiles, we set
\begin{equation*}
\bar{h}^S := \bar{u}^S - (\ln{\bar{v}^S})_\xi,
\end{equation*}
and define
\begin{equation*}
\bar{h}(t,\xi) := \bar{u}^R (t,\xi + \sigma t) + \bar{h}^S (\xi-X(t)) - u_m.
\end{equation*}
Then the approximate composite wave $(\bar{v},\bar{h},\bar{\phi})$ satisfies
\begin{equation} \label{ACWh}
\begin{split}
& \bar{v}_t - \sigma \bar{v}_\xi - \bar{h}_\xi + \dot{X}(t) \left( \bar{v}^S \right)^{X}_\xi = (\ln{\bar{v}})_{\xi\xi} + F_4, \\
& \bar{h}_t - \sigma \bar{h}_\xi + \tilde{p}(\bar{v})_\xi + \dot{X}(t) \left( \bar{h}^S \right)^{X}_\xi = \Phi(\bar{v},\bar{\phi})_\xi + F_2 + F_3,
\end{split}
\end{equation}
where $F_2$ and $F_3$ are defined in \eqref{F1F2_def}--\eqref{F3_def}, and
\begin{equation} \label{F_4}
F_4 := \left( \ln{(\bar{v}^S)^{X}} - \ln{\bar{v}} \right)_{\xi\xi}.
\end{equation}

We now define the perturbations of the $(v,h,\phi)$-variables and of the physical velocity $u$ around the composite wave as
\begin{equation*}
(\tilde{v},\tilde{h},\tilde{\phi})(t,\xi) := (v,h,\phi)(t,\xi) - (\bar{v},\bar{h},\bar{\phi})(t,\xi)
\end{equation*}
and
\begin{equation} \label{pert_u}
\tilde{u}(t,\xi) := u(t,\xi) - \bar{u}(t,\xi).
\end{equation}
By \eqref{NSPh} and \eqref{ACWh}, the perturbations $\tilde{v}$ and $\tilde{h}$ then satisfy
\begin{subequations} \label{pert}
\begin{align}
& \label{pert1} \tilde{v}_t - \sigma \tilde{v}_\xi - \tilde{h}_\xi - \dot{X}(t) \bar{v}^S_\xi = \left( \ln{v} - \ln{\bar{v}} \right)_{\xi\xi} - F_4,\\
& \label{pert2} \tilde{h}_t - \sigma \tilde{h}_\xi + \left( \tilde{p}(v) - \tilde{p}(\bar{v}) \right)_\xi - \dot{X}(t) \bar{h}^S_\xi = \left( \Phi(v,\phi) - \Phi(\bar{v},\bar{\phi}) \right)_\xi - F_2 - F_3.
\end{align}
\end{subequations}
To derive an elliptic equation for $\tilde{\phi}$, we use the Poisson equations \eqref{NSPh3} and \eqref{sh3'}:
\begin{equation*}
\begin{split}
\tilde{v} & = v - \bar{v} = v - \bar{v}^R - \bar{v}^S + v_m \\
& = e^{-\phi} \bigg( \bigg( \frac{\phi_\xi}{v} \bigg)_\xi + 1 \bigg) - e^{-\bar{\phi}^S} \bigg( \bigg( \frac{\bar{\phi}^S_\xi}{\bar{v}^S} \bigg)_\xi + 1 \bigg) - \bar{v}^R + v_m.
\end{split}
\end{equation*}
Expanding this, we arrive at
\begin{equation} \label{v}
\frac{e^{-\bar{\phi}}\tilde{\phi}_{\xi\xi}}{\bar{v}} - e^{-\bar{\phi}}\tilde{\phi} = \tilde{v} - \mathcal{V},
\end{equation}
where $\mathcal{V} = \mathcal{V}^I + \mathcal{V}^L + \mathcal{V}^N$ with
\begin{equation} \label{mcV}
\begin{split}
\mathcal{V}^I & := e^{-\bar{\phi}} \bigg( \bigg( \frac{\bar{\phi}_\xi}{\bar{v}} \bigg)_\xi  + 1 \bigg) - e^{-\bar{\phi}^S} \bigg( \bigg( \frac{\bar{\phi}^S_\xi}{\bar{v}^S} \bigg)_\xi  + 1 \bigg) - \bar{v}^R + v_m, \\
\mathcal{V}^L & := e^{-\bar{\phi}} \left( - \frac{\bar{\phi}_{\xi\xi}}{\bar{v}} \tilde{\phi} - \frac{\bar{\phi}_{\xi\xi}}{\bar{v}^2}\tilde{v} + \frac{\bar{v}_\xi \bar{\phi}_\xi}{\bar{v}^2} \tilde{\phi} - \frac{\bar{\phi}_\xi}{\bar{v}^2}\tilde{v}_\xi - \frac{\bar{v}_\xi}{\bar{v}^2}\tilde{\phi}_\xi + \frac{2\bar{v}_\xi\bar{\phi}_\xi}{\bar{v}^3}\tilde{v} \right), \\
\mathcal{V}^N & := e^{-\bar{\phi}} \left( e^{-\tilde{\phi}} - 1 + \tilde{\phi} \right) \left( 1 + \frac{\phi_{\xi\xi}}{v} - \frac{v_\xi \phi_\xi}{v^2} \right) \\
& \quad +  e^{-\bar{\phi}} \left( -\frac{\tilde{v}\tilde{\phi}_\xi}{v\bar{v}} + \frac{\bar{\phi}_\xi \tilde{v}^2}{v\bar{v}^2} \right)_\xi - e^{-\bar{\phi}}\tilde{\phi} \left( \frac{\tilde{\phi}_\xi}{\bar{v}} - \frac{\bar{\phi}_\xi\tilde{v}}{\bar{v}^2}  -\frac{\tilde{v}\tilde{\phi}_\xi}{v\bar{v}} + \frac{\bar{\phi}_\xi \tilde{v}^2}{v\bar{v}^2} \right)_\xi.
\end{split}
\end{equation}
In \eqref{mcV}, $\mathcal{V}^I$ represents the wave interactions and rarefaction errors, while $\mathcal{V}^L$ and $\mathcal{V}^N$ correspond to the linear and nonlinear contributions, respectively.

The goal of this section is to establish the following lemma.

\begin{lemma} \label{lem_0th}
Under the assumptions of Proposition~\ref{Apriori}, there exists a constant $C>0$ such that
\begin{equation} \label{0th}
\begin{split}
& \lVert (\tilde{v},\tilde{h})(t,\cdot) \rVert_{L^2}^2 + \lVert \tilde{\phi}_\xi(t,\cdot) \rVert_{H^1}^2 \\
& \qquad + \int_0^t \left( \delta_S |\dot{X}|^2 + G_1 + G_2 + G_3 + G^S + G^R + D \right) d\tau \\
& \quad \leq C ( \lVert \tilde{v}_{0} \rVert_{H^2}^2 + \lVert \tilde{h}(0,\cdot) \rVert_{L^2}^2  ) \\
& \qquad + C (\sqrt{\delta_0} + \varepsilon_1) \int_0^t \int_\mathbb{R} \left( \tilde{v}_{\xi\xi}^2 + \tilde{u}_\xi^2 + \bar{v}_\xi \tilde{\phi}^2 + \tilde{\phi}_\xi^2 + \tilde{\phi}_{\xi t}^2 \right) \, d\xi d\tau + C \delta_R^{1/3}
\end{split}
\end{equation}
for all $t \in [0,T]$, where $G_2$, $G_3$, $G^S$, $G^R$, and $D$ are as defined in Proposition~\ref{Apriori}, and 
\begin{equation}
\begin{array}{lll}
G_1 := \dfrac{1}{\sqrt{\delta_S}} \int_\mathbb{R} \bar{v}^S_\xi \left\lvert \tilde{h} - \dfrac{\tilde{p}(v) - \tilde{p}(\bar{v})}{\sigma} \right\rvert^2 \, d\xi.
\end{array}
\end{equation}
\end{lemma}

\subsection{The modulated relative functional and weight function}

We will prove Lemma~\ref{lem_0th} using the method of $a$-contraction with shifts. To apply this method to our problem, we employ the modulated relative functional $\eta(\cdot|\cdot)$ for the NSP system originally introduced in $(v,u,\phi)$-variables in \cite{KKSh}. Specifically, we define the relative functional $\eta(W|\bar{W})$ between a solution $W := (v,h,\phi)$ to \eqref{NSPh} and the approximate composite wave $\bar{W}:=(\bar{v},\bar{h},\bar{\phi})$ by
\begin{equation} \label{relentropy}
\begin{split}
\eta(W|\bar{W}) & := \frac{|h-\bar{h}|^2}{2} + Q(v|\bar{v}) - \frac{(v-\bar{v})(\phi-\bar{\phi})_{\xi\xi}}{\bar{v}^2} \\
& \quad + \frac{e^{-\bar{\phi}}|(\phi-\bar{\phi})_{\xi\xi}|^2}{2\bar{v}^3} + \frac{e^{-\bar{\phi}} |(\phi-\bar{\phi})_{\xi}|^2}{2\bar{v}^2}.
\end{split}
\end{equation}
This functional can be viewed as a modulation of the relative entropy 
\begin{equation} \label{NSeta}
\tilde{\eta}(W|\bar{W}) := \frac{|h-\bar{h}|^2}{2} + Q(v|\bar{v})
\end{equation}
associated with the isothermal Navier-Stokes equations in $(v,h)$-variables:
\begin{equation} \label{NSh}
\begin{split}
& v_t - \sigma v_\xi - h_\xi = (\ln{v})_{\xi\xi}, \\
& h_t - \sigma h_\xi + \tilde{p}(v)_\xi = 0.
\end{split}
\end{equation}

The following lemma establishes a key property of the modulated relative functional $\eta(W|\bar{W})$.

\begin{lemma} \label{entsim}
Under the assumptions of Proposition \ref{Apriori}, there exist constants $c>0$ and $C>0$ such that 
\begin{equation*}
\eta(W|\bar{W}) \geq c \left( \lvert h - \bar{h} \rvert^2 + \lvert v - \bar{v} \rvert^2 + \lvert ( \phi - \bar{\phi} )_{\xi\xi} \rvert^2 + \lvert ( \phi - \bar{\phi} )_\xi \rvert^2 \right)
\end{equation*}
and
\begin{equation*}
\eta(W|\bar{W}) \leq C \left( \lvert h - \bar{h} \rvert^2 + \lvert v - \bar{v} \rvert^2 + \lvert ( \phi - \bar{\phi} )_{\xi\xi} \rvert^2 + \lvert ( \phi - \bar{\phi} )_\xi \rvert^2 \right)
\end{equation*}
for all $t \in [0,T]$ and $\xi \in \mathbb{R}$.
\end{lemma}

\begin{proof}
Note that
\begin{equation*}
Q(v|\bar{v}) \geq \frac{1}{\bar{v}^2}|v-\bar{v}|^2 - \frac{2}{3\bar{v}^3}(v-\bar{v})^3,
\end{equation*}
and, by \eqref{ra3} and \eqref{sh3'},
\begin{equation*}
\begin{split}
\bigg\lvert \frac{e^{-\bar{\phi}}}{\bar{v}} \bigg\rvert & = \bigg\lvert \frac{\bar{v}^R e^{-\bar{\phi}^S}}{v_m\bar{v}} \bigg\rvert = \bigg\lvert \frac{\bar{v}^R}{v_m\bar{v}} \bigg( \bar{v}^S - e^{-(\bar{\phi}^S)^X} \bigg( \frac{ \bar{\phi}^S_\xi}{\bar{v}^S} \bigg)_\xi \bigg) \bigg\rvert \\
& \geq \frac{v_-}{v_m v_+} \left( v_m - C \delta_S^2 \right) \geq c
\end{split}
\end{equation*}
for some constant $c>0$, provided that $\delta_R + \delta_S $ is sufficiently small. These bounds are the key ingredients in the proof, which we omit for brevity; see the proof of Lemma~3.2 in \cite{KKSh}.
\end{proof}

Next, we define the weight function $a^X$ as
\begin{equation} \label{a}
a^X(\xi) := 1 + \frac{\tilde{p}(v_m) - \tilde{p}(\bar{v}^S(\xi -X(t)))}{\sqrt{\delta_S}},
\end{equation}
where $X(t)$ is the shift function defined by \eqref{XODE}. The weight $a^X$ and its derivative satisfy the bounds:
\begin{equation*}
1 < a^X(\xi) \leq 1 + \frac{2\sqrt{\delta_S}}{v_m^2} < \infty, \quad a^X_\xi(\xi) = \frac{2\bar{v}^S_\xi}{\sqrt{\delta_S}(\bar{v}^S)^2} > 0.
\end{equation*}

\subsection{Identity for the weighted relative functional}

In this and the subsequent subsections, we estimate the following quantity:
\begin{equation} \label{relquan}
\int_\mathbb{R} a^X(\xi) \eta(W(t,\xi)|\bar{W}(t,\xi)) \, d\xi,
\end{equation}
where the weight function $a^X(\xi)$ is as in \eqref{a}. We begin by writing the identity for \eqref{relquan}.

\begin{lemma} \label{Id}
Let $a^X$ be the weight function defined by \eqref{a}. Let $W=(v,h,\phi)$ be a solution to \eqref{NSPh} with \eqref{ic}, and $\bar{W}=(\bar{v},\bar{h},\bar{\phi})$ be the approximate composite wave defined in \eqref{ACW}, with the shift $X(t)$ given in \eqref{XODE}. Then, the following identity holds:
\begin{equation} \label{Idd}
\begin{split}
& \frac{d}{dt} \int_\mathbb{R} a^{X} \eta(W|\bar{W}) \, d\xi  = \dot{X}(t) Y(W) + \mathcal{J}^{\textup{bad}} (W) - \mathcal{J}^{\textup{good}}(W) + \tilde{\mathcal{P}}(W) + \sum_{j=1}^5 \mathcal{P}_{j}(W),
\end{split}
\end{equation}
where the terms $Y(W)$, $\mathcal{J}^{\textup{bad}}(W)$, and $\mathcal{J}^{\textup{good}}(W)$ are defined by
\begin{equation*}
\begin{split}
Y(W) & := - \int_\mathbb{R} a^{X}_\xi \eta (W|\bar{W}) \, d\xi +  \int_\mathbb{R} a^{X} \bigg( \bar{h}^S_\xi \tilde{h} - \tilde{p}'(\bar{v}) \bar{v}^S_\xi \tilde{v} - \frac{\bar{v}^S_\xi \tilde{\phi}_{\xi\xi}}{\bar{v}^2} \bigg) \, d\xi, \\
\mathcal{J}^{\textup{bad}}(W) & := \int_\mathbb{R} a^{X}_\xi \left( \tilde{p}(v) - \tilde{p}(\bar{v}) \right) \tilde{h} \, d\xi + \sigma \int_\mathbb{R} a^{X} \bar{v}^S_\xi \tilde{p}(v|\bar{v}) \, d\xi \\
& \quad - \int_\mathbb{R} a^{X}_\xi \frac{\tilde{p}(v) - \tilde{p}(\bar{v})}{\tilde{p}(v)} \left( \tilde{p}(v) - \tilde{p}(\bar{v}) \right)_\xi \, d\xi + \int_\mathbb{R} a^{X}_\xi \left( \tilde{p}(v) - \tilde{p}(\bar{v}) \right)^2 \frac{\tilde{p}(\bar{v})_\xi}{\tilde{p}(v)\tilde{p}(\bar{v})} \, d\xi \\
& \quad - \int_\mathbb{R} a^{X} \left( \tilde{p}(v) - \tilde{p}(\bar{v}) \right)_\xi \frac{\tilde{p}(\bar{v}) - \tilde{p}(v)}{\tilde{p}(v)\tilde{p}(\bar{v})} \tilde{p}(\bar{v})_\xi \, d\xi + \int_\mathbb{R} a^{X} \left( \tilde{p}(v) - \tilde{p}(\bar{v}) \right) F_4 \, d\xi \\
& \quad - \int_\mathbb{R} a^{X} \tilde{h} F_2 \, d\xi,
\end{split}
\end{equation*}
and
\begin{equation*}
\begin{split}
\mathcal{J}^{\textup{good}}(W) & := \frac{\sigma}{2} \int_\mathbb{R} a^{X}_\xi \tilde{h}^2 \, d \xi + \sigma \int_\mathbb{R} a^{X}_\xi Q(v|\bar{v}) \, d\xi  + \int_\mathbb{R} a^{X} \bar{u}^R_\xi \tilde{p}(v|\bar{v}) \, d\xi \\
& \quad + \int_\mathbb{R} \frac{a^{X}}{\tilde{p}(v)} \left( \tilde{p}(v) - \tilde{p}(\bar{v}) \right)_\xi^2 \, d\xi,
\end{split}
\end{equation*}
respectively, and the terms $\tilde{\mathcal{P}}(W)$ and $\mathcal{P}_{j}(W)$ are as follows:
\begin{equation*}
\begin{split}
\tilde{\mathcal{P}}(W) & := - \int_\mathbb{R} a^{X} \left( \ln{v} - \ln{\bar{v}} \right)_{\xi\xi} \frac{\tilde{\phi}_{\xi\xi}}{\bar{v}^2} \, d\xi, \\
\mathcal{P}_1(W) & := - \int_\mathbb{R} a^{X} \tilde{h} F_3 \, d\xi, \\
\mathcal{P}_2(W) & := \int_\mathbb{R} a^{X}_\xi \frac{\tilde{h}\tilde{\phi}_{\xi\xi}}{\bar{v}^2} \, d\xi + \int_\mathbb{R} \left( a^{X}_\xi \tilde{h} + a^{X} \tilde{h}_\xi \right)  \bigg( \frac{1}{v^2} - \frac{1}{\bar{v}^2} \bigg) \left( \tilde{\phi}_{\xi\xi} + \bar{\phi}_{\xi\xi} \right)  \, d\xi \\
& \quad - \frac{1}{2} \int_\mathbb{R}  \left( a^{X}_\xi \tilde{h} + a^{X} \tilde{h}_\xi \right) \bigg[ \frac{\tilde{\phi}_\xi^2}{v^2} + \frac{2\bar{\phi}_\xi \tilde{\phi}_\xi}{v^2} + \bigg( \frac{1}{v^2} - \frac{1}{\bar{v}^2} \bigg) \bar{\phi}_\xi^2 \bigg] \, d\xi \\
& \quad - \int_\mathbb{R} \left( a^{X}_\xi \tilde{h} + a^{X} \tilde{h}_\xi \right) \bigg[ \frac{\tilde{v}_\xi \tilde{\phi}_\xi}{v^3} + \frac{\bar{v}_\xi \tilde{\phi}_\xi}{v^3} + \frac{\bar{\phi}_\xi \tilde{v}_\xi}{v^3} + \bigg( \frac{1}{v^3} - \frac{1}{\bar{v}^3} \bigg) \bar{v}_\xi \bar{\phi}_\xi \bigg] \, d\xi, \\
\mathcal{P}_3(W) & := - \sigma \int_\mathbb{R} a^{X} \frac{ \tilde{v}_\xi \tilde{\phi}_{\xi\xi}}{\bar{v}^2} \, d\xi, \\
\mathcal{P}_4(W) & := \int_\mathbb{R} a^{X}_\xi \frac{\tilde{v}\tilde{\phi}_{\xi t}}{\bar{v}^2} \, d\xi + \int_\mathbb{R} a^{X} \bigg( \frac{2 (\sigma \bar{v}^R_\xi + \bar{u}^R_\xi) \tilde{v} \tilde{\phi}_{\xi\xi}}{\bar{v}^3} - \frac{2 \bar{v}_\xi \tilde{v} \tilde{\phi}_{\xi t}}{\bar{v}^3}  \bigg) \, d\xi \\
& \quad - \dot{X}(t) \int_\mathbb{R} a^{X} \frac{2 \bar{v}^S_\xi \tilde{v} \tilde{\phi}_{\xi\xi}}{\bar{v}^3} \, d\xi + \int_\mathbb{R} a^{X} \frac{\tilde{\phi}_{\xi\xi}}{\bar{v}^2} F_4 \, d\xi,
\end{split}
\end{equation*}
and
\begin{equation*}
\begin{split}
\mathcal{P}_5(W) & :=  \int_\mathbb{R} a^{X} \bigg( \frac{e^{-\bar{\phi}}}{\bar{v}^3} \bigg)_t \frac{\tilde{\phi}_{\xi\xi}^2}{2} \, d\xi + \int_\mathbb{R} a^{X} \bigg( \frac{e^{-\bar{\phi}}}{\bar{v}^2} \bigg)_t \frac{\tilde{\phi}_\xi^2}{2} \, d\xi - \int_\mathbb{R} a^{X}_\xi \frac{e^{-\bar{\phi}} \tilde{\phi}_{\xi\xi} \tilde{\phi}_{\xi t}}{\bar{v}^3} \, d\xi \\
& \quad - \int_\mathbb{R} a^{X} \bigg( \frac{e^{-\bar{\phi}}}{\bar{v}^3} \bigg)_\xi \tilde{\phi}_{\xi\xi} \tilde{\phi}_{\xi t} \, d\xi  + \int_\mathbb{R} a^{X} \bigg( \frac{e^{-\bar{\phi}}}{\bar{v}} \bigg)_\xi \frac{ \tilde{\phi}_{\xi\xi} \tilde{\phi}_{\xi t}}{\bar{v}^2} \, d\xi \\
& \quad + \int_\mathbb{R} a^{X} \frac{e^{-\bar{\phi}}\bar{\phi}_\xi \tilde{\phi} \tilde{\phi}_{\xi t}}{\bar{v}^2} \, d\xi + \int_\mathbb{R} a^{X} \frac{\tilde{\phi}_{\xi t}}{\bar{v}^2} \mathcal{V}_\xi \, d\xi.
\end{split}
\end{equation*}
\end{lemma}

\begin{proof}[Proof of Lemma~\ref{Id}]

First, we compute the evolution of the relative entropy for the first two equations of \eqref{NSPh}, i.e., for the NS system \eqref{NSh} with forcing term $\Phi(v,\phi)_\xi$. For that, we apply the relative entropy method to the system of the two equations:
\begin{equation}
\partial_t U + \partial_\xi A(U) = \begin{pmatrix}
(\ln{v})_{\xi\xi} \\ \partial_{\xi} \Phi(v,\phi)
\end{pmatrix}, \quad U := \begin{pmatrix}
v \\ h
\end{pmatrix}, \quad A(U) := \begin{pmatrix}
- \sigma v - h \\ -\sigma h + \tilde{p}(v)
\end{pmatrix}.
\end{equation}
Thus, we obtain the following identity for the relative entropy $\tilde{\eta}(\cdot|\cdot)$ of the above quantity $U$ and $\textstyle \bar{U} := \begin{pmatrix} \bar{v} \\ \bar{h} \end{pmatrix}$:
\begin{equation*}
\begin{split}
\frac{d}{dt} \int_\mathbb{R} a^{X} \tilde{\eta}(U|\bar{U}) \, d\xi & =  - \dot{X}(t) \int_\mathbb{R} a_\xi^{X} \tilde{\eta}(U|\bar{U}) \, d\xi \\
& \quad + \int_\mathbb{R} a^{X} \left[ \left( \nabla \tilde{\eta}(U) - \nabla \tilde{\eta}(\bar{U}) \right)\partial_t U - \nabla^2 \tilde{\eta}(\bar{U}) (U-\bar{U}) \partial_t \bar{U} \right] \, d\xi.
\end{split}
\end{equation*}
This yields, by the definition \eqref{NSeta} of $\tilde{\eta}$, 
\begin{equation} \label{NSId}
\begin{split}
& \frac{d}{dt} \int_\mathbb{R} a^X \Big( \frac{\tilde{h}^2}{2} + Q(v|\bar{v}) \Big) \, d\xi \\
& \quad = \dot{X}(t) \left( - \int_\mathbb{R} a^{X}_\xi \Big( \frac{\tilde{h}^2}{2} + Q(v|\bar{v}) \Big)  \, d\xi + \int_\mathbb{R} a^{X} \left( \bar{h}^S_\xi \tilde{h} - \tilde{p}'(\bar{v}) \bar{v}^S_\xi \tilde{v} \right) \, d\xi  \right) \\
& \qquad + \int_\mathbb{R} a^{X}_\xi \left( \tilde{p}(v) - \tilde{p}(\bar{v}) \right) \tilde{h} \, d\xi + \sigma \int_\mathbb{R} a^{X} \bar{v}^S_\xi \tilde{p}(v|\bar{v}) \, d\xi \\
& \qquad - \int_\mathbb{R} a^{X}_\xi \frac{\tilde{p}(v) - \tilde{p}(\bar{v})}{\tilde{p}(v)} \left( \tilde{p}(v) - \tilde{p}(\bar{v}) \right)_\xi \, d\xi  + \int_\mathbb{R} a^{X}_\xi \left( \tilde{p}(v) - \tilde{p}(\bar{v}) \right)^2 \frac{\tilde{p}(\bar{v})_\xi}{\tilde{p}(v)\tilde{p}(\bar{v})} \, d\xi \\
& \qquad - \int_\mathbb{R} a^{X} \left( \tilde{p}(v) - \tilde{p}(\bar{v}) \right)_\xi \frac{\tilde{p}(\bar{v}) - \tilde{p}(v)}{\tilde{p}(v)\tilde{p}(\bar{v})} \tilde{p}(\bar{v})_\xi \, d\xi - \frac{\sigma}{2} \int_\mathbb{R} a^{X}_\xi \tilde{h}^2 \, d \xi  \\
& \qquad - \sigma \int_\mathbb{R} a^{X}_\xi Q(v|\bar{v}) \, d\xi  - \int_\mathbb{R} a^{X} \bar{u}^R_\xi \tilde{p}(v|\bar{v}) \, d\xi - \int_\mathbb{R} \frac{a^X}{\tilde{p}(v)} \left( \tilde{p}(v) - \tilde{p}(\bar{v}) \right)_\xi^2 \, d\xi \\
& \qquad  + \int_\mathbb{R} a^{X} \left( \tilde{p}(v) - \tilde{p}(\bar{v}) \right) F_4 \, d\xi - \int_\mathbb{R} a^{X} \tilde{h} F_2 \, d\xi + \mathcal{P}_1 \\
& \qquad + \int_\mathbb{R} a^{X} \tilde{h} \left( \Phi(v,\phi) - \Phi(\bar{v},\bar{\phi}) \right)_\xi \, d\xi,
\end{split}
\end{equation}
where $\mathcal{P}_1$ is as in this Lemma. A detailed derivation of \eqref{NSId} can be found in the proof of Lemma~4.3 in \cite{KVW2}. It therefore remains to compute the last term.

Recalling the definition \eqref{Phi} of $\Phi$, one can obtain by integration by parts
\begin{equation} \label{Phi1}
\begin{split}
\int_\mathbb{R} a^{X} \tilde{h} \left( \Phi(v,\phi) - \Phi(\bar{v},\bar{\phi}) \right)_\xi \, d\xi & = - \frac{1}{2} \int_\mathbb{R} \left( a^{X} \tilde{h} \right)_\xi \bigg[ \left( \frac{\phi_\xi}{v} \right)^2 - \left( \frac{\bar{\phi}_\xi}{\bar{v}} \right)^2 \bigg] \, d\xi \\
& \quad + \int_\mathbb{R} \left( a^{X} \tilde{h} \right)_\xi \bigg[ \frac{1}{v} \left( \frac{\phi_\xi}{v} \right)_\xi - \frac{1}{\bar{v}} \left( \frac{\bar{\phi}_\xi}{\bar{v}} \right)_\xi \bigg] \, d\xi.
\end{split}
\end{equation}
The first and second terms on the right-hand side of \eqref{Phi1} are expanded as
\begin{equation*}
\begin{split}
& - \frac{1}{2} \int_\mathbb{R} \left( a^{X} \tilde{h} \right)_\xi \bigg[ \bigg( \frac{\phi_\xi}{v} \bigg)^2 - \bigg( \frac{\bar{\phi}_\xi}{\bar{v}} \bigg)^2 \bigg] \, d\xi \\
& \quad = - \frac{1}{2} \int_\mathbb{R}  \left( a^{X}_\xi \tilde{h} + a^{X} \tilde{h}_\xi \right) \bigg[ \frac{\tilde{\phi}_\xi^2}{v^2} + \frac{2\bar{\phi}_\xi \tilde{\phi}_\xi}{v^2} + \bigg( \frac{1}{v^2} - \frac{1}{\bar{v}^2} \bigg) \bar{\phi}_\xi^2 \bigg] \, d\xi
\end{split}
\end{equation*}
and
\begin{equation*}
\begin{split}
& \int_\mathbb{R} \left( a^{X} \tilde{h} \right)_\xi \bigg[ \frac{1}{v} \bigg( \frac{\phi_\xi}{v} \bigg)_\xi - \frac{1}{\bar{v}} \bigg( \frac{\bar{\phi}_\xi}{\bar{v}} \bigg)_\xi \bigg] \, d\xi \\
& \quad = \int_\mathbb{R} \left( a^{X}_\xi \tilde{h} + a^{X} \tilde{h}_\xi \right) \bigg[ \frac{\tilde{\phi}_{\xi\xi}}{\bar{v}^2} + \bigg( \frac{1}{v^2} - \frac{1}{\bar{v}^2} \bigg) \left( \tilde{\phi}_{\xi\xi} + \bar{\phi}_{\xi\xi} \right) \bigg] \, d\xi \\
& \qquad - \int_\mathbb{R} \left( a^{X}_\xi \tilde{h} + a^{X} \tilde{h}_\xi \right) \bigg[ \frac{\tilde{v}_\xi \tilde{\phi}_\xi}{v^3} + \frac{\bar{v}_\xi \tilde{\phi}_\xi}{v^3} + \frac{\bar{\phi}_\xi \tilde{v}_\xi}{v^3} + \bigg( \frac{1}{v^3} - \frac{1}{\bar{v}^3} \bigg) \bar{v}_\xi \bar{\phi}_\xi \bigg] \, d\xi,
\end{split}
\end{equation*}
respectively. Thus, we can rewrite \eqref{Phi1} as
\begin{equation} \label{Phi2}
\int_\mathbb{R} a^{X} \tilde{h} \left( \Phi(v,\phi) - \Phi(\bar{v},\bar{\phi}) \right)_\xi \, d\xi = \int_\mathbb{R} a^{X} \frac{ \tilde{h}_\xi \tilde{\phi}_{\xi\xi}}{\bar{v}^2} \, d\xi + \mathcal{P}_2.
\end{equation}
For the first term on the right-hand side of \eqref{Phi2}, we use \eqref{pert1} and \eqref{ra1} to obtain
\begin{equation} \label{Phi3}
\begin{split}
\int_\mathbb{R} a^{X} \frac{\tilde{h}_\xi \tilde{\phi}_{\xi\xi}}{\bar{v}^2} \, d\xi & = \int_\mathbb{R} a^{X} \frac{\tilde{v}_t \tilde{\phi}_{\xi\xi}}{\bar{v}^2} \, d\xi - \sigma \int_\mathbb{R} a^{X} \frac{\tilde{v}_\xi \tilde{\phi}_{\xi\xi}}{\bar{v}^2} \, d\xi - \int_\mathbb{R} a^{X} \left( \ln{v} - \ln{\bar{v}} \right)_{\xi\xi} \frac{\tilde{\phi}_{\xi\xi}}{\bar{v}^2} \, d\xi \\
& \quad + \int_\mathbb{R} a^{X} \frac{\tilde{\phi}_{\xi\xi}}{\bar{v}^2} F_4 \, d\xi - \dot{X}(t) \int_\mathbb{R} a^{X} \frac{ \bar{v}^S_\xi \tilde{\phi}_{\xi\xi}}{\bar{v}^2} \, d\xi \\
& = \frac{d}{dt} \int_\mathbb{R} a^{X} \frac{\tilde{v} \tilde{\phi}_{\xi\xi}}{\bar{v}^2} \, d\xi + \int_\mathbb{R} a^{X} \frac{\tilde{v}_\xi \tilde{\phi}_{\xi t}}{\bar{v}^2} \, d\xi + \tilde{\mathcal{P}} + \mathcal{P}_3 + \mathcal{P}_4 \\
& \quad + \dot{X}(t) \left( \int_\mathbb{R} a^{X}_\xi \frac{\tilde{v}\tilde{\phi}_{\xi\xi}}{\bar{v}^2} \, d\xi - \int_\mathbb{R} a^{X} \frac{ \bar{v}^S_\xi \tilde{\phi}_{\xi\xi}}{\bar{v}^2} \, d\xi \right).
\end{split}
\end{equation}
The last term on the right-hand side of the last equality in \eqref{Phi3} is part of $\dot{X}(t)Y(W)$. Next, in order to rewrite the second term, we use the identity 
\begin{equation*}
\tilde{v}_\xi = \bigg( \frac{e^{-\bar{\phi}}\tilde{\phi}_{\xi\xi}}{\bar{v}} - e^{-\bar{\phi}}\tilde{\phi} \bigg)_\xi + \mathcal{V}_\xi,
\end{equation*}
which is obtained by differentiating \eqref{v}, where $\mathcal{V}$ is defined by \eqref{mcV}. Using this, and applying integration by parts, we have
\begin{equation*}
\begin{split}
\int_\mathbb{R} a^{X} \frac{\tilde{v}_\xi \tilde{\phi}_{\xi t}}{ \bar{v}^2} \, d\xi & = \int_\mathbb{R} a^{X} \frac{e^{-\bar{\phi}}\tilde{\phi}_{\xi\xi\xi}\tilde{\phi}_{\xi t}}{\bar{v}^3} \, d\xi + \int_\mathbb{R} a^{X} \bigg( \frac{e^{-\bar{\phi}}}{\bar{v}} \bigg)_\xi \frac{ \tilde{\phi}_{\xi\xi} \tilde{\phi}_{\xi t}}{\bar{v}^2} \, d\xi \\
& \quad - \int_\mathbb{R} a^{X} \frac{e^{-\bar{\phi}} \tilde{\phi}_\xi \tilde{\phi}_{\xi t}}{\bar{v}^2} \, d\xi + \int_\mathbb{R} a^{X} \frac{e^{-\bar{\phi}}\bar{\phi}_\xi \tilde{\phi} \tilde{\phi}_{\xi t}}{\bar{v}^2} \, d\xi + \int_\mathbb{R} a^{X} \frac{\tilde{\phi}_{\xi t}}{\bar{v}^2} \mathcal{V}_\xi \, d\xi \\
& = - \frac{d}{dt} \bigg( \int_\mathbb{R} a^{X} \frac{e^{-\bar{\phi}}\tilde{\phi}_{\xi\xi}^2}{2\bar{v}^3} \, d\xi + \int_\mathbb{R} a^{X} \frac{e^{-\bar{\phi}}\tilde{\phi}_\xi^2}{2\bar{v}^2} \, d\xi \bigg)  - \dot{X}(t) \int_\mathbb{R} a^{X}_\xi \bigg( \frac{e^{-\bar{\phi}}\tilde{\phi}_{\xi\xi}^2}{2\bar{v}^3} + \frac{e^{-\bar{\phi}}\tilde{\phi}_\xi^2}{2\bar{v}^2} \bigg) \, d\xi + \mathcal{P}_5.
\end{split}
\end{equation*}
The term
\begin{equation*}
- \dot{X}(t) \int_\mathbb{R} a^{X}_\xi \bigg( \frac{e^{-\bar{\phi}}\tilde{\phi}_{\xi\xi}^2}{2\bar{v}^3} + \frac{e^{-\bar{\phi}}\tilde{\phi}_\xi^2}{2\bar{v}^2} \bigg) \, d\xi
\end{equation*}
is absorbed in $\dot X(t)Y(W)$. Collecting all the terms obtained above, we have the desired identity for the relative functional.
\end{proof}

\subsection{\texorpdfstring{Maximization of $\mathcal{J}^{\textup{bad}}$ with respect to $h-\bar{h}$}{Maximization of the bad term J with respect to (h-h̄)}}

We rewrite $\mathcal{J}^{\textup{bad}}$ into the maximized representation to control the primary bad term
\begin{equation*}
\int_\mathbb{R} a^X_\xi \left( \tilde{p}(v) - \tilde{p}(\bar{v}) \right) \tilde{h} \, d\xi.
\end{equation*}
The next lemma states an identity resulting from the decomposition of $\mathcal{J}^{\textup{bad}}(W)$ and $Y(W)$ in \eqref{Id}, following the approach in Sections~4.4--4.5 of \cite{KVW2}. The proof is omitted to avoid repeating the computations therein.

\begin{lemma} [\cite{KVW2}] \label{Max}
The following identity holds:
\begin{equation} \label{max}
\begin{split}
\frac{d}{dt} \int_\mathbb{R} a^{X} \eta(W|\bar{W}) \, d\xi & =  \dot{X}(t) \sum_{j=1}^9 Y_j(W)  + \sum_{j=1}^5 \mathcal{B}_j(W) + \sum_{j=1}^2 S_j(W) - \sum_{j=1}^2 \mathcal{G}_j(W) \\
& \quad - \mathcal{G}^R(W) - \mathcal{D}(W) + \tilde{\mathcal{P}}(W) + \sum_{j=1}^5 \mathcal{P}_j(W),
\end{split}
\end{equation}
where the terms $\tilde{\mathcal{P}}$ and $\mathcal{P}_j$ are as in Lemma~\ref{Id}, and $Y_j,\mathcal{B}_j,S_j,\mathcal{G}_j,\mathcal{G}^R, \mathcal{D}$ are defined by
\begin{equation*}
\begin{split}
Y_1(W) & := \frac{1}{\sigma} \int_\mathbb{R} a^{X} \bar{h}^S_\xi \left( \tilde{p}(v) - \tilde{p}(\bar{v}) \right) \, d\xi, \quad Y_2(W) := - \int_\mathbb{R} a^{X} \tilde{p}'(\bar{v}^S) \bar{v}^S_\xi \tilde{v} \, d\xi, \\
Y_3(W) & := \int_\mathbb{R} a^{X} \bar{h}^S_\xi \left( \tilde{h} - \frac{\tilde{p}(v) - \tilde{p}(\bar{v})}{\sigma} \right) \, d\xi, \quad Y_4(W) := - \int_\mathbb{R} a^{X} \left( \tilde{p}'(\bar{v}) - \tilde{p}'(\bar{v}^S) \right) \bar{v}^S_\xi \tilde{v} \, d\xi, \\
Y_5(W) & := - \frac{1}{2} \int_\mathbb{R} a^{X}_\xi \left( \tilde{h} - \frac{\tilde{p}(v)- \tilde{p}(\bar{v})}{\sigma} \right) \left( \tilde{h} + \frac{\tilde{p}(v) - \tilde{p}(\bar{v})}{\sigma} \right) \, d\xi, \\
Y_6(W) & := - \int_\mathbb{R} a^{X}_\xi Q(v|\bar{v}) \, d\xi - \frac{1}{2\sigma^2} \int_\mathbb{R} a^{X}_\xi \left( \tilde{p}(v) - \tilde{p}(\bar{v}) \right)^2 \, d\xi, \\
Y_7(W) & := \int_\mathbb{R} a^{X}_\xi \frac{\tilde{v} \tilde{\phi}_{\xi\xi}}{\bar{v}^2} \, d\xi, \quad Y_8 (W) := - \frac{1}{2} \int_\mathbb{R} a^{X}_\xi \bigg( \frac{e^{-\bar{\phi}}\tilde{\phi}_{\xi\xi}^2}{\bar{v}^3} + \frac{e^{-\bar{\phi}}\tilde{\phi}_{\xi}^2}{\bar{v}^2} \bigg) \, d\xi,  \\
Y_9(W) & := - \int_\mathbb{R} a^{X} \frac{\bar{v}^S_\xi \tilde{\phi}_{\xi\xi}}{\bar{v}^2} \, d\xi,
\end{split}
\end{equation*}
\begin{equation*}
\begin{split}
\mathcal{B}_1(W) &:= \frac{1}{2\sigma} \int_\mathbb{R} a^{X}_\xi \lvert \tilde{p}(v) - \tilde{p}(\bar{v}) \rvert^2 \, d\xi, \quad \mathcal{B}_2(W) := \sigma \int_\mathbb{R} a^{X} \bar{v}^S_\xi \tilde{p}(v|\bar{v}) \, d\xi, \\
\mathcal{B}_3(W) &:= - \int_\mathbb{R} a^{X}_\xi \frac{\tilde{p}(v) - \tilde{p}(\bar{v})}{\tilde{p}(v)} \partial_\xi \left( \tilde{p}(v) - \tilde{p}(\bar{v}) \right) \, d\xi, \\
\mathcal{B}_4(W) &:= \int_\mathbb{R} a^{X}_\xi \left( \tilde{p}(v) - \tilde{p}(\bar{v}) \right)^2 \frac{\partial_\xi \tilde{p}(\bar{v})}{ \tilde{p}(v) \tilde{p}(\bar{v})} \, d\xi, \\
\mathcal{B}_5(W) &:= - \int_\mathbb{R} a^{X} \partial_\xi \left( \tilde{p}(v) - \tilde{p}(\bar{v}) \right) \frac{\tilde{p}(\bar{v}) - \tilde{p}(v)}{\tilde{p}(v) \tilde{p}(\bar{v})} \partial_\xi \tilde{p}(\bar{v}) \, d\xi, \\
S_1(W) &:= \int_\mathbb{R} a^{X} \left( \tilde{p}(v) - \tilde{p}(\bar{v}) \right) \left( \ln{\bar{v}^S} - \ln{\bar{v}} \right)_{\xi\xi} \, d\xi, \\
S_2(W) &:= - \int_\mathbb{R} a^{X} \tilde{h} \left( \tilde{p}(\bar{v}) - \tilde{p}(\bar{v}^R) - \tilde{p}(\bar{v}^S) \right)_\xi \, d\xi,
\end{split}
\end{equation*}
and
\begin{equation*}
\begin{array}{l l}
\mathcal{G}_1(W) := \frac{\sigma}{2} \int_\mathbb{R} a^{X}_\xi \left\lvert \tilde{h} - \frac{\tilde{p}(v) - \tilde{p}(\bar{v})}{\sigma} \right\rvert^2 \, d\xi, & \mathcal{G}_2(W) := \sigma \int_\mathbb{R} a^{X}_\xi Q(v|\bar{v}) \, d\xi, \\
\mathcal{G}^R(W) := \int_\mathbb{R} a^{X} \bar{u}^R_\xi \tilde{p}(v|\bar{v}) \, d\xi, & \mathcal{D}(W) := \int_\mathbb{R} a^{X} \frac{\left( \tilde{p}(v) -\tilde{p}(\bar{v}) \right)_\xi^2}{\tilde{p}(v)} \, d\xi.
\end{array}
\end{equation*}
\end{lemma}

\subsection{Modification of the dissipation by the Poisson coupling} \label{ModifedD}

To control the term $\tilde{\mathcal{P}}$ while retaining sufficient dissipation for the leading-order estimates, we introduce a modified dissipation term using the Poisson equation. For this purpose, we define the quantity
\begin{equation} \label{zdef}
z := \tilde{p}(v) - \tilde{p}(\bar{v}) + \frac{2\tilde{\phi}_{\xi\xi}}{\bar{v}^2}.
\end{equation}
Rewriting the Poisson equation \eqref{NSP3''}, we have
\begin{equation*}
\frac{1}{v} + \frac{\phi_{\xi\xi}}{v^2} = e^\phi + \frac{v_\xi\phi_\xi}{v^3}.
\end{equation*}
Thus, by \eqref{zdef} and \eqref{p(v)}, we obtain the identity for $z$:
\begin{equation} \label{zPoisson}
z = 2\left(e^\phi - e^{\bar{\phi}}\right) + 2\mathcal{P}_{\mathrm{prof}} + 2\mathcal{N},
\end{equation}
where
\begin{equation*}
\mathcal{P}_{\mathrm{prof}} := - \frac{1}{\bar{v}} - \frac{1}{\bar{v}} \bigg( \frac{\bar{\phi}_\xi}{\bar{v}} \bigg)_\xi + e^{\bar{\phi}}, \qquad \mathcal{N} := \bigg( \frac{1}{\bar{v}^2} - \frac{1}{v^2} \bigg)\phi_{\xi\xi} + \frac{v_\xi\phi_\xi}{v^3} - \frac{\bar{v}_\xi\bar{\phi}_\xi}{\bar{v}^3}.
\end{equation*}

\begin{lemma}
Under the assumptions of Proposition~\ref{Apriori}, the following inequality holds:
\begin{equation} \label{quad}
\tilde{\mathcal{P}}(W) - \mathcal{D}(W) \leq - \frac34 \mathcal{D}_z (W) - \frac14 \mathcal{D}(W) - \mathcal{G}_3(W) - \mathcal{G}_4(W) + \sum_{j=1}^6 \tilde{P}_j(W),
\end{equation}
where
\begin{equation} \label{Dzdef}
\mathcal{D}_z (W) := \int_\mathbb{R} a^X \frac{z_\xi^2}{\tilde{p}(v)} \, d\xi, \quad \mathcal{G}_3(W) := 2 \int_\mathbb{R} a^X \frac{ve^\phi \tilde{\phi}_{\xi\xi}^2}{\bar{v}^2} \, d\xi, \quad \mathcal{G}_4(W) := \frac14 \int_\mathbb{R} a^X \frac{v\tilde{\phi}_{\xi\xi\xi}^2}{\bar{v}^4} \, d\xi,
\end{equation}
and
\begin{equation*}
\begin{aligned}
\tilde{P}_1(W) &:= \int_\mathbb{R} a^X \frac{\bar{v}_\xi \tilde{v}}{\bar{v}^2} \bigg( \frac{\tilde{\phi}_{\xi\xi}}{\bar{v}^2} \bigg)_\xi \, d\xi, & \tilde{P}_2(W) &:= \int_\mathbb{R} a^X_\xi \bigg( \frac{\tilde{v}_\xi}{v} - \frac{\bar{v}_\xi\tilde{v}}{v\bar{v}} \bigg) \frac{\tilde{\phi}_{\xi\xi}}{\bar{v}^2} \, d\xi, \\
\tilde{P}_3(W) &:= -2 \int_\mathbb{R} (a^X ve^\phi)_\xi \frac{\tilde{\phi}_\xi \tilde{\phi}_{\xi\xi}}{\bar{v}^2} \, d\xi, & \tilde{P}_4(W) &:= -2 \int_\mathbb{R} \left( a^X v (e^\phi - e^{\bar{\phi}}) \bar{\phi}_\xi \right)_\xi \frac{\tilde{\phi}_{\xi\xi}}{\bar{v}^2} \, d\xi, \\
\tilde{P}_5(W) &:= 2 \int_\mathbb{R} a^X v \left( \mathcal{P}_{\mathrm{prof},\xi} + \mathcal{N}_\xi \right) \bigg( \frac{\tilde{\phi}_{\xi\xi}}{\bar{v}^2} \bigg)_\xi \, d\xi, & \tilde{P}_6(W) &:= 2 \int_\mathbb{R} a^X \frac{v\bar{v}_\xi^2 \tilde{\phi}_{\xi\xi}^2}{\bar{v}^6} \, d\xi.
\end{aligned}
\end{equation*}
\end{lemma}

\begin{proof}
We first use the definition of $\tilde{p}(v)$ to obtain
\begin{equation*}
(\ln v - \ln \bar{v})_\xi = - \frac{\left( \tilde{p}(v) - \tilde{p}(\bar{v}) \right)_\xi}{\tilde{p}(v)} + \frac{\bar{v}_\xi \tilde{v}}{\bar{v}^2}.
\end{equation*}
Using this identity and integration by parts, we have
\begin{equation*}
\begin{split}
\mathcal{D} - \tilde{\mathcal{P}} &= \int_\mathbb{R} a^X \frac{\left( \tilde{p}(v) - \tilde{p}(\bar{v}) \right)_\xi^2}{\tilde{p}(v)} \, d\xi + \int_\mathbb{R} a^X (\ln v - \ln \bar{v})_{\xi\xi} \frac{\tilde{\phi}_{\xi\xi}}{\bar{v}^2} \, d\xi \\
& = \int_\mathbb{R} a^X \frac{\left( \tilde{p}(v) - \tilde{p}(\bar{v}) \right)_\xi^2}{\tilde{p}(v)} \, d\xi + \int_\mathbb{R} a^X \frac{\left( \tilde{p}(v) - \tilde{p}(\bar{v}) \right)_\xi}{\tilde{p}(v)} \bigg( \frac{\tilde{\phi}_{\xi\xi}}{\bar{v}^2} \bigg)_\xi \, d\xi - \tilde{P}_1 - \tilde{P}_2.
\end{split}
\end{equation*}
We apply the identity
\begin{equation*}
A^2 + AB = \frac{1}{4} A^2 + \frac{3}{4} (A+2B)^2 + B^2 - 2(A+2B) B
\end{equation*}
with $A= (\tilde{p}(v)-\tilde{p}(\bar{v}))_\xi$ and $B = (\tilde{\phi}_{\xi\xi}/\bar{v}^2)_\xi$. Then
\begin{equation} \label{D-P}
\mathcal{D}-\tilde{\mathcal{P}} = \frac{1}{4} \mathcal{D} + \frac{3}{4} \mathcal{D}_z + \frac12 \int_\mathbb{R} a^X v \bigg( \frac{\tilde{\phi}_{\xi\xi}}{\bar{v}^2} \bigg)_\xi^2 \, d\xi - \int_\mathbb{R} a^X v z_\xi \bigg( \frac{\tilde{\phi}_{\xi\xi}}{\bar{v}^2} \bigg)_\xi \, d\xi - \tilde{P}_1 - \tilde{P}_2.
\end{equation}
Differentiating \eqref{zPoisson} with respect to $\xi$, we have
\begin{equation*}
z_\xi = 2 e^\phi \tilde{\phi}_\xi + 2(e^\phi - e^{\bar{\phi}}) \bar{\phi}_\xi + 2 \mathcal{P}_{\mathrm{prof},\xi} + 2\mathcal{N}_\xi.
\end{equation*}
Substituting this into \eqref{D-P} and applying integration by parts, we obtain
\begin{equation} \label{D-P_1}
- \int_\mathbb{R} a^X v z_\xi \bigg( \frac{\tilde{\phi}_{\xi\xi}}{\bar{v}^2} \bigg)_\xi \, d\xi = \mathcal{G}_3 - \tilde{P}_3 - \tilde{P}_4 - \tilde{P}_5.
\end{equation}
For the third term on the right-hand side of \eqref{D-P}, we find that
\begin{equation} \label{D-P_2}
\frac12 \int_\mathbb{R} a^X v \bigg( \frac{\tilde{\phi}_{\xi\xi}}{\bar{v}^2} \bigg)_\xi^2 \, d\xi = \frac12 \int_\mathbb{R} a^X v \bigg( \frac{\tilde{\phi}_{\xi\xi\xi}}{\bar{v}^2} - \frac{2\bar{v}_\xi \tilde{\phi}_{\xi\xi}}{\bar{v}^3} \bigg)^2 \, d\xi  \geq \mathcal{G}_4 - \tilde{P}_6.
\end{equation}
Combining \eqref{D-P}, \eqref{D-P_1}, and \eqref{D-P_2}, we obtain the desired inequality.
\end{proof}

\subsection{Preliminary estimates}

Before proceeding to the main estimates, we present several preliminary estimates that will be used throughout the rest of the paper.

\begin{lemma}
The shock profile $(\bar{v}^S,\bar{u}^S,\bar{\phi}^S)$ and the approximate rarefaction wave $(\bar{v}^R,\bar{u}^R,\bar{\phi}^R)$ satisfy
\begin{equation} \label{shderiv1}
\left| \frac{d^k \bar{v}^S}{d\xi^k} \right|, \left| \frac{d^k \bar{u}^S}{d\xi^k}  \right|, \left| \frac{d^k \bar{\phi}^S}{d\xi^k}  \right| \leq C \bar{v}^S_\xi \leq C \delta_S^2
\end{equation}
and
\begin{equation} \label{raderiv}
\left| \frac{\partial^k \bar{v}^R}{\partial \xi^k} \right|, \left| \frac{\partial^k \bar{u}^R}{\partial \xi^k} \right|, \left| \frac{\partial^k \bar{\phi}^R}{\partial \xi^k} \right| \leq C \bar{v}^R_\xi \leq C \delta_R,
\end{equation}
where $k=1,2,3$ and $C>0$ is a generic constant.
\end{lemma}

\begin{proof}
As shown in \cite[Lemma~4.6]{KKSh}, the shock profile $(\bar{v}^S, \bar{u}^S, \bar{\phi}^S)$ satisfies the bound \eqref{shderiv1} for $k \in \mathbb{N}$, and \eqref{raderiv} is directly follows from Lemma~\ref{rarefaction}.
\end{proof}

\begin{lemma} \label{waveinter}
Under the assumptions of Proposition~\ref{Apriori}, there exists a constant $C>0$ such that
\begin{equation} \label{inter'}
\begin{split}
& \lVert (\bar{v}^S-v_m)(\bar{v}^R-v_m) \rVert_{L^2(\mathbb{R})} \leq C \delta_R \delta_S^{1/2} e^{-C \delta_S t}, \\
& \lVert \bar{v}^S_\xi (\bar{v}^R - v_m) \rVert_{L^2(\mathbb{R})} + \lVert \bar{v}^R_\xi \bar{v}^S_\xi \rVert_{L^2(\mathbb{R})} \leq C \delta_R \delta_S^{3/2} e^{-C \delta_S t}, \\
& \lVert \bar{v}^R_\xi (\bar{v}^S - v_m) \rVert_{L^2(\mathbb{R})} \leq C \delta_R \delta_S e^{-C \delta_S t}.
\end{split}
\end{equation}
In addition, it holds that
\begin{equation} \label{Raest'}
\begin{split}
\lVert \bar{v}^R_{\xi\xi} \rVert_{L^p}, \lVert \bar{v}^R_{\xi\xi\xi} \rVert_{L^p} & \leq \begin{cases}
\delta_R & \text{if } 1+t \leq \delta_R^{-1} \\
\frac{1}{1+t} & \text{if } 1+t \geq \delta_R^{-1}
\end{cases} \quad \text{for } p=1,2, \\
\lVert \bar{v}^R_\xi \rVert_{L^4} & \leq \begin{cases}
\delta_R & \text{if } 1+t \leq \delta_R^{-1} \\
\delta_R^{1/4} \frac{1}{(1+t)^{3/4}} & \text{if } 1+t \geq \delta_R^{-1}
\end{cases}
\end{split}
\end{equation}
for all $t\geq 0$.
\end{lemma}

\begin{proof}
With the bounds from Lemma~\ref{Prop.1.1} and Lemma~\ref{rarefaction}, the estimates in \eqref{inter'} follow from the proof in \cite[Lemma~4.2]{KVW2}. In addition, \eqref{Raest} holds by Lemma~\ref{rarefaction}.
\end{proof}

As a consequence of Lemma~\ref{waveinter}, we obtain the following time-integrated bounds.

\begin{corollary}
Under the assumptions of Proposition~\ref{Apriori}, there exists a constant $C>0$ such that
\begin{equation} \label{inter}
\begin{split}
& \int_0^\infty \left( \lVert \bar{v}^S_\xi (\bar{v}^R-v_m)\rVert_{L^2} + \lVert \bar{v}^R_\xi (\bar{v}^S-v_m)\rVert_{L^2} + \lVert \bar{v}^R_\xi \bar{v}^S_\xi \rVert_{L^2} \right) \, d\tau \leq C \delta_R, \\
& \int_0^\infty \lVert (\bar{v}^S-v_m)(\bar{v}^R-v_m) \rVert_{L^2}^2 \, d\tau \leq C \delta_R,
\end{split}
\end{equation}
and
\begin{equation} \label{Raest}
\begin{array}{ll}
\int_0^\infty \left( \lVert \bar{v}^R_{\xi\xi} \rVert_{L^1}^{4/3} + \lVert \bar{v}^R_{\xi\xi\xi} \rVert_{L^1}^{4/3} \right) \, d\tau \leq C \delta_R^{1/3}, & \int_0^\infty \left( \lVert \bar{v}^R_{\xi\xi} \rVert_{L^2}^2 + \lVert \bar{v}^R_{\xi\xi\xi} \rVert_{L^2}^2 \right) \, d\tau \leq C \delta_R, \\
\int_0^\infty \lVert \bar{v}^R_\xi \rVert_{L^4}^2 \, d\tau \leq C \delta_R, & \int_0^\infty \lVert \bar{v}^R_\xi \rVert_{L^4}^4 \, d\tau \leq C \delta_R^3.
\end{array}
\end{equation}
\end{corollary}

\begin{lemma}
Under the assumptions of Proposition~\ref{Apriori}, there exists a constant $C>0$ such that
\begin{subequations}
\begin{align}
&\label{hL2}\int_\mathbb{R} \bar{v}^S_\xi \tilde{h}^2 \, d\xi \leq C \left( G_1 + G^S \right), \\
&\label{vL2} \int_\mathbb{R} \bar{v}_\xi \tilde{v}^2 \, d\xi \leq C \left( G^S + G^R \right), \\
&\label{vxiL2}\int_\mathbb{R} \tilde{v}_\xi^2 \, d\xi \leq C \left( G^S + G^R + D \right),
\end{align}
\end{subequations}
and
\begin{equation} \label{hxiL2}
\begin{split}
\int_\mathbb{R} \tilde{h}_\xi^2 \, d\xi & \leq C \left( G^S + G^R + D \right) + C \int_\mathbb{R} \left( \tilde{v}_{\xi\xi}^2 + \tilde{u}_\xi^2 \right)  \, d\xi  \\
& \quad + C \left( \lVert \bar{v}^S_\xi(\bar{v}^R-v_m) \rVert_{L^2}^2 + \lVert \bar{v}^S_\xi\bar{v}^R_\xi \rVert_{L^2}^2 + \lVert \bar{v}^R_{\xi\xi} \rVert_{L^2}^2 \right)
\end{split}
\end{equation}
for all $t \in [0,T]$, where $G_1$, $G^S$, $G^R$, and $D$ are as in Lemma~\ref{lem_0th}.
\end{lemma}

\begin{proof}
Using the inequality $|a|^2 \leq 2 |a-b|^2 + 2|b|^2$, we estimate
\begin{equation*}
\begin{split}
\int_\mathbb{R} \bar{v}^S_\xi \tilde{h}^2 \, d\xi & \leq C \left( \int_\mathbb{R} \bar{v}^S_\xi \bigg| \tilde{h} - \frac{\tilde{p}(v)-\tilde{p}(\bar{v})}{\sigma} \bigg|^2 \, d\xi + \frac{1}{\sigma^2} \int_\mathbb{R} \bar{v}^S_\xi |\tilde{p}(v)-\tilde{p}(\bar{v})|^2 \, d\xi \right) \leq C \left( \sqrt{\delta_S} G_1 + G^S \right). 
\end{split}
\end{equation*}
The bound \eqref{vL2} follows from the definition of $\tilde{p}(\cdot)$:
\begin{equation*}
\begin{split}
\int_\mathbb{R} \bar{v}_\xi \tilde{v}^2 \, d\xi & \leq C \int_\mathbb{R} \bar{v}^S_\xi |\tilde{p}(v)-\tilde{p}(\bar{v})|^2 \, d\xi + C\int_\mathbb{R}  \bar{v}^R_\xi |\tilde{v}|^2 \, d\xi = C (G^S + G^R).
\end{split}
\end{equation*}
To obtain \eqref{vxiL2}, we write
\begin{equation*}
\begin{split}
\left( \tilde{p}(v)-\tilde{p}(\bar{v}) \right)_\xi^2 & = 4  \bigg( \frac{\tilde{v}_\xi^2}{v^4} - \frac{2\bar{v}_\xi\tilde{v}_\xi (\tilde{v}^2 + 2\bar{v}\tilde{v})}{v^4\bar{v}^2} + \frac{\bar{v}_\xi^2 (\tilde{v}^4 + 4\bar{v}\tilde{v}^3 + 4\bar{v}^2\tilde{v}^2)}{v^4\bar{v}^4} \bigg).
\end{split}
\end{equation*}
Using this and applying Young's inequality, we obtain
\begin{equation*}
\begin{split}
\int_\mathbb{R} \tilde{v}_\xi^2 \, d\xi & \leq C \int_\mathbb{R} \left( \tilde{p}(v)-\tilde{p}(\bar{v}) \right)_\xi^2 \, d\xi + C \int_\mathbb{R} \left(|\bar{v}_\xi||\tilde{v}_\xi||\tilde{v}| + |\bar{v}_\xi||\tilde{v}|^2 \right) \, d\xi \\
& \leq C \int_\mathbb{R} \left( \tilde{p}(v)-\tilde{p}(\bar{v}) \right)_\xi^2 \, d\xi + C \int_\mathbb{R} |\bar{v}_\xi| \left( |\tilde{v}_\xi|^2 + |\tilde{v}|^2 \right) \, d\xi + C \int_\mathbb{R} \bar{v}_\xi \tilde{v}^2 \, d\xi \\
& \leq C \int_\mathbb{R} \left( \tilde{p}(v)-\tilde{p}(\bar{v}) \right)_\xi^2 \, d\xi + C (\delta_S^2 + \delta_R) \int_\mathbb{R}  \tilde{v}_\xi^2 \, d\xi + C \int_\mathbb{R} \bar{v}_\xi \tilde{v}^2 \, d\xi.
\end{split}
\end{equation*}
Thus, there exists $C>0$ such that
\begin{equation*}
\int_\mathbb{R} \tilde{v}_\xi^2 \, d\xi \leq C \left( G^S + G^R + D \right)
\end{equation*}
for sufficiently small $\delta_S + \delta_R$.

We recall the definition of $\tilde{h}$:
\begin{equation*}
\tilde{h} = \tilde{u} - \left( \ln{v} - \ln{\bar{v}^S} \right)_\xi = \tilde{u} - \frac{\tilde{v}_\xi + \bar{v}^R_\xi}{v} + \frac{\bar{v}^S_\xi \left( \tilde{v} + \bar{v}^R -v_m \right)}{v\bar{v}^S}.
\end{equation*}
Differentiating with respect to $\xi$, we have
\begin{equation*}
\begin{split}
\tilde{h}_\xi & = \tilde{u}_\xi - \frac{\tilde{v}_{\xi\xi} + \bar{v}^R_{\xi\xi}}{v} + \frac{(\tilde{v}_\xi + \bar{v}^R_\xi + \bar{v}^S_\xi)(\tilde{v}_\xi + \bar{v}^R_\xi)}{v^2} + \frac{\bar{v}^S_{\xi\xi}(\tilde{v} + \bar{v}^R-v_m)}{v\bar{v}^S} + \frac{\bar{v}^S_\xi \left( \tilde{v}_\xi + \bar{v}^R_\xi \right)}{v\bar{v}^S}  \\
& \quad - \frac{(\bar{v}^S_\xi)^2 \left( \tilde{v} + \bar{v}^R - v_m \right) }{v(\bar{v}^S)^2} - \frac{\bar{v}^S_\xi (\tilde{v}_\xi+\bar{v}^R_\xi+\bar{v}^S_\xi)(\tilde{v} + \bar{v}^R-v_m)}{v^2\bar{v}^S}.
\end{split}
\end{equation*}
Hence,
\begin{equation*}
|\tilde{h}_\xi| \leq C \left( | \tilde{u}_\xi | + |\tilde{v}_\xi| + |\tilde{v}_{\xi\xi}| + |\bar{v}_\xi||\tilde{v}| + |\bar{v}^S_\xi(\bar{v}^R-v_m)| + |\bar{v}^S_\xi\bar{v}^R_\xi| + |\bar{v}^R_{\xi\xi}| + |\bar{v}^R_\xi|^2 \right),
\end{equation*}
where in the inequality we used $|\bar{v}^S_{\xi\xi}| \leq C|\bar{v}^S_\xi|$. Moreover, integration by parts and the Cauchy--Schwarz inequality, together with $|\bar{v}^R-v_m| \leq \delta_R$, yield
\begin{equation*}
\lVert \bar{v}^R_\xi \rVert_{L^4}^4 \leq 9\delta_R^2 \lVert \bar{v}^R_{\xi\xi} \rVert_{L^2}^2.
\end{equation*}
Squaring the bound for $|\tilde{h}_\xi|$, integrating over $\mathbb{R}$, and using this estimate, we obtain
\begin{equation*}
\begin{split}
\int_\mathbb{R} \tilde{h}_\xi^2 \, d\xi & \leq C  \int_\mathbb{R} \left( \bar{v}_\xi \tilde{v}^2 + \tilde{v}_\xi^2 + \tilde{v}_{\xi\xi}^2 + \tilde{u}_\xi^2  \right) \, d\xi  \\
& \quad + C \left( \lVert \bar{v}^S_\xi(\bar{v}^R-v_m) \rVert_{L^2}^2 + \lVert \bar{v}^S_\xi\bar{v}^R_\xi \rVert_{L^2}^2 + \lVert \bar{v}^R_{\xi\xi} \rVert_{L^2}^2 \right),
\end{split}
\end{equation*}
which completes the proof.
\end{proof}

\begin{lemma} \label{lem_inf}
Under the assumptions of Proposition~\ref{Apriori}, there exists a constant $C>0$ such that
\begin{equation} \label{infty}
\lVert \tilde{h}(t,\cdot) \rVert_{L^\infty} + \lVert (\tilde{v},\tilde{u})(t,\cdot) \rVert_{W^{1,\infty}} + \lVert \tilde{\phi}(t,\cdot) \rVert_{W^{2,\infty}} \leq C \left( \delta_R + \varepsilon_1 \right)
\end{equation}
and
\begin{equation} \label{Xbd}
\lvert \dot{X}(t) \rvert \leq C \lVert \tilde{v} \rVert_{L^\infty} \leq C \varepsilon_1
\end{equation}
for all $t \in [0,T]$, where $C>0$ is a generic constant.
\end{lemma}

\begin{proof}
A standard Sobolev inequality implies
\begin{equation*}
\begin{split}
& \lVert \tilde{h}(t,\cdot) \rVert_{L^\infty} + \lVert (\tilde{v},\tilde{u})(t,\cdot) \rVert_{W^{1,\infty}} + \lVert \tilde{\phi}(t,\cdot) \rVert_{W^{2,\infty}} \\
& \quad \leq C \left( \lVert \tilde{h}(t,\cdot) \rVert_{H^1} + \lVert (\tilde{v},\tilde{u})(t,\cdot) \rVert_{H^2} + \lVert \tilde{\phi}(t,\cdot) \rVert_{H^3} \right).
\end{split}
\end{equation*}
Here $\lVert (\tilde{v},\tilde{u}) \rVert_{H^2}$ and $\lVert \tilde{\phi} \rVert_{H^3}$ are bounded by $\varepsilon_1$ as assumed in \eqref{AssH2}. Since $\tilde{h} = \tilde{u} - \left( \ln{v} - \ln{\bar{v}^S} \right)_\xi$,
we have
\begin{equation*}
\begin{split}
\int_\mathbb{R} \tilde{h}^2 \, d\xi \leq C \int_\mathbb{R} \left( \tilde{v}^2 + \tilde{v}_\xi^2 + \tilde{u}^2 \right) \, d\xi + C \left( \|\bar{v}^R_\xi\|_{L^2}^2 + \|\bar{v}^S_\xi(\bar{v}^R-v_m)\|_{L^2}^2 \right)
\end{split}
\end{equation*}
and \eqref{hxiL2}. Thus, the bound \eqref{infty} follows from the assumption \eqref{AssH2} together with the bounds from Lemma~\ref{waveinter}.

To show \eqref{Xbd}, we estimate the ODE \eqref{XODE} using \eqref{shderiv1} as
\begin{equation*}
\begin{split}
|\dot{X}(t)| & \leq \frac{C}{\delta_S} \int_\mathbb{R} \left( |\bar{u}^S_\xi| + |\bar{v}^S_\xi|^2 + |\bar{v}^S_{\xi\xi}| + |\tilde{p}'(\bar{v}^S)| |\bar{v}^S_\xi| \right) |\tilde{v}| \, d\xi \\
& \leq \frac{C}{\delta_S} \lVert \tilde{v} (t,\cdot) \rVert_{L^\infty} \int_\mathbb{R} \left( |\bar{u}^S_\xi| + |\bar{v}^S_\xi| + |\bar{v}^S_{\xi\xi}| \right) \, d\xi \leq C \lVert \tilde{v}(t,\cdot) \rVert_{L^\infty}.
\end{split}
\end{equation*}
This concludes the proof.
\end{proof}

\begin{lemma} \label{V_est}
Under the assumptions of Proposition~\ref{Apriori}, there exists a constant $C>0$ such that
\begin{subequations}
\begin{align}
\begin{split}
\label{V}|\mathcal{V}| & \leq C \left( |\bar{v}^S_\xi (\bar{v}^R-v_m)| + |\bar{v}^R_\xi\bar{v}^S_\xi| + |\bar{v}^R_\xi|^2 + |\bar{v}^R_{\xi\xi}| + |\bar{v}^S-v_m| |\bar{v}^R-v_m| \right) \\
& \quad + C \left( \delta_S^2 + \delta_R + \varepsilon_1 \right) \left( |\tilde{v}| + |\tilde{v}_\xi| + |\tilde{\phi}| + |\tilde{\phi}_\xi| + |\tilde{\phi}_{\xi\xi}| \right),
\end{split} \\
\begin{split}
\label{Vxi}|\mathcal{V}_\xi| & \leq C \left( | \bar{v}^S_\xi (\bar{v}^R - v_m) | + |\bar{v}^R_\xi (\bar{v}^S - v_m) | + | \bar{v}^R_\xi \bar{v}^S_\xi | + |\bar{v}^R_\xi|^2 + |\bar{v}^R_{\xi\xi\xi}| \right) \\
& \quad + C |\bar{v}_\xi| \left( |\tilde{v}| + |\tilde{\phi}| \right)  + C ( \delta_S^2 + \delta_R + \varepsilon_1 ) \left( |\tilde{v}_\xi|  + |\tilde{v}_{\xi\xi}| + |\tilde{\phi}_\xi| + |\tilde{\phi}_{\xi\xi}| + |\tilde{\phi}_{\xi\xi\xi}| \right),
\end{split} \\
\begin{split}
\label{Vt}|\mathcal{V}_t| & \leq C \left( |\bar{v}^R_\xi(\bar{v}^S-v_m)| + |\bar{v}^R_\xi \bar{v}^S_\xi| + |\bar{v}^R_\xi|^2 + |\bar{v}^R_{\xi\xi\xi}| \right) + C |\dot{X}||\bar{v}^S_\xi| + C |\bar{v}_\xi| \left( |\tilde{v}| + |\tilde{\phi}| \right) \\
& \quad + C \left( \delta_S^2 + \delta_R + \varepsilon_1 \right) \left( |\tilde{v}_\xi| + |\tilde{v}_{\xi\xi}| + |\tilde{u}_\xi| + |\tilde{u}_{\xi\xi}| + |\tilde{\phi}_\xi| + |\tilde{\phi}_{\xi\xi}| + |\tilde{\phi}_t| + |\tilde{\phi}_{\xi t}| + |\tilde{\phi}_{\xi\xi t}| \right)
\end{split}
\end{align}
\end{subequations}
for all $t \in [0,T]$, where $\mathcal{V}$ is as defined in \eqref{v}--\eqref{mcV}.
\end{lemma}

\begin{proof}
Recall that $\mathcal{V} = \mathcal{V}^I + \mathcal{V}^L + \mathcal{V}^N$ with \eqref{mcV}, where $\mathcal{V}^I$ is given by
\begin{equation*}
\begin{split}
\mathcal{V}^I & = e^{-\bar{\phi}^S} \bigg[ v_m^{-1} \bar{v}^R \bigg( \frac{\bar{\phi}_{\xi\xi}}{\bar{v}} - \frac{\bar{v}_\xi \bar{\phi}_\xi}{\bar{v}^2} \bigg) - \bigg( \frac{\bar{\phi}^S_{\xi\xi}}{\bar{v}^S} - \frac{\bar{v}^S_\xi \bar{\phi}^S_\xi}{(\bar{v}^S)^2} \bigg) \bigg] + (e^{-(\bar{\phi}^S-\phi_m)} - 1) (\bar{v}^R - v_m).
\end{split}
\end{equation*}
To estimate the first term, we write
\begin{equation*}
\begin{split}
\frac{v_m^{-1} \bar{v}^R (\bar{\phi}^R_{\xi\xi} + \bar{\phi}^S_{\xi\xi}) }{\bar{v}} - \frac{\bar{\phi}^S_{\xi\xi}}{\bar{v}^S} & = \frac{ v_m^{-1}\bar{v}^R \bar{v}^S \bar{\phi}^R_{\xi\xi} - (\bar{v}^R - v_m) \bar{\phi}^S_{\xi\xi} + v_m^{-1} \bar{v}^S ( \bar{v}^R - v_m)\bar{\phi}^S_{\xi\xi} }{\bar{v}\bar{v}^S}
\end{split}
\end{equation*}
and
\begin{equation*}
\begin{split}
& \frac{v_m^{-1}\bar{v}^R (\bar{v}^R_\xi + \bar{v}^S_\xi)(\bar{\phi}^R_\xi + \bar{\phi}^S_\xi)}{\bar{v}^2} - \frac{\bar{v}^S_\xi \bar{\phi}^S_\xi}{(\bar{v}^S)^2} \\
& \quad = \frac{v_m^{-1}\bar{v}^R }{\bar{v}^2} ( \bar{v}^R_\xi \bar{\phi}^R_\xi + \bar{v}^R_\xi \bar{\phi}^S_\xi + \bar{v}^S_\xi \bar{\phi}^R_\xi )  + \frac{( \bar{v}^R - v_m) \bar{v}^S_\xi \bar{\phi}^S_\xi }{\bar{v}^2(\bar{v}^S)^2} \left( v_m^{-1} (\bar{v}^S)^2 - 2 \bar{v}^S - (\bar{v}^R - v_m) \right).
\end{split}
\end{equation*}
For the last term, we use the bounds
\begin{equation*}
| e^{-(\bar{\phi}^S - \phi_m)} - 1| \leq C |\bar{\phi}^S - \phi_m| \leq C |\bar{v}^S - v_m|
\end{equation*}
for sufficiently small amplitude $|\bar{\phi}^S - \phi_m| \sim \delta_S$. These together with \eqref{shderiv1} and \eqref{raderiv} yield
\begin{equation*}
\begin{split}
|\mathcal{V}^I| & \leq C \left( |\bar{v}^S_\xi (\bar{v}^R-v_m)| + |\bar{v}^R_\xi\bar{v}^S_\xi| + |\bar{v}^R_\xi|^2 + |\bar{v}^R_{\xi\xi}| + |\bar{v}^S-v_m| |\bar{v}^R-v_m| \right).
\end{split}
\end{equation*}
The linear part $\mathcal{V}^L$ and the nonlinear part $\mathcal{V}^N$, defined in \eqref{mcV}, can be bounded as
\begin{equation*}
\begin{split}
|\mathcal{V}^L| & \leq C |\bar{v}_\xi| \left( |\tilde{v}| + |\tilde{v}_\xi| + |\tilde{\phi}| + |\tilde{\phi}_\xi| \right)
\end{split}
\end{equation*}
and
\begin{equation} \label{Vnest}
\begin{split}
|\mathcal{V}^N| & \leq C \left( |\tilde{v}| + |\tilde{\phi}| + |\tilde{\phi}_\xi| \right) \left( |\bar{v}_\xi| |\tilde{v}| + |\tilde{v}_\xi| + |\tilde{\phi}_\xi| + |\tilde{\phi}_{\xi\xi}| \right) +  C \left( |\bar{v}_\xi| + |\tilde{\phi}| \right) |\tilde{\phi}|,
\end{split}
\end{equation}
respectively. Collecting these estimates, and applying the bounds \eqref{shderiv1}, \eqref{raderiv}, and \eqref{infty}, we have \eqref{V}.

To obtain \eqref{Vxi}, we first see $\mathcal{V}^I_\xi$:
\begin{equation*}
\begin{split}
\mathcal{V}^I_\xi & = (e^{-\bar{\phi}})_\xi \bigg( \frac{\bar{\phi}_\xi}{\bar{v}} \bigg)_\xi  - (e^{-\bar{\phi}^S})_\xi \bigg( \frac{\bar{\phi}^S_\xi}{\bar{v}^S} \bigg)_\xi + e^{-\bar{\phi}} \bigg( \frac{\bar{\phi}_\xi}{\bar{v}} \bigg)_{\xi\xi} - e^{-\bar{\phi}^S} \bigg( \frac{\bar{\phi}^S_\xi}{\bar{v}^S} \bigg)_{\xi\xi} + \left( e^{-\bar{\phi}} - e^{-\bar{\phi}^S} - \bar{v}^R \right)_\xi.
\end{split}
\end{equation*}
The first two terms can be rewritten as
\begin{equation*}
\begin{split}
& (e^{-\bar{\phi}})_\xi \bigg( \frac{\bar{\phi}_\xi}{\bar{v}} \bigg)_\xi  - (e^{-\bar{\phi}^S})_\xi \bigg( \frac{\bar{\phi}^S_\xi}{\bar{v}^S} \bigg)_\xi \\
& \quad = - e^{-\bar{\phi}^S} \bar{\phi}^S_\xi \bigg[ \bigg( \frac{\bar{v}^R \bar{\phi}_{\xi\xi}  }{v_m \bar{v}} - \frac{\bar{\phi}^S_{\xi\xi}}{\bar{v}^S} \bigg) - \bigg( \frac{\bar{v}^R \bar{v}_\xi \bar{\phi}_\xi }{v_m\bar{v}^2} - \frac{\bar{v}^S_\xi \bar{\phi}^S_\xi}{(\bar{v}^S)^2} \bigg) \bigg] - e^{-\bar{\phi}^S} v_m^{-1} \bar{v}^R \bar{\phi}^R_\xi  \bigg( \frac{\bar{\phi}_{\xi\xi}}{\bar{v}} - \frac{\bar{v}_\xi \bar{\phi}_\xi}{\bar{v}^2} \bigg) \\
& \quad = - e^{-\bar{\phi}^S} \bar{\phi}^S_\xi  \bigg( \frac{ v_m^{-1}\bar{v}^R \bar{v}^S \bar{\phi}^R_{\xi\xi} - (\bar{v}^R - v_m) \bar{\phi}^S_{\xi\xi} + v_m^{-1} \bar{v}^S ( \bar{v}^R - v_m)\bar{\phi}^S_{\xi\xi} }{\bar{v}\bar{v}^S} \bigg) \\
& \qquad + e^{-\bar{\phi}^S}\bar{\phi}^S_\xi \bigg( \frac{v_m^{-1}\bar{v}^R }{\bar{v}^2} ( \bar{v}^R_\xi \bar{\phi}^R_\xi + \bar{v}^R_\xi \bar{\phi}^S_\xi + \bar{v}^S_\xi \bar{\phi}^R_\xi ) + \frac{( \bar{v}^R - v_m) \bar{v}^S_\xi \bar{\phi}^S_\xi }{\bar{v}^2(\bar{v}^S)^2} \left( v_m^{-1} (\bar{v}^S)^2 - 2 \bar{v}^S - (\bar{v}^R - v_m) \right) \bigg) \\
& \qquad - e^{-\bar{\phi}^S} v_m^{-1}\bar{v}^R \bar{\phi}^R_\xi \bigg( \frac{\bar{\phi}^R_{\xi\xi}}{\bar{v}} - \frac{\bar{v}^R_\xi \bar{\phi}^R_\xi}{\bar{v}^2} \bigg) -  e^{-\bar{\phi}^S} v_m^{-1}\bar{v}^R \bar{\phi}^R_\xi \bigg( \frac{\bar{\phi}^S_{\xi\xi}}{\bar{v}} - \frac{\bar{v}^R_\xi \bar{\phi}^S_\xi + \bar{v}^S_\xi \bar{\phi}^R_\xi + \bar{v}^S_\xi \bar{\phi}^S_\xi}{\bar{v}^2} \bigg),
\end{split}
\end{equation*}
and so, by \eqref{shderiv1}--\eqref{raderiv}, we find
\begin{equation*}
\begin{split}
(e^{-\bar{\phi}})_\xi \bigg( \frac{\bar{\phi}_\xi}{\bar{v}} \bigg)_\xi  - (e^{-\bar{\phi}^S})_\xi \bigg( \frac{\bar{\phi}^S_\xi}{\bar{v}^S} \bigg)_\xi \leq C \left(  |\bar{v}^S_\xi \bar{v}^R_\xi| + |\bar{v}^S_\xi(\bar{v}^R-v_m)| + |\bar{v}^R_\xi|^2 \right).
\end{split}
\end{equation*}
In a similar manner, we estimate the next two terms
\begin{equation*}
\begin{split}
e^{-\bar{\phi}} \left( \frac{\bar{\phi}_\xi}{\bar{v}} \right)_{\xi\xi} - e^{-\bar{\phi}^S} \left( \frac{\bar{\phi}^S_\xi}{\bar{v}^S} \right)_{\xi\xi} \quad \leq C \left( |\bar{v}^R_{\xi\xi\xi}| + |\bar{v}^S_\xi(\bar{v}^R-v_m)| + |\bar{v}^R_\xi|^2 + |\bar{v}^R_\xi\bar{v}^S_\xi| \right). 
\end{split}
\end{equation*}
For the last term, we compute
\begin{equation*}
\begin{split}
\left( e^{-\bar{\phi}} - e^{-\bar{\phi}^S} - \bar{v}^R \right)_\xi &  = - e^{-\bar{\phi}^S} \bar{\phi}^S_\xi v_m^{-1} \left( \bar{v}^R - v_m \right) + \bar{v}^R_\xi \left( e^{- (\bar{\phi}^S - \phi_m)} - 1 \right) \\
& \leq C \left( |\bar{v}^S_\xi(\bar{v}^R - v_m)| + |\bar{v}^R_\xi(\bar{v}^S-v_m)| \right)
\end{split}
\end{equation*}
for sufficiently small $\delta_S$. Collecting the above estimates, we obtain
\begin{equation} \label{VIxi}
|\mathcal{V}^I_\xi| \leq C \left( | \bar{v}^S_\xi (\bar{v}^R - v_m) | + |\bar{v}^R_\xi (\bar{v}^S - v_m) | + | \bar{v}^R_\xi \bar{v}^S_\xi | + |\bar{v}^R_\xi|^2 + |\bar{v}^R_{\xi\xi\xi}| \right).
\end{equation}
Thanks to \eqref{shderiv1}--\eqref{raderiv}, the linear part $\mathcal{V}^L_\xi$ is estimated as
\begin{equation} \label{VLxi}
\begin{split}
|\mathcal{V}^L_\xi| & = \bigg| \bigg( - \frac{e^{-\bar{\phi}}\bar{\phi}_{\xi\xi}}{\bar{v}} \tilde{\phi} - \frac{e^{-\bar{\phi}}\bar{\phi}_{\xi\xi}}{\bar{v}^2}\tilde{v} + \frac{e^{-\bar{\phi}}\bar{v}_\xi \bar{\phi}_\xi}{\bar{v}^2} \tilde{\phi} - \frac{e^{-\bar{\phi}}\bar{\phi}_\xi}{\bar{v}^2}\tilde{v}_\xi - \frac{e^{-\bar{\phi}}\bar{v}_\xi}{\bar{v}^2}\tilde{\phi}_\xi + \frac{2e^{-\bar{\phi}}\bar{v}_\xi\bar{\phi}_\xi}{\bar{v}^3}\tilde{v} \bigg)_\xi \bigg| \\
& \leq C |\bar{v}_\xi| \left( |\tilde{\phi}| + |\tilde{\phi}_\xi| + |\tilde{\phi}_{\xi\xi}| + |\tilde{v}| + |\tilde{v}_\xi| + |\tilde{v}_{\xi\xi}| \right) \\
& \leq C |\bar{v}_\xi| \left( |\tilde{v}| + |\tilde{\phi}| \right) + C (\delta_S^2 + \delta_R) \left( |\tilde{v}_\xi| + |\tilde{v}_{\xi\xi}|  + |\tilde{\phi}_\xi| + |\tilde{\phi}_{\xi\xi}| \right).
\end{split}
\end{equation}
By \eqref{infty}, the nonlinear part $\mathcal{V}^N_\xi$ is bounded as
\begin{equation} \label{VNxi}
\begin{split}
|\mathcal{V}^N_\xi| & \leq C |\bar{v}_\xi| \left( |\tilde{v}| + |\tilde{\phi}| \right) + C \left( \lVert \tilde{v} \rVert_{W^{1,\infty}} + \lVert \tilde{\phi} \rVert_{W^{2,\infty}} \right) \left( |\tilde{v}_\xi| + |\tilde{v}_{\xi\xi}| + |\tilde{\phi}_\xi| + |\tilde{\phi}_{\xi\xi}| + |\tilde{\phi}_{\xi\xi\xi}| \right) \\
& \leq C |\bar{v}_\xi| \left( |\tilde{v}| + |\tilde{\phi}| \right) + C ( \delta_R +\varepsilon_1) \left( |\tilde{v}_\xi| + |\tilde{v}_{\xi\xi}| + |\tilde{\phi}_\xi| + |\tilde{\phi}_{\xi\xi}| + |\tilde{\phi}_{\xi\xi\xi}| \right).
\end{split}
\end{equation}
Combining \eqref{VIxi}--\eqref{VNxi}, we have \eqref{Vxi}.

Note that, by \eqref{ra1} and \eqref{ra3}, $\bar{v}^R$, $\bar{\phi}^R$, $\bar{v}^S$, and $\bar{\phi}^S$ satisfy
\begin{equation} \label{RSt}
\begin{array}{ll}
\bar{v}^R_t = \sigma \bar{v}^R_\xi + \bar{u}^R_\xi, & \bar{\phi}^R_t = (-\ln{\bar{v}^R})_t = - \frac{\bar{v}^R_t}{\bar{v}^R}, \\
\bar{v}^S_t = -\dot{X}(t)\bar{v}^S_\xi, & \bar{\phi}^S_t = -\dot{X}(t)\bar{\phi}^S_\xi.
\end{array}
\end{equation}
Using this, a similar computation in the estimate on $\mathcal{V}^I_\xi$ yields
\begin{equation*}
|\mathcal{V}^I_t| \leq C \left( |\bar{v}^S_\xi(\bar{v}^R-v_m)| + |\bar{v}^R_\xi(\bar{v}^S-v_m)| + |\bar{v}^S_\xi\bar{v}^R_\xi| + |\bar{v}^R_\xi|^2 + |\bar{v}^R_{\xi\xi\xi}|  \right).
\end{equation*}
We also note that, by \eqref{NSP1''} and \eqref{CW1}, the perturbation $\tilde{v}$ satisfies
\begin{equation*}
\tilde{v}_t = \sigma \tilde{v}_\xi + \tilde{u}_\xi + \dot{X}(t) \bar{v}^S_\xi.
\end{equation*}
Using this and \eqref{RSt}, one can find
\begin{equation*}
\begin{split}
|\mathcal{V}^L_t| & \leq C |\bar{v}_\xi| \left( |\tilde{\phi}| + |\tilde{\phi}_\xi| + |\tilde{v}| + |\tilde{v}_\xi| + | \tilde{\phi}_t | + |\tilde{\phi}_{\xi t}| + \lvert \tilde{v}_t | + |\tilde{v}_{\xi t} | \right) \\
& \leq C |\bar{v}_\xi| \left( |\tilde{v}| + |\tilde{\phi}| \right) + C |\bar{v}_\xi| \left( |\tilde{v}_\xi| + |\tilde{v}_{\xi\xi}| + |\tilde{u}_\xi| + |\tilde{u}_{\xi\xi}| + |\tilde{\phi}_\xi| + |\tilde{\phi}_t| + |\dot{X}(t)| |\bar{v}^S_\xi| \right)
\end{split}
\end{equation*}
and
\begin{equation*}
\begin{split}
|\mathcal{V}^N_t| & \leq C |\bar{v}_\xi| \left( |\tilde{v}| + |\tilde{\phi}| \right)  + C \left( \lVert \tilde{v} \rVert_{W^{1,\infty}} + \lVert \tilde{\phi} \rVert_{W^{2,\infty}} \right) \\
& \qquad \times \left( |\tilde{v}_\xi| + |\tilde{\phi}_\xi| + |\tilde{\phi}_{\xi\xi}|+ |\tilde{\phi}_t| + |\tilde{\phi}_{\xi t}| + |\tilde{\phi}_{\xi\xi t}| + |\tilde{v}_t| + |\tilde{v}_{\xi t}| \right) \\
& \leq C |\bar{v}_\xi| \left( |\tilde{v}| + |\tilde{\phi}| \right) + C \left( \delta_R + \varepsilon_1 \right) \\
& \qquad \times \left( |\tilde{v}_\xi| + |\tilde{v}_{\xi\xi}| + |\tilde{u}_\xi| + |\tilde{u}_{\xi\xi}| + |\tilde{\phi}_\xi| + |\tilde{\phi}_{\xi\xi}|+ |\tilde{\phi}_t| + |\tilde{\phi}_{\xi t}| + |\tilde{\phi}_{\xi\xi t}| + |\dot{X}(t)| |\bar{v}^S_\xi| \right).
\end{split}
\end{equation*}
which completes the proof.

\end{proof}

\subsection{Main estimates}

By the definition of $X(t)$, \eqref{XODE}, we find
\begin{equation} \label{XODEY}
\dot{X}(t) = - \frac{M}{\delta_S} (Y_1 + Y_2).
\end{equation}
Combining \eqref{max}, \eqref{quad}, and \eqref{XODEY}, we decompose the inequality for relative functional as
\begin{equation*}
\begin{split}
\frac{d}{dt} \int_\mathbb{R} a^{X} \eta(W|\bar{W}) \, d\xi & \leq \underbrace{ - \frac{\delta_S}{2M} \lvert \dot{X} \rvert^2 + \mathcal{B}_1 + \mathcal{B}_2 - \mathcal{G}_2 - \frac{3}{4} \mathcal{D}_z }_{=: \mathcal{R}_1} \\
& \quad \underbrace{ - \frac{\delta_S}{2M} \lvert \dot{X} \rvert^2 + \dot{X} \sum_{j=3}^6 Y_j + \sum_{j=3}^5 \mathcal{B}_j + \sum_{j=1}^2 S_j - \mathcal{G}_1 - \mathcal{G}^R - \frac{1}{4} \mathcal{D} }_{=: \mathcal{R}_2} \\
& \quad  \underbrace{ + \dot{X} \sum_{j=7}^9 Y_j + \sum_{j=1}^5 \mathcal{P}_j + \sum_{j=1}^6 \tilde{P}_j - \sum_{j=3}^4 \mathcal{G}_j }_{=: \mathcal{R}_3}.
\end{split}
\end{equation*}
We now estimate the terms $\mathcal{R}_1$, $\mathcal{R}_2$, and $\mathcal{R}_3$.

\subsubsection{Estimates of $\mathcal{R}_1$ and $\mathcal{R}_2$} \label{Sec_4.6.1}

First, we estimate $\mathcal{R}_1 + \mathcal{R}_2$.

\begin{lemma} [\cite{KVW2}] \label{R1R2est}
Under the assumptions of Proposition~\ref{Apriori}, there exist constants $C>0$ and $C_1>0$ such that
\begin{equation*}
\begin{split}
 \mathcal{R}_1 + \mathcal{R}_2 & \leq - \frac{\delta_S}{4M} |\dot{X}|^2 - C_1 \left( G_1 + G^S + G^R + D \right) + C \delta_S^2 G_2 \\
& \quad + C \varepsilon_1 \delta_S^{4/3} \delta_R^{4/3}e^{-C\delta_S t} + C \varepsilon_1^{2/3} \lVert \bar{v}^R_{\xi\xi} \rVert_{L^1}^{4/3} + C \varepsilon_1 \lVert \bar{v}^R_\xi \rVert_{L^4}^2 \\
& \quad + C (\delta_R + \varepsilon_1) \left( \lVert \bar{v}^S_\xi (\bar{v}^R-v_m) \rVert_{L^2} + \lVert \bar{v}^R_\xi (\bar{v}^S-v_m) \rVert_{L^2} + \lVert \bar{v}^S_\xi \bar{v}^R_\xi \rVert_{L^2} \right)
\end{split}
\end{equation*}
for all $t \in [0,T]$, where $G_1$, $G_2$, $G^S$, $G^R$, and $D$ are as in Lemma~\ref{lem_0th}.
\end{lemma}

\begin{proof}
The terms $\mathcal{R}_1$ and $\mathcal{R}_2$ correspond to the Navier--Stokes (NS) part of the NSP system. They have the same form as the terms treated in \cite[Section~4.5]{KVW2}, except that $\mathcal{D}$ is replaced by $\mathcal{D}_z$ in $\mathcal{R}_1$. In particular, the estimate of the leading-order terms associated with the shock profile in $\mathcal{R}_1$ requires a delicate analysis, where the sharp Poincar\'e-type inequality, combined with the shift function $X(t)$, plays a crucial role. Since the dissipation enters this estimate, we prove the corresponding leading-order estimate with $\mathcal{D}_z$ below. The remaining computations are the same as in the NS case and are omitted.

Recall that $\delta_S=v_+-v_m$, and set
\begin{equation} \label{dpsmam}
\delta_p := \tilde{p}(v_m)-\tilde{p}(v_+) = \frac{2\delta_S}{v_mv_+}=\sigma^2\delta_S, \quad \sigma_m := \sqrt{-\tilde{p}'(v_m)}, \quad \alpha_m := \frac{1}{\sigma_m\tilde{p}(v_m)}.
\end{equation}


We fix $t\in[0,T]$ and introduce the shock-profile coordinate
\begin{equation} \label{ywdef}
y := \frac{\tilde{p}(v_m)-\tilde{p}(\bar{v}^S)}{\delta_p}, \quad w := \tilde{p}(v)-\tilde{p}(\bar{v}).
\end{equation}
Since $dy/d\xi=-\tilde{p}'(\bar{v}^S)\bar{v}^S_\xi/\delta_p>0$, the change of variables $\xi\in\mathbb{R}\mapsto y\in(0,1)$ is well defined. We first write the term $\mathcal{D}_z$ in this profile coordinate:
\begin{equation} \label{Dzy1}
\mathcal{D}_z = \frac12 \int_0^1 a^X v \frac{dy}{d\xi} z_y^2 \, dy = \int_0^1 \frac{v\bar{v}^S_\xi}{\delta_p(\bar{v}^S)^2} z_y^2 \, dy + \int_0^1 \frac{y}{\sqrt{\delta_S}}\frac{v\bar{v}^S_\xi}{(\bar{v}^S)^2} z_y^2 \, dy \geq \int_0^1 \frac{v\bar{v}^S_\xi}{\delta_p(\bar{v}^S)^2} z_y^2 \, dy.
\end{equation}
Here, the second equality follows from the definitions of $\tilde{p}$ and $a^X$, and the last inequality follows from the nonnegativity of the second integral. Note that, by the bounds $\bar{v}^R \geq v_m - \delta_R$, $\bar{v}^S \geq v_m$, and $\|\tilde{v}\|_{L^\infty} \leq C \varepsilon_1$,
\begin{equation*}
v=\tilde{v}+\bar{v} \geq v_m-C(\delta_R+\varepsilon_1)\geq v_m\left(1-C(\delta_R+\varepsilon_1)\right).
\end{equation*}
Hence, \eqref{Dzy1} yields
\begin{equation} \label{Dzy2}
\mathcal{D}_z \geq \left(1-C(\delta_R+\varepsilon_1)\right) \int_0^1\frac{v_m\bar{v}^S_\xi}{\delta_p(\bar{v}^S)^2}z_y^2\,dy.
\end{equation}
We then use the shock ODE \eqref{SHODE} to rewrite $\bar{v}^S_\xi$. Integrating the second equation of \eqref{SHODE}, with \eqref{trf}, and using the first equation of \eqref{SHODE}, \eqref{RH}, and \eqref{ywdef}, we obtain
\begin{equation} \label{vSxi}
\bar{v}^S_\xi = \sigma(\bar{v}^S-v_m)(v_+-\bar{v}^S)+\frac{\bar{v}^S}{\sigma}\Phi(\bar{v}^S,\bar{\phi}^S) = \frac{\sigma \delta_S^2 (\bar{v}^S)^2}{v_m v_+} y (1-y) +\frac{\bar{v}^S}{\sigma}\Phi(\bar{v}^S,\bar{\phi}^S).
\end{equation}
Substituting this into the last integral of \eqref{Dzy2} and using \eqref{dpsmam} and $\sigma < \sigma_m$, we have
\begin{equation} \label{Dzy3}
\mathcal{D}_z \geq \alpha_m \delta_p \left(1 - C (\delta_R +\varepsilon_1) \right) \int_0^1 y (1-y) z_y^2 \, dy + \left(1 - C (\delta_R +\varepsilon_1) \right) \int_0^1 \frac{v_m}{\sigma \delta_p} \frac{\Phi(\bar{v}^S,\bar{\phi}^S)}{\bar{v}^S} z_y^2 \, dy.
\end{equation}

We now estimate the second term on the right-hand side of \eqref{Dzy3} in terms of the first term. To estimate $\Phi(\bar{v}^S,\bar{\phi}^S)$, we have
\begin{equation} \label{relativederiv}
|\bar{v}^S_{\xi\xi}|+|\bar{\phi}^S_{\xi\xi}| \leq C\delta_S\bar{v}^S_\xi,
\end{equation}
which follows from the center-manifold argument in \cite[Section~2]{DLZ}. By \eqref{Phi}, \eqref{shderiv1}, and \eqref{relativederiv}, we have
\begin{equation*}
|\Phi(\bar{v}^S,\bar{\phi}^S)|\leq C\delta_S\bar{v}^S_\xi.
\end{equation*}
Using this and \eqref{vSxi}, and absorbing the resulting $\Phi$-term for sufficiently small $\delta_S>0$, we obtain
\begin{equation*}
|\Phi(\bar{v}^S,\bar{\phi}^S)|\leq C\delta_S^3 y(1-y),
\end{equation*}
and hence
\begin{equation*}
\left|\frac{v_m}{\sigma\delta_p}\frac{\Phi(\bar{v}^S,\bar{\phi}^S)}{\bar{v}^S}\right|\leq C\delta_S\alpha_m\delta_p y(1-y).
\end{equation*}
Therefore, \eqref{Dzy3} gives
\begin{equation} \label{Dzchange}
\mathcal{D}_z \geq \alpha_m\delta_p\left(1-C(\delta_S+\delta_R+\varepsilon_1)\right)\int_0^1 y(1-y)z_y^2\,dy.
\end{equation}

Next, we apply the Poincar\'e-type inequality in \cite[Lemma~1.1]{KVW2} to $z$:
\begin{equation} \label{poincare}
\int_0^1 \left| z - \int_0^1 z\,dy\right|^2 \, dy \leq \frac12\int_0^1 y(1-y)z_y^2\,dy.
\end{equation}
To write the left-hand side in terms of $w$, we first estimate $w-z$. By \eqref{zdef}, \eqref{ywdef}, and \eqref{shderiv1}, we have
\begin{equation} \label{wzerror}
\delta_p\int_0^1|w-z|^2\,dy = 4\int_\mathbb{R}\frac{|\tilde{p}'(\bar{v}^S)|\bar{v}^S_\xi}{\bar{v}^4}\tilde{\phi}_{\xi\xi}^2\,d\xi \leq C\delta_S^2G_2.
\end{equation}
By Young's inequality,
\begin{equation*}
\begin{split}
\int_0^1 \left| w - \int_0^1 w \, dy \right|^2 \, dy &= \int_0^1 \bigg[ \bigg( z - \int_0^1 z \, dy \bigg) + \bigg( (w-z) - \int_0^1 (w-z) \, dy \bigg) \bigg]^2 \, dy \\
& \leq \frac98 \int_0^1 \left| z - \int_0^1z\,dy\right|^2 \, dy + 9 \int_0^1 \bigg| (w-z) - \int_0^1 (w-z) \, dy \bigg|^2 \,dy \\
& \leq \frac98 \int_0^1 \left| z - \int_0^1z\,dy\right|^2 \, dy + 9\int_0^1|w-z|^2\,dy.
\end{split}
\end{equation*}
Combining this with \eqref{Dzchange}--\eqref{wzerror}, we obtain
\begin{equation} \label{Dsharp}
\frac34\mathcal{D}_z \geq \frac54\alpha_m\delta_p\left(\int_0^1w^2\,dy-\left|\int_0^1w\,dy\right|^2\right)-C\delta_S^2G_2
\end{equation}
for sufficiently small $\delta_0>0$ and $\varepsilon_1>0$.

We next verify that the shift term can be treated as in the NS case. By \eqref{vup'}, \eqref{shderiv1}, and \eqref{relativederiv}, we have
\begin{equation*}
\left|\bar{h}^S_\xi-\frac{\tilde{p}'(\bar{v}^S)}{\sigma}\bar{v}^S_\xi\right|\leq C\delta_S\bar{v}^S_\xi.
\end{equation*}
Therefore, using the same argument as in \cite[Section~4.5.1]{KVW2}, we obtain
\begin{equation*}
\left|Y_1+Y_2+\frac{2\delta_p}{\sigma_m^2}\int_0^1 w\,dy\right|\leq C\delta_p\left(\sqrt{\delta_S}+\delta_R+\varepsilon_1\right)\int_0^1|w|\,dy.
\end{equation*}
Since $M=2\alpha_m\sigma_m^4\delta_S/\delta_p$, it follows from \eqref{XODEY} and Young's inequality that
\begin{equation*}
-\frac{\delta_S}{2M}|\dot{X}|^2\leq -2\alpha_m\delta_p\left|\int_0^1w\,dy\right|^2+C\delta_p\left(\sqrt{\delta_S}+\delta_R+\varepsilon_1\right)^2\int_0^1w^2\,dy.
\end{equation*}
Hence, the positive term involving $\textstyle \left|\int_0^1w\,dy\right|^2$ resulting from \eqref{Dsharp} is controlled by the shift term, while the remaining leading-order terms are treated as in \cite[Section~4.5.1]{KVW2}. This yields
\begin{equation*}
\mathcal{R}_1 \leq -cG^S + C\delta_S^2G_2 + C\int_{\mathbb{R}}a^X_\xi|w|^3\,d\xi + C\int_{\mathbb{R}}a^X_\xi|\bar{v}^R-v_m|w^2\,d\xi
\end{equation*}
for some constant $c>0$. The last two terms, together with the terms in $\mathcal{R}_2$, are estimated by the arguments in \cite[Section~4.5.2]{KVW2}. Combining these estimates, we obtain the desired estimate.
\end{proof}

\subsubsection{Estimate of $\mathcal{R}_3$}

We next estimate the term $\mathcal{R}_3$.

\begin{lemma} \label{R3est}
Under the assumptions of Proposition~\ref{Apriori}, there exist constants $C>0$ and $C_2>0$ such that
\begin{equation} \label{R3ineq}
\begin{split}
\mathcal{R}_3 & \leq - C_2 \left( G_2 + G_3 \right) + C (\sqrt{\delta_0} + \varepsilon_1) \left( \delta_S \lvert \dot{X} \rvert^2 + G_1 + G^S + G^R + D \right) \\
& \quad + C (\sqrt{\delta_0} + \varepsilon_1) \int_\mathbb{R} \left( \tilde{v}_{\xi\xi}^2 + \tilde{u}_\xi^2 + \bar{v}_\xi \tilde{\phi}^2 + \tilde{\phi}_{\xi}^2 + \tilde{\phi}_{\xi t}^2 \right) \, d\xi + C (\delta_R^{1/3} + \varepsilon_1^{1/3}) \lVert \bar{v}^R_{\xi\xi\xi} \rVert_{L^1}^{4/3} \\
& \quad + C ( \delta_R + \varepsilon_1 ) \left( \lVert \bar{v}^S_\xi ( \bar{v}^R - v_m ) \rVert_{L^2} + \lVert  \bar{v}^R_\xi \bar{v}^S_\xi \rVert_{L^2} + \lVert \bar{v}^R_\xi \rVert_{L^4}^2 \right) \\
& \quad + C \delta_R^{-1/2} \left( \lVert \bar{v}^S_\xi(\bar{v}^R-v_m) \rVert_{L^2}^2 + \lVert \bar{v}^R_\xi (\bar{v}^S-v_m)\rVert_{L^2}^2 + \lVert \bar{v}^R_\xi \bar{v}^S_\xi \rVert_{L^2}^2 + \lVert \bar{v}^R_\xi \rVert_{L^4}^4 + \lVert \bar{v}^R_{\xi\xi} \rVert_{L^2}^2 + \lVert \bar{v}^R_{\xi\xi\xi} \rVert_{L^2}^2 \right)
\end{split}
\end{equation}
for all $t \in [0,T]$, where $G_1$, $G_2$, $G_3$, $G^S$, $G^R$, and $D$ are as in Lemma~\ref{lem_0th}.
\end{lemma}

\begin{proof}

We estimate each term in turn.

\noindent $\bullet$ Estimates of the terms $\dot{X}Y_j$ for $j=7,8,9$: By young's inequality and \eqref{Xbd}, we obtain
\begin{equation*}
\begin{split}
\lvert \dot{X} Y_7 \rvert & \leq C \lvert \dot{X} \rvert \int_\mathbb{R} \lvert a^{X}_\xi \rvert  \lvert \tilde{v} \rvert \lvert \tilde{\phi}_{\xi\xi} \rvert \, d\xi \\
& \leq C \lVert \tilde{v} \rVert_{L^\infty} \left( \int_\mathbb{R} \lvert \tilde{\phi}_{\xi\xi} \rvert^2 \, d\xi + \int_\mathbb{R} \lvert a^{X}_\xi \rvert^2 \lvert \tilde{v} \rvert^2 \, d\xi \right) \\
& \leq C \varepsilon_1 \int_\mathbb{R} \left( \bar{v}^S_\xi \tilde{v}^2 + \tilde{\phi}_{\xi\xi}^2 \right) \, d\xi,
\end{split}
\end{equation*}
where we used
\begin{equation*}
\lvert a^{X}_\xi \rvert^2 \leq C \delta_S^{-1} \lvert \bar{v}^S_\xi \rvert^2 \leq C \delta_S \lvert \bar{v}^S_\xi \rvert.
\end{equation*}
Similarly, we estimate
\begin{equation*}
\begin{split}
\lvert \dot{X} Y_8 \rvert & \leq C \lvert \dot{X} \rvert \int_\mathbb{R} \lvert a^{X}_\xi \rvert \left( \lvert \tilde{\phi}_{\xi\xi} \rvert^2 +  \lvert \tilde{\phi}_{\xi} \rvert^2 \right) \, d\xi \leq C \varepsilon_1 \int_\mathbb{R} \left( \tilde{\phi}_{\xi\xi}^2 + \tilde{\phi}_{\xi}^2 \right) \, d\xi.
\end{split}
\end{equation*}
To estimate $\dot{X} Y_9$, we apply H\"older's inequality and the bound \eqref{shderiv1}:
\begin{equation*}
\begin{split}
\lvert \dot{X} Y_9 \rvert & \leq C \lvert \dot{X} \rvert \int_\mathbb{R} \lvert \bar{v}^S_\xi \rvert \lvert \tilde{\phi}_{\xi\xi} \rvert \, d\xi \\
& \leq C \lvert \dot{X} \rvert \sqrt{\int_\mathbb{R} \bar{v}^S_\xi \, d\xi} \sqrt{\int_\mathbb{R} \bar{v}^S_\xi \tilde{\phi}_{\xi\xi}^2 \, d\xi} \\
& \leq C \delta_S^2 \lvert \dot{X} \rvert^2 + \frac{C}{\delta_S^2} \left( \int_\mathbb{R} \bar{v}^S_\xi \, d\xi \right) \left( \int_\mathbb{R} \bar{v}^S_\xi \tilde{\phi}_{\xi\xi}^2 \, d\xi \right) \\
& \leq C \delta_S \left( \delta_S \lvert \dot{X} \rvert^2 + \int_\mathbb{R} \tilde{\phi}_{\xi\xi}^2 \, d\xi \right).
\end{split}
\end{equation*}

\noindent $\bullet$ Estimate of the term $\mathcal{P}_1$: From the definitions of $F_3$ and $\Phi$ in \eqref{F3_def} and \eqref{Phi}, it follows that
\begin{equation*}
F_3 = - \bigg( \frac{1}{2} \bigg( \frac{\bar{\phi}^S_\xi}{\bar{v}^S} \bigg)^2 - \frac{1}{2} \left( \frac{\bar{\phi}_\xi}{ \bar{v} } \right)^2  \bigg)_\xi + \bigg( \frac{1}{\bar{v}^S} \bigg( \frac{\bar{\phi}^S_\xi}{\bar{v}^S} \bigg)_\xi -  \frac{1}{\bar{v}} \bigg( \frac{\bar{\phi}_\xi}{\bar{v}} \bigg)_\xi  \bigg)_\xi =: F_{31} + F_{32}.
\end{equation*}
In analogy with the treatment of $\mathcal{V}^I$ and its derivatives in the proof of Lemma~\ref{V_est}, we obtain the bounds on $F_{31}$ and $F_{32}$:
\begin{equation*}
\begin{split}
|F_{31}| & = \bigg| \frac{\bar{\phi}^S_\xi \bar{\phi}^S_{\xi\xi}}{(\bar{v}^S)^2} - \frac{\bar{\phi}_\xi \bar{\phi}_{\xi\xi}}{\bar{v}^2} - \bigg( \frac{\bar{v}^S_\xi (\bar{\phi}^S_\xi)^2}{(\bar{v}^S)^3} - \frac{\bar{v}_\xi \bar{\phi}_\xi^2}{\bar{v}^3} \bigg) \bigg| \\
& \leq C \left( |\bar{v}^S_\xi(\bar{v}^R-v_m)| + |\bar{v}^S_\xi\bar{v}^R_\xi| + |\bar{v}^R_\xi|^2 \right)
\end{split}
\end{equation*}
and
\begin{equation*}
\begin{split}
|F_{32}| & = - \frac{\bar{v}^S_\xi \bar{\phi}^S_{\xi\xi}}{(\bar{v}^S)^3} + \frac{(\bar{v}^S_\xi)^2 \bar{\phi}^S_\xi}{(\bar{v}^S)^4} + \frac{\bar{\phi}^S_{\xi\xi\xi}}{(\bar{v}^S)^2} + \frac{\bar{v}^S_{\xi\xi}\bar{\phi}^S_\xi}{(\bar{v}^S)^3} - \frac{2(\bar{v}^S_\xi)^2 \bar{\phi}^S_\xi}{(\bar{v}^S)^4} \\
& \quad - \bigg( - \frac{\bar{v}_\xi \bar{\phi}_{\xi\xi}}{\bar{v}^3} + \frac{\bar{v}_\xi^2 \bar{\phi}_\xi}{\bar{v}^4} + \frac{\bar{\phi}_{\xi\xi\xi}}{\bar{v}^2} + \frac{\bar{v}_{\xi\xi}\bar{\phi}_\xi}{\bar{v}^3} - \frac{2\bar{v}_\xi^2 \bar{\phi}_\xi}{\bar{v}^4} \bigg) \\
& \leq C \left( |\bar{v}^S_\xi(\bar{v}^R-v_m)| + |\bar{v}^S_\xi\bar{v}^R_\xi| + |\bar{v}^R_\xi|^2 + |\bar{v}^R_{\xi\xi\xi}| \right),
\end{split}
\end{equation*}
where we have used the bounds \eqref{shderiv1} and \eqref{raderiv}. Substituting these bounds into $\mathcal{P}_1$, we have
\begin{equation*}
\begin{split}
| \mathcal{P}_1 | & \leq C \left| \int_\mathbb{R} a^{X} \tilde{h} F_3 \, d\xi \right| \\
& \leq C \int_\mathbb{R} |\tilde{h}| \left( \lvert \bar{v}^S_\xi \rvert \lvert \bar{v}^R - v_m \rvert + \lvert \bar{v}^S_\xi  \rvert \lvert \bar{v}^R_\xi \rvert + \lvert \bar{v}^R_\xi \rvert^2 \right) \, d\xi + C \int_\mathbb{R} |\tilde{h}| |\bar{v}^R_{\xi\xi\xi}| \, d\xi =: J_1 + J_2.
\end{split}
\end{equation*}
Applying H\"older's inequality, we obtain
\begin{equation*}
\begin{split}
J_1 & \leq C \lVert \tilde{h} \rVert_{L^2} \lVert \left( \lvert \bar{v}^S_\xi \rvert \lvert \bar{v}^R - v_m \rvert + \lvert \bar{v}^S_\xi  \rvert \left\lvert \bar{v}^R_\xi \rvert + \lvert \bar{v}^R_\xi \rvert^2 \right) \right\rVert_{L^2} \\
& \leq C (\delta_R + \varepsilon_1) \left( \left\lVert \bar{v}^S_\xi ( \bar{v}^R - v_m ) \right\rVert_{L^2} + \left\lVert  \bar{v}^R_\xi \bar{v}^S_\xi \right\rVert_{L^2} + \left\lVert \bar{v}^R_\xi \right\rVert_{L^4}^2 \right).
\end{split}
\end{equation*}
In the second inequality, we have used $\lVert \tilde{h} \rVert_{L^2} \leq C (\delta_R + \varepsilon_1)$ (see the proof of Lemma~\ref{lem_inf}). For $J_2$, we use the interpolation inequality and Young's inequality.
\begin{equation*}
\begin{split}
J_2 & \leq C \lVert \tilde{h} \rVert_{L^\infty} \lVert \bar{v}^R_{\xi\xi\xi} \rVert_{L^1} \leq C \lVert \tilde{h} \rVert_{L^2}^{1/2} \lVert \tilde{h}_\xi \rVert_{L^2}^{1/2} \lVert \bar{v}^R_{\xi\xi\xi} \rVert_{L^1} \\
& \leq C \lVert \tilde{h} \rVert_{L^2} \lVert \tilde{h}_\xi \rVert_{L^2}^2 + C \lVert \tilde{h} \rVert_{L^2}^{1/3} \lVert \bar{v}^R_{\xi\xi\xi} \rVert_{L^1}^{4/3} \\
& \leq C \left( \delta_R + \varepsilon_1 \right) \lVert \tilde{h}_\xi \rVert_{L^2}^2 + C ( \delta_R^{1/3} + \varepsilon_1^{1/3} ) \lVert \bar{v}^R_{\xi\xi\xi} \rVert_{L^1}^{4/3}.
\end{split}
\end{equation*}

\noindent $\bullet$ Estimate of the term $\mathcal{P}_2$: We decompose $\mathcal{P}_2$ as
\begin{equation*}
\mathcal{P}_2 = \sum_{j=1}^5 \mathcal{P}_{2j},
\end{equation*}
where
\begin{equation*}
\begin{split}
\mathcal{P}_{21}(W) & := \int_\mathbb{R} a^{X}_\xi \frac{\tilde{h} \tilde{\phi}_{\xi\xi}}{\bar{v}^2} \, d\xi, \quad \mathcal{P}_{22}(W) := - \int_\mathbb{R} a^{X} \frac{2\bar{\phi}_{\xi\xi} \tilde{h}_\xi \tilde{v}}{v^2\bar{v}} \, d\xi, \\
\mathcal{P}_{23}(W) & := - \int_\mathbb{R}  a^{X}_\xi \tilde{h} \bigg( \frac{2\bar{\phi}_{\xi\xi} \tilde{v}}{v^2\bar{v}} + \frac{\bar{\phi}_\xi \tilde{\phi}_\xi}{v^2} - \frac{(\bar{\phi}_\xi)^2 \tilde{v}}{v^2\bar{v}} + \frac{\bar{v}_\xi \tilde{\phi}_\xi}{v^3} + \frac{\bar{\phi}_\xi \tilde{v}_\xi}{v^3} - \frac{3\bar{v}_\xi \bar{\phi}_\xi \tilde{v}}{v^3 \bar{v}}  \bigg) \, d\xi, \\
\mathcal{P}_{24}(W) & := - \int_\mathbb{R} a^{X} \tilde{h}_\xi \bigg( \frac{\bar{\phi}_\xi \tilde{\phi}_\xi}{v^2} - \frac{(\bar{\phi}_\xi)^2 \tilde{v}}{v^2\bar{v}} + \frac{\bar{v}_\xi \tilde{\phi}_\xi}{v^3} + \frac{\bar{\phi}_\xi \tilde{v}_\xi}{v^3} - \frac{3\bar{v}_\xi \bar{\phi}_\xi \tilde{v}}{v^3 \bar{v}}  \bigg) \, d\xi,
\end{split}
\end{equation*}
and
\begin{equation*}
\begin{split}
\mathcal{P}_{25}(W) & := - \frac{1}{2} \int_\mathbb{R} \left( a^{X}_\xi \tilde{h} + a^{X} \tilde{h}_\xi \right) \bigg( \frac{\tilde{\phi}_\xi^2}{v^2} - \frac{\bar{\phi}_\xi^2 \tilde{v}^2}{v^2 \bar{v}^2} \bigg) \, d\xi \\
& \quad - \int_\mathbb{R} \left( a^{X}_\xi \tilde{h} + a^{X} \tilde{h}_\xi \right) \bigg( \frac{\bar{\phi}_{\xi\xi} \tilde{v}^2}{v^2 \bar{v}^2} + \frac{ \tilde{v}^2 \tilde{\phi}_{\xi\xi}}{v^2 \bar{v}^2} + \frac{2 \tilde{v} \tilde{\phi}_{\xi\xi}}{v^2 \bar{v}} + \frac{\tilde{v}_\xi \tilde{\phi}_\xi}{v^3} - \frac{\bar{v}_\xi \bar{\phi}_\xi \tilde{v}^3}{v^3 \bar{v}^3} - \frac{3\bar{v}_\xi \bar{\phi}_\xi \tilde{v}^2}{v^3 \bar{v}^2} \bigg) \, d\xi.
\end{split}
\end{equation*}
Applying Young's inequality and the bound $a^X_\xi \leq C \delta_S^{-1/2} \bar{v}^S_\xi$, we obtain
\begin{equation*}
\begin{split}
\lvert \mathcal{P}_{21} \rvert & \leq C \left( \int_\mathbb{R} \lvert a^X_\xi \rvert^{1/3} \lvert \tilde{\phi}_{\xi\xi} \rvert^2 \, d\xi + \int_\mathbb{R} \lvert a^X_\xi \rvert^{5/3} \lvert \tilde{h} \rvert^2 \, d\xi \right) \\
& \leq C \delta_S^{1/2} \int_\mathbb{R} \left( \bar{v}^S_\xi \tilde{h}^2 + \tilde{\phi}_{\xi\xi}^2 \right) \, d\xi.
\end{split}
\end{equation*}
Similarly, using \eqref{shderiv1} and \eqref{raderiv}, we obtain the estimates for $\mathcal{P}_{22}$, $\mathcal{P}_{23}$, and $\mathcal{P}_{24}$:
\begin{equation*}
\begin{split}
\lvert \mathcal{P}_{22} \rvert & \leq C \int_\mathbb{R} \lvert \bar{v}_\xi | \lvert \tilde{h}_\xi \rvert \lvert \tilde{v} \rvert \, d\xi \leq C \int_\mathbb{R} |\bar{v}_\xi|^{1/2} |\tilde{h}_\xi|^2 \, d\xi + C \int_{\mathbb{R}} |\bar{v}_\xi|^{3/2} | \tilde{v}|^2 \, d\xi \\
& \leq C ( \delta_S + \delta_R^{1/2} ) \int_\mathbb{R} \left( \bar{v}_\xi \tilde{v}^2 + \tilde{h}_\xi^2 \right) \, d\xi,\\
|\mathcal{P}_{23}| & \leq C \int_\mathbb{R} |a^X_\xi| |\tilde{h}| |\bar{v}_\xi| |\tilde{v}| \, d\xi + C \int_\mathbb{R} |\bar{v}_\xi| |a^X_\xi| |\tilde{h}| \left( |\tilde{v}_\xi| + |\tilde{\phi}_\xi| \right) \, d\xi \\
& \leq C \int_\mathbb{R} |\bar{v}_\xi|^{1/2} |a^X_\xi|^2 |\tilde{h}|^2 \, d\xi + C \int_\mathbb{R} |\bar{v}_\xi|^{3/2}|\tilde{v}|^2 \, d\xi + C \int_\mathbb{R} |\bar{v}_\xi| \left( |a^X_\xi|^2 |\tilde{h}|^2 + |\tilde{v}_\xi|^2 + |\tilde{\phi}_\xi|^2 \right) \, d\xi \\
& \leq C (\delta_S + \delta_R^{1/2}) \int_\mathbb{R} \left( \bar{v}^S_\xi \tilde{h}^2 + \bar{v}_\xi \tilde{v}^2 + \tilde{v}_\xi^2 + \tilde{\phi}_\xi^2 \right) \, d\xi,
\end{split}
\end{equation*}
and
\begin{equation*}
\begin{split}
|\mathcal{P}_{24}| & \leq C \int_\mathbb{R} |\bar{v}_\xi| |\tilde{h}_\xi| \left( |\bar{v}_\xi| |\tilde{v}| + |\tilde{v}_\xi| + |\tilde{\phi}_\xi| \right) \, d\xi  \leq C (\delta_S^2 + \delta_R) \int_\mathbb{R} \left( \bar{v}_\xi \tilde{v}^2 + \tilde{v}_\xi^2 + \tilde{h}_\xi^2 + \tilde{\phi}_\xi^2 \right) \, d\xi.
\end{split}
\end{equation*}
For the nonlinear term $\mathcal{P}_{25}$, we use $|a^X_\xi|^2 \leq C |\bar{v}^S_\xi| \leq C |\bar{v}_\xi|$ and \eqref{infty} to estimate
\begin{equation*}
\begin{split}
| \mathcal{P}_{25}| & \leq \int_\mathbb{R} |\tilde{v}| |\tilde{h}_\xi| \left( |\bar{v}_\xi| |\tilde{v}| + |\bar{v}_\xi| |\tilde{v}|^2 + |\tilde{\phi}_{\xi\xi}| + |\tilde{v}| |\tilde{\phi}_{\xi\xi}| \right) \, d\xi + \int_\mathbb{R} |\tilde{\phi}_{\xi}| |\tilde{h}_\xi| \left( |\tilde{\phi}_\xi| + |\tilde{v}_\xi| \right) \, d\xi \\
& \quad + \int_\mathbb{R} |\tilde{h}| \left( |\bar{v}_\xi| |\tilde{v}|^2 + |\bar{v}_\xi| |\tilde{v}|^3 + |\tilde{\phi}_\xi|^2 + |\tilde{\phi}_\xi| |\tilde{v}_\xi| \right) \, d\xi + \int_\mathbb{R} |a^X_\xi| |\tilde{h}| |\tilde{v}| \left( |\tilde{\phi}_{\xi\xi}| + |\tilde{v}| |\tilde{\phi}_{\xi\xi}| \right) \, d\xi \\
& \leq C \lVert \tilde{v} \rVert_{L^\infty} \int_\mathbb{R} \left( \bar{v}_\xi \tilde{v}^2 + \tilde{h}_\xi^2 + \tilde{\phi}_{\xi\xi}^2 \right) \, d\xi + C \lVert \tilde{\phi}_\xi \rVert_{L^\infty} \int_\mathbb{R} \left( \tilde{v}_\xi^2 + \tilde{h}_\xi^2 + \tilde{\phi}_\xi^2 \right) \, d\xi \\
& \quad + C \lVert (\tilde{v},\tilde{h}) \rVert_{L^\infty} \int_\mathbb{R} \left( \bar{v}_\xi \tilde{v}^2 + \tilde{v}_\xi^2 + \tilde{h}_\xi^2 + \tilde{\phi}_\xi^2 \right) \, d\xi + C \lVert (\tilde{v},\tilde{h}) \rVert_{L^\infty} \int_\mathbb{R} \left( \bar{v}_\xi \tilde{v}^2 + \tilde{\phi}_{\xi\xi}^2 \right) \, d\xi \\
& \leq C (\delta_R + \varepsilon_1) \int_\mathbb{R} \left( \bar{v}_\xi \tilde{v}^2 + \tilde{v}_\xi^2 + \tilde{h}_\xi^2 + \tilde{\phi}_\xi^2 + \tilde{\phi}_{\xi\xi}^2 \right) \, d\xi.
\end{split}
\end{equation*}

\noindent $\bullet$ Estimate of the term $\mathcal{P}_3$: Using \eqref{v}, and applying integration by parts, $\mathcal{P}_3$ can be rewritten as
\begin{equation*}
\mathcal{P}_3 = \sum_{j=1}^3 \mathcal{P}_{3j},
\end{equation*}
where
\begin{equation*}
\begin{split}
\mathcal{P}_{31} & := - \sigma \int_\mathbb{R} a^X \frac{1}{\bar{v}^2} \left( \frac{e^{-\bar{\phi}}}{\bar{v}} \right)_\xi \tilde{\phi}_{\xi\xi}^2 \, d\xi + \sigma \int_\mathbb{R} a^X_\xi \bigg( \frac{e^{-\bar{\phi}}\tilde{\phi}_{\xi\xi}^2}{2\bar{v}^3} - \frac{e^{-\bar{\phi}}\tilde{\phi}_\xi^2}{2\bar{v}^2} \bigg) \, d\xi \\
& \quad + \sigma \int_\mathbb{R} a^X \bigg( \bigg( \frac{e^{-\bar{\phi}}}{2\bar{v}^3} \bigg)_\xi \tilde{\phi}_{\xi\xi}^2 - \bigg( \frac{e^{-\bar{\phi}}}{2\bar{v}^2} \bigg)_\xi \tilde{\phi}_{\xi}^2 \bigg) \, d\xi, \\
\mathcal{P}_{32} & := - \sigma \int_\mathbb{R} a^X \frac{e^{-\bar{\phi}}\bar{\phi}_\xi \tilde{\phi}\tilde{\phi}_{\xi\xi}}{\bar{v}^2} \, d\xi, \quad \mathcal{P}_{33} := - \sigma \int_\mathbb{R} a^X \frac{\mathcal{V}_\xi \tilde{\phi}_{\xi\xi}}{\bar{v}^2} \, d\xi.
\end{split}
\end{equation*}
With the bounds \eqref{shderiv1}, \eqref{raderiv}, and $|a^X_\xi| \leq C\delta_S^{3/2}$, it follows that
\begin{equation*}
|\mathcal{P}_{31}| \leq C \int_\mathbb{R} \left( |\bar{v}_\xi| + |a^X_\xi| \right) \left( | \tilde{\phi}_\xi |^2 + |\tilde{\phi}_{\xi\xi} |^2 \right) \, d\xi \leq C \left( \delta_S^{3/2} + \delta_R \right) \int_\mathbb{R} \left( \tilde{\phi}_{\xi}^2 + \tilde{\phi}_{\xi\xi}^2 \right) \, d\xi.
\end{equation*}
The term $\mathcal{P}_{32}$ is estimated by
\begin{equation*}
|\mathcal{P}_{32}| \leq C \int_\mathbb{R} \left( |\bar{v}_\xi|^{1/2} |\tilde{\phi}_{\xi\xi}|^2 + |\bar{v}_\xi|^{3/2}|\tilde{\phi}|^2 \right) \, d\xi \leq C \left( \delta_S + \delta_R^{1/2} \right) \int_\mathbb{R} \left( \bar{v}_\xi \tilde{\phi}^2 + \tilde{\phi}_{\xi\xi}^2 \right) \, d\xi.
\end{equation*}
Thanks to \eqref{Vxi}, we have
\begin{equation} \label{Vxiphi}
\begin{split}
|\mathcal{P}_{33}| & \leq C \int_\mathbb{R} |\bar{v}_\xi| \left( | \tilde{v}| + |\tilde{\phi}| \right) |\tilde{\phi}_{\xi\xi}| \, d\xi \\
& \quad + C \left( \delta_S^2 + \delta_R + \varepsilon_1 \right) \int_\mathbb{R} \left( |\tilde{v}_\xi|  + |\tilde{v}_{\xi\xi}| + |\tilde{\phi}_\xi| + |\tilde{\phi}_{\xi\xi}| + |\tilde{\phi}_{\xi\xi\xi}| \right) |\tilde{\phi}_{\xi\xi}| \, d\xi \\
& \quad + C \int_\mathbb{R} |\tilde{\phi}_{\xi\xi}| \left( | \bar{v}^S_\xi (\bar{v}^R - v_m) | + |\bar{v}^R_\xi (\bar{v}^S - v_m) | + | \bar{v}^R_\xi \bar{v}^S_\xi | + |\bar{v}^R_\xi|^2 + |\bar{v}^R_{\xi\xi\xi}| \right) \, d\xi.
\end{split}
\end{equation} 
Applying Young's inequality to the right-hand side, we obtain
\begin{equation} \label{Vxiphi1}
\begin{split}
R.H.S & \leq C \int_\mathbb{R} |\bar{v}_\xi|^{1/2} |\tilde{\phi}_{\xi\xi}|^2 \, d\xi + C \int_\mathbb{R} |\bar{v}_\xi|^{3/2} \left( |\tilde{v}|^2 + |\tilde{\phi}|^2 \right) \, d\xi \\
& \quad + C ( \delta_S + \delta_R + \varepsilon_1 ) \int_\mathbb{R} \left( \tilde{\phi}_\xi^2 + \tilde{\phi}_{\xi\xi}^2 + \tilde{\phi}_{\xi\xi\xi}^2 + \tilde{v}_\xi^2  + \tilde{v}_{\xi\xi}^2 \right) \, d\xi + C \delta_R^{1/2} \int_\mathbb{R} \tilde{\phi}_{\xi\xi}^2 \, d\xi \\
& \quad + C \delta_R^{-1/2} \int_\mathbb{R} \left( |\bar{v}^S_\xi ( \bar{v}^R - v_m )|^2 + | \bar{v}^R_\xi ( \bar{v}^S - v_m )|^2 + |\bar{v}^R_\xi \bar{v}^S_\xi|^2 + |\bar{v}^R_\xi|^4 + |\bar{v}^R_{\xi\xi\xi}|^2 \right) \, d\xi \\
& \leq C (\delta_S + \delta_R^{1/2} + \varepsilon_1) \int_\mathbb{R} \left( \bar{v}_\xi \tilde{v}^2 + \bar{v}_\xi \tilde{\phi}^2 + \tilde{\phi}_\xi^2 + \tilde{\phi}_{\xi\xi}^2 + \tilde{\phi}_{\xi\xi\xi}^2 + \tilde{v}_\xi^2  + \tilde{v}_{\xi\xi}^2 \right) \, d\xi \\
& \quad + C \delta_R^{-1/2} \left( \lVert \bar{v}^S_\xi (\bar{v}^R - v_m) \rVert_{L^2}^2 + \lVert \bar{v}^R_\xi (\bar{v}^S - v_m) \rVert_{L^2}^2 + \lVert \bar{v}^S_\xi \bar{v}^R_\xi \rVert_{L^2}^2 + \lVert \bar{v}^R_\xi \rVert_{L^4}^4 + \lVert \bar{v}^R_{\xi\xi\xi} \rVert_{L^2}^2 \right).
\end{split}
\end{equation}

\noindent $\bullet$ Estimate of the term $\mathcal{P}_4$: The term $\mathcal{P}_4$ consists of the following four terms:
\begin{equation*}
\begin{array}{ll}
\mathcal{P}_{41} (W) := \int_\mathbb{R} a^{X}_\xi \frac{\tilde{v}\tilde{\phi}_{\xi t}}{\bar{v}^2} \, d\xi, & \mathcal{P}_{42} (W) := \int_\mathbb{R} a^{X} \bigg( \frac{2 (\sigma \bar{v}^R_\xi + \bar{u}^R_\xi) \tilde{v} \tilde{\phi}_{\xi\xi}}{\bar{v}^3} - \frac{2 \bar{v}_\xi \tilde{v} \tilde{\phi}_{\xi t}}{\bar{v}^3}  \bigg) \, d\xi, \\
\mathcal{P}_{43}(W) := - \dot{X}(t) \int_\mathbb{R} a^{X} \frac{2 \bar{v}^S_\xi \tilde{v} \tilde{\phi}_{\xi\xi}}{\bar{v}^3} \, d\xi, & \mathcal{P}_{44} (W) := \int_\mathbb{R} a^{X} \frac{\tilde{\phi}_{\xi\xi}}{\bar{v}^2} F_4 \, d\xi.
\end{array}
\end{equation*}
As in the estimates for $\mathcal{P}_{21}$ and $\mathcal{P}_{22}$, we obtain
\begin{equation*}
|\mathcal{P}_{41}| \leq C \int_\mathbb{R} |a^X_\xi|^{1/3} |\tilde{\phi}_{\xi t}|^2 \, d\xi + C \int_\mathbb{R} |a^X_\xi|^{5/3} | \tilde{v}|^2 \, d\xi \leq C \delta_S^{1/2} \int_\mathbb{R} \left( \bar{v}^S_\xi \tilde{v}^2 + \tilde{\phi}_{\xi t}^2 \right) \, d\xi.
\end{equation*}
and
\begin{equation*}
\begin{split}
| \mathcal{P}_{42} | & \leq C \int_\mathbb{R} |\bar{v}_\xi| |\tilde{v}| \left( |\tilde{\phi}_{\xi\xi}| + |\tilde{\phi}_{\xi t}| \right) \, d\xi \\
& \leq C \int |\bar{v}_\xi|^{1/2} \left( |\tilde{\phi}_{\xi\xi}|^2 + |\tilde{\phi}_{\xi t}|^2 \right) \, d\xi + C \int_\mathbb{R} |\bar{v}_\xi|^{3/2} |\tilde{v}|^2 \, d\xi \\
& \leq C (\delta_S + \delta_R^{1/2}) \int_\mathbb{R} \left( \bar{v}_\xi \tilde{v}^2 + \tilde{\phi}_{\xi\xi}^2 + \tilde{\phi}_{\xi t}^2 \right) \, d\xi.
\end{split}
\end{equation*}
By \eqref{Xbd}, $\mathcal{P}_{43}$ is bounded as
\begin{equation*}
|\mathcal{P}_{43}| \leq C \varepsilon_1 \int_\mathbb{R} \left( \bar{v}_\xi \tilde{v}^2 + \tilde{\phi}_{\xi\xi}^2 \right) \, d\xi.
\end{equation*}
Next we estimate $\mathcal{P}_{44}$ as
\begin{equation*}
\begin{split}
|\mathcal{P}_{44}| & \leq C \int_\mathbb{R} \left| \left( \ln{\bar{v}^S} - \ln{\bar{v}} \right)_{\xi\xi} \tilde{\phi}_{\xi\xi} \right| \, d\xi \\
& \leq C \int_\mathbb{R} \left( |\bar{v}^R_{\xi\xi}| + |\bar{v}^R_\xi|^2 + |\bar{v}^R_\xi||\bar{v}^S_\xi| + |\bar{v}^S_\xi(\bar{v}^R-v_m)| \right) |\tilde{\phi}_{\xi\xi}| \, d\xi \\
& \leq C \delta_R^{1/2} \int_\mathbb{R} \tilde{\phi}_{\xi\xi}^2 \, d\xi + C \delta_R^{-1/2} \left( \|\bar{v}^S_\xi (\bar{v}^R-v_m)\|_{L^2}^2 + \|\bar{v}^R_\xi\bar{v}^S_\xi\|_{L^2}^2 + \|\bar{v}^R_\xi\|_{L^4}^4 + \|\bar{v}^R_{\xi\xi}\|_{L^2}^2 \right),
\end{split}
\end{equation*}
where in the second inequality, we have used
\begin{equation*}
\begin{split}
| (\ln{\bar{v}^S} - \ln{\bar{v}})_{\xi\xi} | & = \bigg| \bar{v}^S_{\xi\xi} \left( \frac{1}{\bar{v}^S} - \frac{1}{\bar{v}} \right)- \bar{v}^S_\xi \bigg( \frac{\bar{v}^S_\xi}{(\bar{v}^S)^2} - \frac{\bar{v}_\xi}{\bar{v}^2} \bigg) - \frac{\bar{v}^R_{\xi\xi}}{\bar{v}} + \frac{\bar{v}_\xi\bar{v}^R_\xi}{\bar{v}^2} \bigg| \\
& \leq C \left( |\bar{v}^R_{\xi\xi}| + |\bar{v}^R_\xi|^2 + |\bar{v}^R_\xi||\bar{v}^S_\xi| + |\bar{v}^S_\xi(\bar{v}^R-v_m)| \right).
\end{split}
\end{equation*}

\noindent $\bullet$ Estimate of the term $\mathcal{P}_5$: Note that, by \eqref{CW1} with the bounds \eqref{shderiv1}, \eqref{raderiv}, and \eqref{Xbd},
\begin{equation*}
\bigg| \bigg( \frac{e^{-\bar{\phi}}}{\bar{v}^3} \bigg)_t \bigg| \leq C \left( |\bar{v}_\xi| + |\bar{u}_\xi| + |\dot{X}| |\bar{v}^S_\xi| \right) \leq C \bar{v}_\xi.
\end{equation*}
Thus, we have
\begin{equation*}
\begin{split}
|\mathcal{P}_5| & \leq C \int_\mathbb{R} |\bar{v}_\xi| \left( |\tilde{\phi}_{\xi\xi}|^2 + |\tilde{\phi}_\xi|^2 \right) \, d\xi + C \int_\mathbb{R} \left(|\bar{v}_\xi| + |a^X_\xi| \right) |\tilde{\phi}_{\xi\xi}||\tilde{\phi}_{\xi t}| \, d\xi \\
& \quad + C \int_\mathbb{R} |\bar{v}_\xi| |\tilde{\phi}||\tilde{\phi}_{\xi t}| \, d\xi + C \int_\mathbb{R} |\mathcal{V}_\xi| |\tilde{\phi}_{\xi t}| \, d\xi.
\end{split}
\end{equation*}
Since $|\bar{v}_\xi| \leq C (\delta_S^2 + \delta_R)$ and $|a^X_\xi| \leq C \delta_S^{3/2}$, the first two terms are bounded by
\begin{equation*}
C(\delta_S^{3/2} + \delta_R) \int_\mathbb{R} \left( \tilde{\phi}_\xi^2 + \tilde{\phi}_{\xi\xi}^2 + \tilde{\phi}_{\xi t}^2 \right) \, d\xi.
\end{equation*}
We also have by Young's inequality that
\begin{equation*}
\begin{split}
\int_\mathbb{R} |\bar{v}_\xi| |\tilde{\phi}| |\tilde{\phi}_{\xi t}| \, d\xi & \leq C \int_\mathbb{R} |\bar{v}_\xi|^{1/2} |\tilde{\phi}_{\xi t}|^2 \, d\xi + C \int_\mathbb{R} |\bar{v}_\xi|^{3/2} |\tilde{\phi}|^2 \, d\xi \\
& \leq C \left( \delta_S + \delta_R^{1/2}\right) \int_\mathbb{R} \left(  \bar{v}_\xi \tilde{\phi}^2 + \tilde{\phi}_{\xi t}^2 \right) \, d\xi.
\end{split}
\end{equation*}
Similarly as in \eqref{Vxiphi}--\eqref{Vxiphi1}, the last term is estimates as 
\begin{equation*}
\begin{split}
\int_\mathbb{R} |\mathcal{V}_\xi||\tilde{\phi}_{\xi t}| \, d\xi & \leq C (\delta_S + \delta_R^{1/2} + \varepsilon_1) \int_\mathbb{R} \left( \bar{v}_\xi \tilde{v}^2 + \bar{v}_\xi \tilde{\phi}^2 + \tilde{\phi}_\xi^2 + \tilde{\phi}_{\xi\xi}^2 + \tilde{\phi}_{\xi\xi\xi}^2 + \tilde{\phi}_{\xi t}^2 + \tilde{v}_\xi^2  + \tilde{v}_{\xi\xi}^2 \right) \, d\xi \\
& \quad + C \delta_R^{-1/2} \left( \lVert \bar{v}^S_\xi (\bar{v}^R - v_m) \rVert_{L^2}^2 + \lVert \bar{v}^R_\xi (\bar{v}^S - v_m) \rVert_{L^2}^2 + \lVert \bar{v}^S_\xi \bar{v}^R_\xi \rVert_{L^2}^2 + \lVert \bar{v}^R_\xi \rVert_{L^4}^4 + \lVert \bar{v}^R_{\xi\xi\xi} \rVert_{L^2}^2 \right).
\end{split}
\end{equation*}

\noindent $\bullet$ Estimates of the terms $\tilde{P}_j$ with $j=1,2,3,4,5,6$: The terms $\tilde{P}_1$ and $\tilde{P}_2$ are estimated in the same manner as before, using Young's inequality together with the bounds \eqref{shderiv1}, \eqref{raderiv}, \eqref{infty}, and $|a^X_\xi| \leq C \delta_S^{3/2}$:
\begin{equation*}
\begin{split}
|\tilde{P}_1| & \leq C \int_\mathbb{R} |\bar{v}_\xi| |\tilde{v}| \left( |\tilde{\phi}_{\xi\xi\xi}| + |\bar{v}_\xi| |\tilde{\phi}_{\xi\xi}| \right) \, d\xi \\
& \leq C \int_\mathbb{R} |\bar{v}_\xi|^{1/2} |\tilde{\phi}_{\xi\xi\xi}|^2 \, d\xi + C \int_\mathbb{R} |\bar{v}_\xi|^{3/2} |\tilde{v}|^2 \, d\xi + C \int_\mathbb{R} |\bar{v}_\xi|^2 \left( |\tilde{v}|^2 + |\tilde{\phi}_{\xi\xi}|^2 \right) \, d\xi \\
& \leq C (\delta_S + \delta_R^{1/2}) \int_\mathbb{R} \left( \bar{v}_\xi \tilde{v}^2 + \tilde{\phi}_{\xi\xi}^2 + \tilde{\phi}_{\xi\xi\xi}^2 \right) \, d\xi, \\
|\tilde{P}_2| & \leq C \int_\mathbb{R} |a^X_\xi| |\tilde{\phi}_{\xi\xi}| \left( |\bar{v}_\xi| |\tilde{v}| + |\tilde{v}_\xi| \right) \, d\xi \leq C \delta_S^{3/2} \int_\mathbb{R} \left( \bar{v}_\xi \tilde{v}^2 + \tilde{v}_\xi^2 + \tilde{\phi}_{\xi\xi}^2 \right) \, d\xi.
\end{split}
\end{equation*}
Using $v_\xi = \tilde{v}_\xi + \bar{v}_\xi$ and $\phi_\xi = \tilde{\phi}_\xi + \bar{\phi}_\xi$, together with \eqref{shderiv1}--\eqref{raderiv}, we have
\begin{equation*}
\begin{split}
|(a^X ve^\phi)_\xi| & = \left| a^X_\xi ve^\phi + a^X e^\phi \left( \tilde{v}_\xi + \bar{v}_\xi \right) + a^X ve^\phi \left( \tilde{\phi}_\xi + \bar{\phi}_\xi \right) \right| \\
 & \leq C \left( |a^X_\xi| + |\bar{v}_\xi| + |\tilde{v}_\xi| + |\tilde{\phi}_\xi| \right).
\end{split}
\end{equation*}
By Young's inequality and \eqref{infty}, $\tilde{P}_3$ is bounded as
\begin{equation*}
\begin{split}
|\tilde{P}_3| & \leq C \int_\mathbb{R} \left( |a^X_\xi| + |\bar{v}_\xi| \right) |\tilde{\phi}_\xi| |\tilde{\phi}_{\xi\xi}| \, d\xi + C \int_\mathbb{R} \left( |\tilde{v}_\xi| + |\tilde{\phi}_\xi| \right) |\tilde{\phi}_\xi| |\tilde{\phi}_{\xi\xi}| \, d\xi \\
& \leq C \left( \delta_S^{3/2} + \delta_R + \lVert \tilde{v}_\xi \rVert_{L^\infty} + \lVert \tilde{\phi}_\xi \rVert_{L^\infty} \right) \int_\mathbb{R} \left( \tilde{\phi}_\xi^2 + \tilde{\phi}_{\xi\xi}^2 \right) \, d\xi \\
& \leq C (\delta_S^{3/2} + \delta_R + \varepsilon_1) \int_\mathbb{R} \left( \tilde{\phi}_\xi^2 + \tilde{\phi}_{\xi\xi}^2 \right) \, d\xi.
\end{split}
\end{equation*}
Similarly, using $|e^\phi-e^{\bar{\phi}}| \leq C|\tilde{\phi}|$ and \eqref{shderiv1}--\eqref{raderiv}, we have
\begin{equation*}
\begin{split}
\left| \left( a^X v (e^\phi-e^{\bar{\phi}}) \bar{\phi}_\xi \right)_\xi \right| & = \left| \left( a^X_\xi v + a^X v_\xi \right) (e^\phi-e^{\bar{\phi}}) \bar{\phi}_\xi + a^X ve^\phi \tilde{\phi}_\xi \bar{\phi}_\xi + a^X v (e^\phi-e^{\bar{\phi}}) \left( \bar{\phi}_{\xi\xi} + \bar{\phi}_\xi^2 \right) \right| \\
& \leq C |\bar{v}_\xi| \left( \left( 1 + |a^X_\xi| + |\bar{v}_\xi| + |\tilde{v}_\xi| \right) |\tilde{\phi}| + |\tilde{\phi}_\xi| \right).
\end{split}
\end{equation*}
Thus, Young's inequality and \eqref{infty} yield
\begin{equation*}
\begin{split}
|\tilde{P}_4| & \leq C \int_\mathbb{R} |\bar{v}_\xi| |\tilde{\phi}_{\xi\xi}| \left( \left( 1 + |a^X_\xi| + |\bar{v}_\xi| + |\tilde{v}_\xi| \right) |\tilde{\phi}| + |\tilde{\phi}_\xi| \right) \, d\xi \\
& \leq C (\delta_S + \delta_R^{1/2} + \varepsilon_1) \int_\mathbb{R} \left( \bar{v}_\xi \tilde{\phi}^2 + \tilde{\phi}_\xi^2 + \tilde{\phi}_{\xi\xi}^2 \right) \, d\xi.
\end{split}
\end{equation*}
Note that, by \eqref{sh3'} and \eqref{mcV},
\begin{equation*}
\mathcal{V}^I = e^{-\bar{\phi}} \bigg( \bigg( \frac{\bar{\phi}_\xi}{\bar{v}} \bigg)_\xi + 1 \bigg) - \bar{v}, \quad \mathcal{P}_{\mathrm{prof}} = - \frac{e^{\bar{\phi}}}{\bar{v}} \mathcal{V}^I.
\end{equation*}
The estimates of $\mathcal{V}^I$ and $\mathcal{V}^I_\xi$ in the proof of Lemma~\ref{V_est}, together with \eqref{shderiv1}--\eqref{raderiv}, yield
\begin{equation*}
\begin{split}
|\mathcal{P}_{\mathrm{prof},\xi}| & \leq C \left( |\mathcal{V}^I_\xi| + |\bar{v}_\xi| |\mathcal{V}^I| \right)  \leq C \left( |\bar{v}^S_\xi (\bar{v}^R-v_m)| + |\bar{v}^R_\xi (\bar{v}^S-v_m)| + |\bar{v}^R_\xi \bar{v}^S_\xi| + |\bar{v}^R_\xi|^2 + |\bar{v}^R_{\xi\xi\xi}| \right).
\end{split}
\end{equation*}
By the definition of $\mathcal{N}$, we have
\begin{equation*}
\begin{split}
\mathcal{N} & = \bigg( \frac{1}{\bar{v}^2} - \frac{1}{v^2} \bigg) \left( \tilde{\phi}_{\xi\xi} + \bar{\phi}_{\xi\xi} \right) + \frac{\tilde{v}_\xi \tilde{\phi}_\xi + \bar{v}_\xi \tilde{\phi}_\xi + \bar{\phi}_\xi \tilde{v}_\xi}{v^3}  + \bigg( \frac{1}{v^3} - \frac{1}{\bar{v}^3} \bigg) \bar{v}_\xi \bar{\phi}_\xi.
\end{split}
\end{equation*}
Using \eqref{shderiv1}, \eqref{raderiv}, and \eqref{infty}, we estimate
\begin{equation*}
\begin{split}
|\mathcal{N}_\xi| & \leq C |\bar{v}_\xi| |\tilde{v}| + C \left( |\bar{v}_\xi| + \lVert \tilde{v} \rVert_{W^{1,\infty}} + \lVert \tilde{\phi} \rVert_{W^{2,\infty}} \right) \left( |\tilde{v}_\xi| + |\tilde{v}_{\xi\xi}| + |\tilde{\phi}_\xi| + |\tilde{\phi}_{\xi\xi}| + |\tilde{\phi}_{\xi\xi\xi}| \right) \\
& \leq C |\bar{v}_\xi| |\tilde{v}| + C (\delta_S^2 + \delta_R + \varepsilon_1) \left( |\tilde{v}_\xi| + |\tilde{v}_{\xi\xi}| + |\tilde{\phi}_\xi| + |\tilde{\phi}_{\xi\xi}| + |\tilde{\phi}_{\xi\xi\xi}| \right).
\end{split}
\end{equation*}
Analogously to \eqref{Vxiphi}--\eqref{Vxiphi1}, we obtain
\begin{equation*}
\begin{split}
|\tilde{P}_5| & \leq C \int_\mathbb{R} \left( |\mathcal{P}_{\mathrm{prof},\xi}| + |\mathcal{N}_\xi| \right) \left( |\tilde{\phi}_{\xi\xi\xi}| + |\bar{v}_\xi| |\tilde{\phi}_{\xi\xi}| \right) \, d\xi \\
& \leq C (\delta_S + \delta_R^{1/2} + \varepsilon_1) \int_\mathbb{R} \left( \bar{v}_\xi \tilde{v}^2 + \tilde{\phi}_\xi^2 + \tilde{\phi}_{\xi\xi}^2 + \tilde{\phi}_{\xi\xi\xi}^2 + \tilde{v}_\xi^2 + \tilde{v}_{\xi\xi}^2 \right) \, d\xi \\
& \quad + C \delta_R^{-1/2} \left( \lVert \bar{v}^S_\xi (\bar{v}^R-v_m) \rVert_{L^2}^2 + \lVert \bar{v}^R_\xi (\bar{v}^S-v_m) \rVert_{L^2}^2 + \lVert \bar{v}^S_\xi \bar{v}^R_\xi \rVert_{L^2}^2 + \lVert \bar{v}^R_\xi \rVert_{L^4}^4 + \lVert \bar{v}^R_{\xi\xi\xi} \rVert_{L^2}^2 \right).
\end{split}
\end{equation*}
The term $\tilde{P}_6$ is estimated by
\begin{equation*}
|\tilde{P}_6| \leq C \int_\mathbb{R} |\bar{v}_\xi|^2 |\tilde{\phi}_{\xi\xi}|^2 \, d\xi \leq C (\delta_S^2 + \delta_R)^2 \int_\mathbb{R} \tilde{\phi}_{\xi\xi}^2 \, d\xi.
\end{equation*}

Collecting all the estimates, we have
\begin{equation} \label{R3ineq1}
\begin{split}
\mathcal{R}_3 & \leq - \mathcal{G}_3 - \mathcal{G}_4 + C \delta_S^2 |\dot{X}|^2 + C ( \delta_S^{1/2} + \delta_R^{1/2} + \varepsilon_1 ) \int_\mathbb{R} \left( \tilde{v}_{\xi\xi}^2 + \bar{v}_\xi \tilde{\phi}^2 + \tilde{\phi}_\xi^2 + \tilde{\phi}_{\xi t}^2 \right) \, d\xi  \\
& \quad + C ( \delta_S^{1/2} + \delta_R^{1/2} + \varepsilon_1 ) \int_\mathbb{R} \left( \bar{v}_\xi \tilde{v}^2 + \tilde{v}_\xi^2 + \bar{v}^S_\xi \tilde{h}^2 + \tilde{h}_\xi^2 + \tilde{\phi}_{\xi\xi}^2 + \tilde{\phi}_{\xi\xi\xi}^2 \right) \, d\xi \\
& \quad + C \left( \delta_R + \varepsilon_1 \right) \left( \lVert \bar{v}^S_\xi(\bar{v}^R-v_m) \rVert_{L^2} + \lVert \bar{v}^R_\xi \bar{v}^S_\xi \rVert_{L^2} + \lVert \bar{v}^R_\xi \rVert_{L^4}^2 \right) + C ( \delta_R^{1/3} + \varepsilon_1^{1/3} ) \lVert \bar{v}^R_{\xi\xi\xi} \rVert_{L^1}^{4/3} \\
& \quad + C \delta_R^{-1/2} \left( \lVert \bar{v}^S_\xi(\bar{v}^R-v_m) \rVert_{L^2}^2 + \lVert \bar{v}^R_\xi (\bar{v}^S-v_m)\rVert_{L^2}^2 + \lVert \bar{v}^R_\xi \bar{v}^S_\xi \rVert_{L^2}^2 + \lVert \bar{v}^R_\xi \rVert_{L^4}^4 + \lVert \bar{v}^R_{\xi\xi} \rVert_{L^2}^2 + \lVert \bar{v}^R_{\xi\xi\xi} \rVert_{L^2}^2 \right) .
\end{split}
\end{equation}
Using the bounds \eqref{hL2}--\eqref{vxiL2}, \eqref{hxiL2}, together with
\begin{equation*}
\mathcal{G}_3(W) \sim G_2(W), \quad \mathcal{G}_4(W) \sim G_3(W),
\end{equation*}
the estimate \eqref{R3ineq1} yields the bound \eqref{R3ineq} for sufficiently small $\delta_S +\delta_R + \varepsilon_1$.

\end{proof}

\subsubsection{Proof of Lemma~\ref{lem_0th}}

Combining the estimates in Lemma~\ref{R1R2est} and Lemma~\ref{R3est}, and integrating in time over $[0,t]$, we have
\begin{equation*}
\begin{split}
& \int_\mathbb{R} \eta(W|\bar{W})(t,\xi) \, d\xi + \int_0^t \left( \delta_S |\dot{X}|^2 + G_1 + G_2 + G_3 + G^S + G^R + D \right) d\tau \\
& \quad \leq C \int_\mathbb{R} \eta(W|\bar{W})(0,\xi) \, d\xi + C (\delta_0^{1/2} + \varepsilon_1) \int_0^t \int_\mathbb{R} \left( \tilde{v}_{\xi\xi}^2 + \tilde{u}_\xi^2 + \bar{v}_\xi \tilde{\phi}^2 + \tilde{\phi}_\xi^2 + \tilde{\phi}_{\xi t}^2 \right) \, d\xi d\tau \\
& \qquad + C \varepsilon_1 \delta_R^{4/3} + C ( \delta_R^{1/3} + \varepsilon_1^{1/3} ) \int_0^t \left( \lVert \bar{v}^R_{\xi\xi} \rVert_{L^1}^{4/3} +  \lVert \bar{v}^R_{\xi\xi\xi} \rVert_{L^1}^{4/3} \right) \, d\tau + C (\delta_R+ \varepsilon_1) \int_0^t \lVert \bar{v}^R_\xi \rVert_{L^4}^2 \, d\tau \\
& \qquad + C (\delta_R + \varepsilon_1 ) \int_0^t \left( \lVert \bar{v}^S_\xi (\bar{v}^R-v_m)\rVert_{L^2} + \lVert \bar{v}^R_\xi (\bar{v}^S-v_m)\rVert_{L^2} + \lVert \bar{v}^R_\xi \bar{v}^S_\xi \rVert_{L^2} \right) \, d\tau \\
& \qquad + C \delta_R^{-1/2} \int_0^t \Big( \|\bar{v}^S_\xi (\bar{v}^R-v_m)\|_{L^2}^2 + \lVert \bar{v}^R_\xi (\bar{v}^S-v_m)\rVert_{L^2}^2 + \|\bar{v}^R_\xi\bar{v}^S_\xi\|_{L^2}^2 + \|\bar{v}^R_\xi\|_{L^4}^4 + \|\bar{v}^R_{\xi\xi}\|_{L^2}^2 + \|\bar{v}^R_{\xi\xi\xi}\|_{L^2}^2 \Big)  \, d\tau.
\end{split}
\end{equation*}
Therefore, using \eqref{inter}--\eqref{Raest}, we obtain
\begin{equation*}
\begin{split}
& \int_\mathbb{R} \eta(W|\bar{W})(t,\xi) \, d\xi + \int_0^t \left( \delta_S |\dot{X}|^2 + G_1 + G_2 + G_3 + G^S + G^R + D \right) d\tau \\
& \quad \leq C \int_\mathbb{R} \eta(W|\bar{W})(0,\xi) \, d\xi + C (\delta_0^{1/2} + \varepsilon_1) \int_0^t \int_\mathbb{R} \left( \tilde{v}_{\xi\xi}^2 + \tilde{u}_\xi^2 + \bar{v}_\xi \tilde{\phi}^2 + \tilde{\phi}_\xi^2 + \tilde{\phi}_{\xi t}^2 \right) \, d\xi d\tau + C \delta_R^{1/3}.
\end{split}
\end{equation*}
This, combined with the equivalence property of $\eta(W | \bar{W})$ in Lemma~\ref{entsim} and the application of \eqref{phiH3}, together with \eqref{inter'}--\eqref{Raest'}, to the initial data, yields the desired inequality \eqref{0th}.

\section{Elliptic estimates} \label{Sec_5}

This section is devoted to providing estimates for $\tilde{\phi}$ and its derivatives appearing on the right-hand side of \eqref{0th}. With the elliptic equation \eqref{v} at hand, we establish the desired bounds in the following two lemmas.

\begin{lemma} \label{E_est}
Under the assumptions of Proposition~\ref{Apriori}, there exists a constant $C>0$ such that
\begin{equation} \label{phi_est}
\int_0^t \int_\mathbb{R} \left( \bar{v}_\xi \tilde{\phi}^2 + \tilde{\phi}_\xi^2 \right) \, d\xi d\tau \leq C \int_0^t \left( G^S + G^R + D + \lVert \tilde{v}_{\xi\xi} \rVert_{L^2}^2 \right) \, d\tau + C \sqrt{\delta_R}
\end{equation}
for all $t \in [0,T]$, where $G^S$, $G^R$, and $D$ are as defined in Proposition~\ref{Apriori}.
\end{lemma}

\begin{proof}

Multiplying \eqref{v} by $-\bar{v}_\xi\tilde{\phi}$ and integrating the resulting equation with respect to $\xi$, we have
\begin{equation} \label{phitemp}
\begin{split}
\int_\mathbb{R} \bigg( \frac{e^{-\bar{\phi}}\bar{v}_\xi \tilde{\phi}_\xi^2}{\bar{v}} + e^{-\bar{\phi}} \bar{v}_\xi \tilde{\phi}^2  \bigg) \, d\xi & = - \int_\mathbb{R} \bigg( \frac{e^{-\bar{\phi}}\bar{v}_\xi}{\bar{v}} \bigg)_\xi \tilde{\phi}_\xi \tilde{\phi} \, d\xi - \int_\mathbb{R} \bar{v}_\xi \tilde{\phi}\tilde{v} \, d\xi + \int_\mathbb{R} \bar{v}_\xi \tilde{\phi} \mathcal{V} \, d\xi \\
& =: \RNum{1} + \RNum{2} + \RNum{3}.
\end{split}
\end{equation}
Using Young's inequality together with \eqref{shderiv1}--\eqref{raderiv}, we obtain
\begin{equation*}
\begin{split}
|\RNum{1}| & \leq C \int_\mathbb{R} |\bar{v}_\xi| |\tilde{\phi}_\xi||\tilde{\phi}| \, d\xi \leq C \int_\mathbb{R} |\bar{v}_\xi|^{1/2} |\tilde{\phi}_\xi|^2 \, d\xi + C \int_\mathbb{R} |\bar{v}_\xi|^{3/2} |\tilde{\phi}|^2 \, d\xi \\
& \leq C (\delta_S + \delta_R^{1/2}) \int_\mathbb{R} \left( \bar{v}_\xi \tilde{\phi}^2 + \tilde{\phi}_\xi^2 \right) \, d\xi, \\
|\RNum{2}| & \leq \theta \int_\mathbb{R} \bar{v}_\xi \tilde{\phi}^2 \, d\xi + C \theta^{-1} \int_\mathbb{R} \bar{v}_\xi \tilde{v}^2 \, d\xi,
\end{split}
\end{equation*}
where $\theta>0$ is an arbitrarily small constant. To handle $\RNum{3}$, we use \eqref{V}:
\begin{equation*}
\begin{split}
|\RNum{3}| & \leq C \left( \delta_S^2 + \delta_R + \varepsilon_1 \right) \int_\mathbb{R} |\bar{v}_\xi| |\tilde{\phi}| \left( |\tilde{v}| + |\tilde{v}_\xi| + |\tilde{\phi}| + |\tilde{\phi}_\xi| + |\tilde{\phi}_{\xi\xi}| \right) \, d\xi \\
& \quad + C \int_\mathbb{R} |\bar{v}_\xi| |\tilde{\phi}| \left( |\bar{v}^S_\xi (\bar{v}^R-v_m)| + |\bar{v}^R_\xi\bar{v}^S_\xi| + |\bar{v}^R_\xi|^2 + |\bar{v}^R_{\xi\xi}| + |\bar{v}^S-v_m| |\bar{v}^R-v_m| \right) \, d\xi.
\end{split}
\end{equation*}
This yields, by Young’s inequality,
\begin{equation*}
\begin{split}
|\RNum{3}| & \leq C \left( \delta_S^2 + \delta_R + \varepsilon_1 \right) \int_\mathbb{R} \left( \bar{v}_\xi \tilde{\phi}^2 + \bar{v}_\xi \tilde{v}^2 + \tilde{\phi}_\xi^2 + \tilde{\phi}_{\xi\xi}^2 \right) \, d\xi \\
& \quad + C \left( \|\bar{v}^S_\xi (\bar{v}^R-v_m)\|_{L^2}^2 + \|\bar{v}^R_\xi\bar{v}^S_\xi\|_{L^2}^2 + \|\bar{v}^R_\xi\|_{L^4}^4 + \|\bar{v}^R_{\xi\xi}\|_{L^2}^2 + \|(\bar{v}^S-v_m )(\bar{v}^R-v_m)\|_{L^2}^2 \right).
\end{split}
\end{equation*}
Collecting the estimates for $\RNum{1}$, $\RNum{2}$, and $\RNum{3}$, we integrate \eqref{phitemp} in time over $[0,t]$ and obtain
\begin{equation} \label{phi0}
\begin{split}
& \int_0^t \int_\mathbb{R} \bar{v}_\xi \left( \tilde{\phi}^2 + \tilde{\phi}_{\xi}^2 \right) \, d\xi d\tau \\
& \quad  \leq C \int_0^t \int_\mathbb{R} \bar{v}_\xi \tilde{v}^2 \, d\xi d\tau  +  C ( \delta_0^{1/2} + \varepsilon_1) \int_0^t \int_\mathbb{R} \left( \tilde{\phi}_\xi^2 + \tilde{\phi}_{\xi\xi}^2 \right) \, d\xi  d\tau + C \delta_R
\end{split}
\end{equation}
for sufficiently small $\delta_0$ and $\varepsilon_1$, where we have used \eqref{inter}--\eqref{Raest}. Finally, applying \eqref{phiH3t} from Lemma~\ref{B_E} together with \eqref{vL2}–\eqref{vxiL2} to \eqref{phi0} yields the desired bound \eqref{phi_est}.

\end{proof}

\begin{lemma}
Under the assumptions of Proposition~\ref{Apriori}, there exists a constant $C>0$ such that
\begin{equation} \label{phixt_est}
\begin{split}
\int_0^t \int_\mathbb{R} \tilde{\phi}_{\xi t}^2 \, d\xi d\tau & \leq C \int_0^t \left( \delta_S |\dot{X}|^2 + G_2 + G_3 + G^S + G^R + D \right) \, d\tau \\
& \quad + C \int_0^t \left( \lVert \tilde{v}_{\xi\xi} \rVert_{L^2}^2 + \lVert \tilde{u}_{\xi} \rVert_{H^1}^2 \right) \,  d\tau + C \sqrt{\delta_R}
\end{split}
\end{equation}
for all $t \in [0,T]$, where $G_2$, $G_3$, $G^S$, $G^R$, and $D$ are as defined in Proposition~\ref{Apriori}.
\end{lemma}

\begin{proof}
We rewrite \eqref{v} as
\begin{equation*}
\tilde{\phi}_{\xi\xi} - \bar{v}\tilde{\phi} = \bar{v}e^{\bar{\phi}} \tilde{v} - \bar{v}e^{\bar{\phi}} \mathcal{V}.
\end{equation*}
Differentiating this with respect to $t$ and taking the $L^2$-inner product against $-\tilde{\phi}_t$, we obtain
\begin{equation*}
\begin{split}
\int_\mathbb{R} \left( \tilde{\phi}_{\xi t}^2 + \bar{v} \tilde{\phi}_t^2 \right) \, d\xi &  = - \int_\mathbb{R} \bar{v}_t \tilde{\phi}_t \tilde{\phi} \, d\xi - \int_\mathbb{R} (\bar{v}e^{\bar{\phi}})_t \tilde{\phi}_t \tilde{v} \, d\xi - \int_\mathbb{R} \bar{v}e^{\bar{\phi}} \tilde{\phi}_t \tilde{v}_t \, d\xi \\
& \quad + \int_\mathbb{R} (\bar{v}e^{\bar{\phi}})_t \tilde{\phi}_t \mathcal{V} \, d\xi +  \int_\mathbb{R} \bar{v} e^{\bar{\phi}} \tilde{\phi}_t \mathcal{V}_t \, d\xi.
\end{split}
\end{equation*}
Note that, by \eqref{ra1}, \eqref{shderiv1}, and \eqref{raderiv},
\begin{equation} \label{bar_t}
|\bar{v}_t| = | \bar{v}^R_t + \bar{v}^S_t | = | \sigma \bar{v}^R_\xi + \bar{u}^R_\xi - \dot{X}(t) \bar{v}^S_\xi | \leq C \bar{v}^R_\xi + C \varepsilon_1 \bar{v}^S_\xi \leq C \bar{v}_\xi \leq C (\delta_S^2 + \delta_R).
\end{equation}
Similar estimates hold for $\bar{u}$ and $\bar{\phi}$ as well. Thus, we have
\begin{equation*}
\begin{split}
\int_\mathbb{R} \left( \tilde{\phi}_{\xi t}^2 + \bar{v} \tilde{\phi}_t^2 \right) \, d\xi & \leq C \int_\mathbb{R} |\bar{v}_\xi| | \tilde{\phi}_t|| \tilde{\phi} | \, d\xi + C \int_\mathbb{R} | \bar{v}_\xi| | \tilde{\phi}_t|| \tilde{v} | \, d\xi + C \int_\mathbb{R} | \tilde{\phi}_t || \tilde{v}_t | \, d\xi \\
& \quad + C \int_\mathbb{R} |\bar{v}_\xi| |\tilde{\phi}_t|| \mathcal{V}| \, d\xi + C \int_\mathbb{R} |\tilde{\phi}_t| |\mathcal{V}_t| \, d\xi \\
& =: \RNum{1} + \RNum{2} + \RNum{3} + \RNum{4} + \RNum{5}.
\end{split}
\end{equation*}
By Young's inequality, we estimate the first two terms as
\begin{equation} \label{vxiphit}
\begin{split}
|\RNum{1}| & \leq C \int_\mathbb{R} \left( |\bar{v}_\xi|^{1/2}|\tilde{\phi}_t|^2 + |\bar{v}_\xi|^{3/2}|\tilde{\phi}|^2 \right) \, d\xi \leq C (\delta_S + \delta_R^{1/2}) \int_\mathbb{R} \left( \bar{v}_\xi \tilde{\phi}^2 + \tilde{\phi}_t^2 \right) \, d\xi, \\
|\RNum{2}| & \leq C \int_\mathbb{R} \left( |\bar{v}_\xi|^{1/2}|\tilde{\phi}_t|^2 + |\bar{v}_\xi|^{3/2}|\tilde{v}|^2 \right) \, d\xi \leq C (\delta_S + \delta_R^{1/2}) \int_\mathbb{R} \left( \bar{v}_\xi \tilde{v}^2 + \tilde{\phi}_t^2 \right) \, d\xi.
\end{split}
\end{equation}
The third term $\RNum{3}$ is estimated using \eqref{eq_v}:
\begin{equation} \label{term3}
\begin{split}
|\RNum{3}| & \leq C \int_\mathbb{R} |\tilde{\phi}_t| \left( |\tilde{v}_\xi| + |\tilde{u}_\xi| + |\dot{X}||\bar{v}^S_\xi| \right) \, d\xi \\
& \leq \theta \int_\mathbb{R} \tilde{\phi}_t^2 \, d\xi + C \theta^{-1} \int_\mathbb{R} \left( \tilde{v}_\xi^2 + \tilde{u}_\xi^2 \right) \, d\xi + C \delta_S \left( \delta_S |\dot{X}|^2 +  \int_\mathbb{R} \tilde{\phi}_t^2 \, d\xi \right),
\end{split}
\end{equation}
where $\theta>0$ is arbitrarily small, and H\"older's inequality yields
\begin{equation} \label{Xphit}
\begin{split}
|\dot{X}| \int_\mathbb{R} |\bar{v}^S_\xi||\tilde{\phi}_t| \, d\xi & \leq C\delta_S^2 |\dot{X}|^2 + \frac{C}{\delta_S^2} \left( \int_\mathbb{R} |\bar{v}^S_\xi||\tilde{\phi}_t| \, d\xi \right)^2 \\
& \leq C\delta_S^2 |\dot{X}|^2 + \frac{C}{\delta_S^2} \left( \int_\mathbb{R} \bar{v}^S_\xi \, d\xi \right) \left( \int_\mathbb{R} \bar{v}^S_\xi \tilde{\phi}_t^2 \, d\xi \right) \\
& \leq C \delta_S \left( \delta_S |\dot{X}|^2 + \int_\mathbb{R} \tilde{\phi}_t^2 \, d\xi \right).
\end{split}
\end{equation}
Using \eqref{V}, we obtain
\begin{equation*}
\begin{split}
|\RNum{4}| & \leq C \left( \delta_S^2 + \delta_R + \varepsilon_1 \right) \int_\mathbb{R} |\bar{v}_\xi| |\tilde{\phi}_t| \left( |\tilde{v}| + |\tilde{v}_\xi| + |\tilde{\phi}| + |\tilde{\phi}_\xi| + |\tilde{\phi}_{\xi\xi}| \right) \, d\xi \\
& \quad + C \int_\mathbb{R} |\bar{v}_\xi| |\tilde{\phi}_t| \left( |\bar{v}^S_\xi (\bar{v}^R-v_m)| + |\bar{v}^R_\xi\bar{v}^S_\xi| + |\bar{v}^R_\xi|^2 + |\bar{v}^R_{\xi\xi}| + |\bar{v}^S-v_m| |\bar{v}^R-v_m| \right) \, d\xi,
\end{split}
\end{equation*}
and, by Young’s inequality,
\begin{equation*}
\begin{split}
|\RNum{4}| & \leq C \left( \delta_S^2 + \delta_R + \varepsilon_1 \right) \int_\mathbb{R} \left( \bar{v}_\xi \tilde{\phi}^2 + \bar{v}_\xi \tilde{v}^2 + \tilde{\phi}_\xi^2 + \tilde{\phi}_{\xi\xi}^2 + \tilde{\phi}_t^2 \right) \, d\xi \\
& \quad + C \left( \|\bar{v}^S_\xi (\bar{v}^R-v_m)\|_{L^2}^2 + \|\bar{v}^R_\xi\bar{v}^S_\xi\|_{L^2}^2 + \|\bar{v}^R_\xi\|_{L^4}^4 + \|\bar{v}^R_{\xi\xi}\|_{L^2}^2 + \|(\bar{v}^S-v_m )(\bar{v}^R-v_m)\|_{L^2}^2 \right).
\end{split}
\end{equation*}
Lastly, using \eqref{Vt}, we have
\begin{equation} \label{term5}
\begin{split}
|\RNum{5}| & \leq C \left( \delta_S^2 + \delta_R + \varepsilon_1 \right) \int_\mathbb{R} |\tilde{\phi}_t| \left( |\tilde{v}_\xi| + |\tilde{v}_{\xi\xi}| + |\tilde{u}_\xi| + |\tilde{u}_{\xi\xi}| + |\tilde{\phi}_\xi| + |\tilde{\phi}_{\xi\xi}| + |\tilde{\phi}_t| + |\tilde{\phi}_{\xi t}| + |\tilde{\phi}_{\xi\xi t}| \right) \, d\xi \\
& \quad + C \int_\mathbb{R} |\tilde{\phi}_t|  \left( |\bar{v}^R_\xi(\bar{v}^S-v_m)| + |\bar{v}^R_\xi \bar{v}^S_\xi| + |\bar{v}^R_\xi|^2 + |\bar{v}^R_{\xi\xi\xi}| \right) \, d\xi \\
& \quad +  C |\dot{X}| \int_\mathbb{R} |\bar{v}^S_\xi||\tilde{\phi}_t| \, d\xi + C \int_\mathbb{R} |\bar{v}_\xi| |\tilde{\phi}_t| \left( |\tilde{v}| + |\tilde{\phi}| \right)\, d\xi.
\end{split}
\end{equation}
Notice that the last two terms are already handled in \eqref{vxiphit} and \eqref{Xphit}. Thus, we obtain
\begin{equation} \label{term51}
\begin{split}
|\RNum{5}| & \leq C \left( \delta_S + \delta_R^{1/2} + \varepsilon_1 \right) \\
& \qquad \times \left( \delta_S |\dot{X}|^2 +  \int_\mathbb{R} \left( \bar{v}_\xi \tilde{v}^2 + \bar{v}_\xi \tilde{\phi}^2 + \tilde{v}_\xi^2 + \tilde{v}_{\xi\xi}^2 + \tilde{u}_\xi^2 + \tilde{u}_{\xi\xi}^2 + \tilde{\phi}_\xi^2 + \tilde{\phi}_{\xi\xi}^2 + \tilde{\phi}_t^2 + \tilde{\phi}_{\xi t}^2 + \tilde{\phi}_{\xi\xi t}^2 \right) \, d\xi \right) \\
& \quad + C \delta_R^{1/2} \lVert \tilde{\phi}_t \rVert_{L^2}^2 + C \delta_R^{-1/2} \left( \lVert \bar{v}^R_\xi (\bar{v}^S-v_m)\rVert_{L^2}^2 + \lVert \bar{v}^R_\xi \bar{v}^S_\xi \rVert_{L^2}^2 + \lVert \bar{v}^R_\xi \rVert_{L^4}^4 + \lVert \bar{v}^R_{\xi\xi\xi} \rVert_{L^2}^2 \right).
\end{split}
\end{equation}
Combining all the estimates, we deduce
\begin{equation} \label{phixt}
\begin{split}
& \int_\mathbb{R} \left( \tilde{\phi}_t^2 + \tilde{\phi}_{\xi t}^2 \right) \, d\xi \\
& \quad \leq C (\delta_S + \delta_R^{1/2} + \varepsilon_1) \left( \delta_S |\dot{X}|^2 + \int_\mathbb{R} \left( \bar{v}_\xi \tilde{v}^2 + \bar{v}_\xi \tilde{\phi}^2 + \tilde{v}_{\xi\xi}^2 + \tilde{u}_{\xi\xi}^2 + \tilde{\phi}_\xi^2 + \tilde{\phi}_{\xi\xi}^2 + \tilde{\phi}_{\xi\xi t}^2 \right) \, d\xi  \right) \\
& \qquad  + C \int_\mathbb{R} \left( \tilde{v}_\xi^2 + \tilde{u}_\xi^2 \right) \, d\xi  + C \delta_R^{-1/2} \left( \lVert \bar{v}^R_\xi (\bar{v}^S-v_m)\rVert_{L^2}^2 + \lVert \bar{v}^R_\xi \bar{v}^S_\xi \rVert_{L^2}^2 + \lVert \bar{v}^R_\xi \rVert_{L^4}^4 + \lVert \bar{v}^R_{\xi\xi\xi} \rVert_{L^2}^2 \right)
\end{split}
\end{equation}
for sufficiently small $\delta_0$ and $\varepsilon_1$.

We differentiate \eqref{v} once with respect to $\xi$ and $t$:
\begin{equation*}
\begin{split}
& \bigg( \frac{e^{-\bar{\phi}}}{\bar{v}} \bigg)_{\xi t} \tilde{\phi}_{\xi\xi} + \bigg( \frac{e^{-\bar{\phi}}}{\bar{v}} \bigg)_{\xi} \tilde{\phi}_{\xi\xi t} + \bigg( \frac{e^{-\bar{\phi}}}{\bar{v}} \bigg)_{t} \tilde{\phi}_{\xi\xi\xi} + \bigg( \frac{e^{-\bar{\phi}}}{\bar{v}} \bigg) \tilde{\phi}_{\xi\xi\xi t} \\
&\quad  - (e^{-\bar{\phi}})_{\xi t} \tilde{\phi} - (e^{-\bar{\phi}})_{\xi} \tilde{\phi}_t - (e^{-\bar{\phi}})_{t} \tilde{\phi}_\xi - e^{-\bar{\phi}} \tilde{\phi}_{\xi t} = \tilde{v}_{\xi t} - \mathcal{V}_{\xi t}.
\end{split}
\end{equation*}
Multiplying this by $-\tilde{\phi}_{\xi t}$ and integrating over $\mathbb{R}$, we obtain
\begin{equation*}
\begin{split}
\int_{\mathbb{R}} \bigg( \frac{e^{-\bar{\phi}}}{\bar{v}} \tilde{\phi}_{\xi\xi t}^2 + e^{-\bar{\phi}} \tilde{\phi}_{\xi t}^2 \bigg) \, d\xi & = \int_\mathbb{R} \bigg( \frac{e^{-\bar{\phi}}}{\bar{v}} \bigg)_{\xi t} \tilde{\phi}_{\xi t} \tilde{\phi}_{\xi\xi}  \, d\xi  + \int_\mathbb{R} \bigg( \frac{e^{-\bar{\phi}}}{\bar{v}} \bigg)_{t} \tilde{\phi}_{\xi t} \tilde{\phi}_{\xi\xi\xi} \, d\xi \\
& \quad  - \int_\mathbb{R} (e^{-\bar{\phi}})_{\xi t} \tilde{\phi}_{\xi t} \tilde{\phi} \, d\xi -  \int_\mathbb{R} (e^{-\bar{\phi}})_{\xi} \tilde{\phi}_{\xi t} \tilde{\phi}_t  \, d\xi - \int_\mathbb{R} (e^{-\bar{\phi}})_{t} \tilde{\phi}_{\xi t}  \tilde{\phi}_\xi \, d\xi \\
& \quad + \int_\mathbb{R} \tilde{\phi}_{\xi\xi t} \tilde{v}_{t} \, d\xi - \int_\mathbb{R} \tilde{\phi}_{\xi\xi t} \mathcal{V}_{ t} \, d\xi.
\end{split}
\end{equation*}
Recalling \eqref{bar_t}, and using similar computations as before, we bound the first two lines by
\begin{equation*}
C \left( \delta_S + \delta_R^{1/2} \right) \int_\mathbb{R} \left( \bar{v}_\xi \tilde{\phi}^2 + \tilde{\phi}_\xi^2 + \tilde{\phi}_{\xi\xi}^2 + \tilde{\phi}_{\xi\xi\xi}^2 + \tilde{\phi}_{t}^2 + \tilde{\phi}_{\xi t}^2 \right) \, d\xi.
\end{equation*}
Analogously to \eqref{term3}, we estimate the first term in the last line as
\begin{equation*}
\begin{split}
\left| \int_\mathbb{R} \tilde{\phi}_{\xi\xi t} \tilde{v}_{t} \, d\xi \right| & \leq C \int_\mathbb{R} |\tilde{\phi}_{\xi\xi t}| \left( |\tilde{v}_{\xi}| + |\tilde{u}_{\xi}| + |\dot{X}||\bar{v}^S_\xi| \right) \, d\xi \\
& \leq \theta \int_\mathbb{R} \tilde{\phi}_{\xi\xi t}^2 \, d\xi + C \theta^{-1} \int_\mathbb{R} \left( \tilde{v}_\xi^2 + \tilde{u}_\xi^2 \right) \, d\xi + C \delta_S \left( \delta_S |\dot{X}|^2 + \int_\mathbb{R} \tilde{\phi}_{\xi\xi t}^2 \, d\xi \right).
\end{split}
\end{equation*}
As in \eqref{term5}--\eqref{term51}, we obtain
\begin{equation*}
\begin{split}
& \int_\mathbb{R} |\tilde{\phi}_{\xi\xi t}| |\mathcal{V}_t| \, d\xi \\
& \quad \leq C \left( \delta_S + \delta_R^{1/2} + \varepsilon_1 \right) \\
& \qquad \quad \times \left( \delta_S |\dot{X}|^2 +  \int_\mathbb{R} \left( \bar{v}_\xi \tilde{v}^2 + \bar{v}_\xi \tilde{\phi}^2 + \tilde{v}_\xi^2 + \tilde{v}_{\xi\xi}^2 + \tilde{u}_\xi^2 + \tilde{u}_{\xi\xi}^2 + \tilde{\phi}_\xi^2 + \tilde{\phi}_{\xi\xi}^2 + \tilde{\phi}_t^2 + \tilde{\phi}_{\xi t}^2 + \tilde{\phi}_{\xi\xi t}^2 \right) \, d\xi \right) \\
& \qquad + C \delta_R^{1/2} \lVert \tilde{\phi}_{\xi\xi t} \rVert_{L^2}^2 + C \delta_R^{-1/2} \left( \lVert \bar{v}^R_\xi (\bar{v}^S-v_m)\rVert_{L^2}^2 + \lVert \bar{v}^R_\xi \bar{v}^S_\xi \rVert_{L^2}^2 + \lVert \bar{v}^R_\xi \rVert_{L^4}^4 + \lVert \bar{v}^R_{\xi\xi\xi} \rVert_{L^2}^2 \right).
\end{split}
\end{equation*}
Hence, 
\begin{equation} \label{phixxt}
\begin{split}
& \int_\mathbb{R} \left( \tilde{\phi}_{\xi t}^2 + \tilde{\phi}_{\xi\xi t}^2 \right) \, d\xi \\
& \quad \leq C \left( \delta_S + \delta_R^{1/2} + \varepsilon_1 \right) \\
& \qquad \quad \times \left( \delta_S |\dot{X}|^2 + \int_\mathbb{R} \left( \bar{v}_\xi \tilde{v}^2 + \bar{v}_\xi \tilde{\phi}^2 + \tilde{v}_\xi^2 + \tilde{v}_{\xi\xi}^2 + \tilde{u}_\xi^2 + \tilde{u}_{\xi\xi}^2 + \tilde{\phi}_\xi^2 + \tilde{\phi}_{\xi\xi}^2 + \tilde{\phi}_{\xi\xi\xi}^2 + \tilde{\phi}_t^2 \right) \, d\xi \right) \\
& \qquad + C \int_\mathbb{R} \left( \tilde{v}_\xi^2 + \tilde{u}_\xi^2 \right) \, d\xi + C \delta_R^{-1/2} \left( \lVert \bar{v}^R_\xi (\bar{v}^S-v_m)\rVert_{L^2}^2 + \lVert \bar{v}^R_\xi \bar{v}^S_\xi \rVert_{L^2}^2 + \lVert \bar{v}^R_\xi \rVert_{L^4}^4 + \lVert \bar{v}^R_{\xi\xi\xi} \rVert_{L^2}^2 \right).
\end{split}
\end{equation}

Combining \eqref{phixt} with \eqref{phixxt}, and using the smallness of parameters, we conclude that
\begin{equation*}
\begin{split}
& \int_\mathbb{R} \left( \tilde{\phi}_t^2 + \tilde{\phi}_{\xi t}^2 + \tilde{\phi}_{\xi\xi t}^2\right) \, d\xi \\
& \quad  \leq  C (\delta_0^{1/2} + \varepsilon_1) \left(\delta_S |\dot{X}|^2  + \int_\mathbb{R} \left( \bar{v}_\xi \tilde{v}^2 + \bar{v}_\xi \tilde{\phi}^2 + \tilde{v}_\xi^2 + \tilde{v}_{\xi\xi}^2 + \tilde{u}_\xi^2 + \tilde{u}_{\xi\xi}^2 + \tilde{\phi}_\xi^2 + \tilde{\phi}_{\xi\xi}^2 + \tilde{\phi}_{\xi\xi\xi}^2 \right) \, d\xi \right) \\
& \qquad +  C \int_\mathbb{R} \left( \tilde{v}_\xi^2 + \tilde{u}_\xi^2 \right) \, d\xi + C \delta_R^{-1/2} \left( \lVert \bar{v}^R_\xi (\bar{v}^S-v_m)\rVert_{L^2}^2 + \lVert \bar{v}^R_\xi \bar{v}^S_\xi \rVert_{L^2}^2 + \lVert \bar{v}^R_\xi \rVert_{L^4}^4 + \lVert \bar{v}^R_{\xi\xi\xi} \rVert_{L^2}^2 \right).
\end{split}
\end{equation*}
This, together with \eqref{inter}, \eqref{Raest}, \eqref{vL2}, \eqref{vxiL2}, and \eqref{phi_est}, implies the desired estimate \eqref{phixt_est}.

\end{proof}

\section{Proof of the a priori estimate} \label{Sec_6}

We observe that the estimate \eqref{0th}, together with \eqref{phi_est} and \eqref{phixt_est}, yields
\begin{equation} \label{vhphi}
\begin{split}
& \lVert (\tilde{v},\tilde{h})(t,\cdot) \rVert_{L^2}^2 + \lVert \tilde{\phi}_\xi(t,\cdot) \rVert_{H^1}^2 + \int_0^t \left( \delta_S |\dot{X}|^2 + G_2 + G_3 + G^S + G^R + D \right) d\tau \\
& \quad \leq C ( \lVert \tilde{v}_{0} \rVert_{H^2}^2 + \lVert \tilde{h}(0,\cdot) \rVert_{L^2}^2 ) + C ( \sqrt{\delta_0} + \varepsilon_1) \int_0^t \left( \lVert \tilde{v}_{\xi\xi} \rVert_{L^2}^2 + \lVert \tilde{u}_\xi \rVert_{H^1}^2 \right) \,  d\tau + C \delta_R^{1/3}
\end{split}
\end{equation}
for all \( t \in [0,T] \). In this section, we complete the proof of the a priori estimate in Proposition~\ref{Apriori}. To this end, we first rewrite the energy estimate \eqref{vhphi}, expressed in terms of the effective velocity \( h \), in the original \((v,u,\phi)\)-variables. We then proceed to present higher-order estimates for \((\tilde{v},\tilde{u})\) and conclude the proof of Proposition~\ref{Apriori}.

\subsection{
\texorpdfstring{Estimate for $\| u-\bar{u} \|_{L^2(\mathbb{R})} $}{Estimate for || u-u¯||L2  }} \label{Sec_6.1}

The following lemma provides the zeroth-order energy estimate for the system \eqref{NSP''}.

\begin{lemma} \label{vu}
Under the assumptions of Proposition~\ref{Apriori}, there exists a constant $C>0$ such that
\begin{equation} \label{vu_ineq}
\begin{split}
& \lVert \tilde{v}(t,\cdot) \rVert_{H^1}^2 + \lVert \tilde{u} (t,\cdot) \rVert_{L^2}^2 + \lVert \tilde{\phi}_\xi (t,\cdot) \rVert_{H^1}^2 + \int_0^t \left( \delta_S |\dot{X}|^2 + G_2 + G_3 + G^S + G^R + D + \lVert \tilde{u}_\xi \rVert_{L^2}^2 \right) \, d\tau \\
& \quad \leq C \left( \lVert \tilde{v}_0 \rVert_{H^2}^2 + \lVert \tilde{u}_0 \rVert_{L^2}^2 \right) + C ( \sqrt{\delta_0} + \varepsilon_1) \int_0^t  \left( \lVert \tilde{v}_{\xi\xi} \rVert_{L^2}^2 + \lVert \tilde{u}_{\xi\xi} \rVert_{L^2}^2 \right) \, d\tau + C \delta_R^{1/3}
\end{split}
\end{equation}
for all $t \in [0,T]$, where $G_2$, $G_3$, $G^S$, $G^R$, and $D$ are as defined in Proposition~\ref{Apriori}.
\end{lemma}

\begin{proof}
We rewrite \eqref{NSP1''}--\eqref{NSP2''} and \eqref{CW1}--\eqref{CW2} in the form
\begin{equation*}
\partial_t U + \partial_\xi A(U) = \begin{pmatrix}
0 \\ \partial_\xi \left( \frac{\partial_\xi u}{v} \right) + \partial_\xi \Phi(v,\phi)
\end{pmatrix}, \quad U:= \begin{pmatrix}
v \\ u
\end{pmatrix}, \quad A(U) := \begin{pmatrix}
-\sigma v - u \\
-\sigma u + \tilde{p}(v)
\end{pmatrix}
\end{equation*}
and
\begin{equation*}
\partial_t \bar{U} + \partial_\xi A (\bar{U}) = \begin{pmatrix}
0 \\ \partial_\xi \left( \frac{\partial_\xi \bar{u}}{\bar{v}} \right) + \partial_\xi \Phi(\bar{v},\bar{\phi}) \end{pmatrix} - \dot{X}(t) \partial_\xi \bar{U}^S + \begin{pmatrix}
0 \\ F_1 + F_2 + F_3
\end{pmatrix},
\end{equation*}
respectively. Then we have the identity for the relative entropy $\tilde{\eta}(U|\bar{U})$:
\begin{equation*}
\begin{split}
\frac{d}{dt} \int_\mathbb{R} \tilde{\eta}\left(U(t,\xi)|\bar{U}(t,\xi) \right) \, d\xi = \dot{X}(t) \mathcal{Y}(U) + \sum_{j=1}^4 \mathcal{I}_j(U) + \sum_{j=1}^2 \mathcal{J}_j(U),
\end{split}
\end{equation*}
where
\begin{align*}
\mathcal{Y}(U) & := - \int_\mathbb{R} \tilde{p}'(\bar{v})\bar{v}^S_\xi \tilde{v} \, d\xi + \int_\mathbb{R} \bar{u}^S_\xi \tilde{u} \, d\xi, &
\mathcal{I}_1(U) & := - \int_\mathbb{R} \bar{u}_\xi  \tilde{p}(v|\bar{v}) \, d\xi, \\
\mathcal{I}_2(U) & := \int_\mathbb{R} \tilde{u} \bigg( \frac{\tilde{u}_\xi}{v} \bigg)_\xi \, d\xi, &
\mathcal{I}_3(U) & := \int_\mathbb{R} \tilde{u} \bigg( \bigg( \frac{1}{v} - \frac{1}{\bar{v}} \bigg) \bar{u}_\xi \bigg)_\xi \, d\xi, \\
\mathcal{I}_4(U) & := - \int_\mathbb{R} \tilde{u} \left( F_1 + F_2 \right) \, d\xi,
\end{align*}
and
\begin{equation*}
\mathcal{J}_1(U) := \int_\mathbb{R} \tilde{u} \left( \Phi(v,\phi) - \Phi(\bar{v},\bar{\phi}) \right)_\xi \, d\xi, \quad \mathcal{J}_2(U) := - \int_\mathbb{R} \tilde{u} F_3 \, d\xi.
\end{equation*}
Note that the terms $\mathcal{Y}$ and $\mathcal{I}_j$ (for $j=1,2,3,4$) correspond to the Navier-Stokes part of the NSP system. The contributions from these terms are estimated as in \cite[Lemma~5.1]{KVW2} (see also the proof of Lemma~\ref{R3est} for a similar treatment), leading to the following inequality:
\begin{equation} \label{tempu0th}
\begin{split}
& \int_\mathbb{R} \bigg( \frac{|u-\bar{u}|^2}{2} + Q(v|\bar{v}) \bigg)(t,\xi) \, d\xi + \frac{1}{2} \int_0^t \int_\mathbb{R} \left( \bar{u}^R_\xi \tilde{p}(v|\bar{v}) + |(u-\bar{u})_\xi|^2 \right)  \, d\xi d\tau \\
& \quad \leq \int_\mathbb{R} \bigg( \frac{|u_0-\bar{u}(0,\xi)|^2}{2} + Q(v_0|\bar{v}(0,\xi)) \bigg) \, d\xi + \int_0^t \bigg( \frac{\delta_S}{2}|\dot{X}|^2 + C\delta_S G_1 + c_1 G^S + C \delta_S D \bigg) \, d\tau \\
& \qquad + C \delta_R^{1/3} + \int_0^t \left( \left| \mathcal{J}_1| + |\mathcal{J}_2 \right| \right) \, d\tau,
\end{split}
\end{equation}
for a positive constant $c_1>0$. Thus, it remains to estimate $\mathcal{J}_1$ and $\mathcal{J}_2$.

We first decompose $\mathcal{J}_1$ by
\begin{equation*}
\begin{split}
\mathcal{J}_1 
& = - \frac{1}{2} \int_\mathbb{R} \tilde{u}_\xi \bigg( \frac{\tilde{\phi}_\xi^2}{v^2} + \frac{2\bar{\phi}_\xi \tilde{\phi}_\xi}{v^2} + \bigg(\frac{1}{v^2} - \frac{1}{\bar{v}^2} \bigg)\bar{\phi}_\xi^2 \bigg) \, d\xi \\
& \quad + \int_\mathbb{R} \tilde{u}_\xi \bigg( \frac{\tilde{\phi}_{\xi\xi}}{v^2} + \bigg( \frac{1}{v^2} - \frac{1}{\bar{v}^2} \bigg) \bar{\phi}_{\xi\xi} \bigg) \, d\xi \\
& \quad - \int_\mathbb{R} \tilde{u}_\xi \bigg( \frac{\tilde{v}_\xi \tilde{\phi}_{\xi}}{v^3} + \frac{\bar{v}_\xi \tilde{\phi}_\xi}{v^3} + \frac{\bar{\phi}_\xi \tilde{v}_\xi}{v^3} + \bigg( \frac{1}{v^3} - \frac{1}{\bar{v}^3} \bigg) \bar{v}_\xi \bar{\phi}_\xi \bigg) \, d\xi \\
& =: \mathcal{J}_{11} + \mathcal{J}_{12} + \mathcal{J}_{13}.
\end{split}
\end{equation*}
Applying Young's inequality, together with \eqref{infty}, \eqref{shderiv1}, we estimate each term in the decomposition of $\mathcal{J}_1$ as follows:
\begin{equation*}
\begin{split}
|\mathcal{J}_{11}| & \leq C \int_\mathbb{R} (|\tilde{\phi}_\xi| + |\bar{\phi}_\xi| ) |\tilde{u}_\xi| |\tilde{\phi}_\xi| \, d\xi + C \int_\mathbb{R} (|\tilde{v}| + |\bar{\phi}_\xi|) |\tilde{u}_\xi| |\bar{\phi}_\xi||\tilde{v}| \, d\xi \\
& \leq C \left( \delta_S^2 + \delta_R + \varepsilon_1 \right) \int_\mathbb{R} \left( \bar{v}_\xi \tilde{v}^2 + \tilde{u}_\xi^2 + \tilde{\phi}_\xi^2 \right) \, d\xi,\\
|\mathcal{J}_{12}| & \leq C \int_\mathbb{R} |\tilde{u}_\xi| |\tilde{\phi}_{\xi\xi}| \, d\xi + C \int_\mathbb{R} |\tilde{u}_\xi| |\bar{\phi}_{\xi\xi}| |\tilde{v}| \, d\xi \\
& \leq \theta \int_\mathbb{R} \tilde{u}_{\xi}^2 \, d\xi + C \theta^{-1} \int_\mathbb{R} \tilde{\phi}_{\xi\xi}^2 \, d\xi + C \left( \delta_S + \delta_R^{1/2} \right) \int_\mathbb{R} \left( \bar{v}_\xi \tilde{v}^2 + \tilde{u}_\xi^2 \right) \, d\xi,
\end{split}
\end{equation*}
with an arbitrary constant $\theta$ in $(0,1)$, and
\begin{equation*}
\begin{split}
|\mathcal{J}_{13}| & \leq C \int_\mathbb{R} |\tilde{u}_\xi| \left( |\tilde{v}_\xi||\tilde{\phi}_\xi| + |\bar{v}_\xi||\tilde{v}|^3 + |\bar{v}_\xi||\tilde{v}|^2 \right) \, d\xi + C \int_\mathbb{R} |\bar{v}_\xi| |\tilde{u}_\xi| \left( |\tilde{\phi}_\xi| + |\tilde{v}_\xi| + |\bar{v}_\xi||\tilde{v}| \right) \, d\xi \\
& \leq C \lVert \tilde{v} \rVert_{W^{1,\infty}} \int_\mathbb{R} \left( \bar{v}_\xi^2 + \tilde{u}_\xi^2 + \tilde{\phi}_\xi^2 \right) \, d\xi + C \int_\mathbb{R} |\bar{v}_\xi| \left( \bar{v}_\xi \tilde{v}^2 + \tilde{v}_\xi^2 + \tilde{u}_\xi^2 + \tilde{\phi}_\xi^2 \right) \, d\xi \\
& \leq C (\delta_S^2 + \delta_R + \varepsilon_1) \int_\mathbb{R} \left( \bar{v}_\xi \tilde{v}^2 + \tilde{v}_\xi^2 + \tilde{u}_\xi^2 + \tilde{\phi}_\xi^2 \right) \, d\xi.
\end{split}
\end{equation*}
Notice that the remaining term $\mathcal{J}_2$ has the same structure as the term $\mathcal{P}_1$ appearing in Lemma~\ref{Id}. Hence, by applying an analogous computation to that used in the proof of Lemma~\ref{R3est}, where $\mathcal{P}_1$ was estimated, we deduce that
\begin{equation*}
\begin{split}
|\mathcal{J}_2| & \leq C (\delta_R + \varepsilon_1) \lVert \tilde{u}_\xi \rVert_{L^2}^2 \\
& \quad + (\delta_R^{1/3} + \varepsilon_1^{1/3}) \left( \lVert \bar{v}^S_\xi (\bar{v}^R-v_m)\rVert_{L^2} + \lVert \bar{v}^R_\xi\bar{v}^S_\xi \rVert_{L^2} + \lVert \bar{v}^R_\xi \rVert_{L^4}^2 + \lVert \bar{v}^R_{\xi\xi} \rVert_{L^1}^{4/3} \right).
\end{split}
\end{equation*}
From the above estimates and \eqref{inter}--\eqref{Raest}, it follows that
\begin{equation*}
\begin{split}
\int_0^t \left( |\mathcal{J}_1|+|\mathcal{J}_2| \right) \, d\tau & \leq \theta \int_0^t \int_\mathbb{R} \tilde{u}_{\xi}^2 \, d\xi d\tau + C \theta^{-1} \int_0^t \int_\mathbb{R} \tilde{\phi}_{\xi\xi}^2 \, d\xi d\tau \\
& \quad  + C (\delta_S + \delta_R^{1/2} + \varepsilon_1) \int_0^t \int_\mathbb{R} \left( \bar{v}_\xi \tilde{v}^2 + \tilde{v}_\xi^2 + \tilde{u}_\xi^2 + \tilde{\phi}_\xi^2 \right) \, d\xi d\tau + C \delta_R^{1/3}.
\end{split}
\end{equation*}
Substituting the estimate into \eqref{tempu0th}, along with the application of \eqref{vL2}, \eqref{vxiL2}, \eqref{phi_est}, and
\begin{equation*}
\int_\mathbb{R} \bar{v}^R_\xi \tilde{p}(v|\bar{v}) \, d\xi \sim G^R,
\end{equation*}
we obtain
\begin{equation} \label{tempu}
\begin{split}
& \int_\mathbb{R} \left( \frac{|u-\bar{u}|^2}{2} + Q(v|\bar{v}) \right)(t,\xi) \, d\xi + \frac{1}{4} \int_0^t \int_\mathbb{R} \left( \bar{u}^R_\xi \tilde{p}(v|\bar{v}) + |(u-\bar{u})_\xi|^2 \right)  \, d\xi d\tau \\
& \quad \leq \int_\mathbb{R} \left( \frac{|u_0-\bar{u}(0,\xi)|^2}{2} + Q(v_0|\bar{v}(0,\xi)) \right) \, d\xi + \int_0^t \left( \frac{\delta_S}{2}|\dot{X}|^2 + c_1 G^S + c_2 G_2 \right) \, d\tau \\
& \qquad + C \delta_R^{1/3}  + C (\delta_S + \delta_R^{1/2} + \varepsilon_1) \int_0^t \left( G_1 + D \right) d\tau
\end{split}
\end{equation}
for a constant $c_2>0$, and for sufficiently small $\delta_S +\delta_R$ and $\varepsilon_1$.

Multiplying \eqref{tempu} by the constant $\textstyle \frac{1}{2\max{\{ 1,c_1,c_2 \}}}$, and then adding the result to \eqref{vhphi}, we have
\begin{equation} \label{tempu1}
\begin{split}
& \lVert (\tilde{v},\tilde{h},\tilde{u})(t,\cdot) \rVert_{L^2}^2 + \lVert \tilde{\phi}_\xi (t,\cdot) \rVert_{H^1}^2 + \int_0^t \left( \delta_S |\dot{X}|^2 + G_1 + G_2 + G_3 + G^S + G^R + D + D_1 \right) d\tau \\
& \quad \leq C ( \lVert \tilde{v}_{0} \rVert_{H^1}^2  +\lVert \tilde{h}(0,\cdot) \rVert_{L^2}^2 + \lVert \tilde{u}_0 \rVert_{L^2}^2 ) + C (\delta_0^{1/2} + \varepsilon_1) \int_0^t \int_\mathbb{R} \left( \tilde{v}_{\xi\xi}^2 + \tilde{u}_{\xi\xi}^2 \right) \, d\xi d\tau + C \delta_R^{1/3}.
\end{split}
\end{equation}

To complete the proof, we now rewrite the inequality in terms of $(\tilde{v},\tilde{u},\tilde{\phi})$. From the definitions of $h$ and $\tilde{h}$, we observe that
\begin{equation*}
\tilde{u} - \tilde{h} = \left(\ln{v} - \ln{\bar{v}^S} \right)_\xi = \frac{(v-\bar{v}^S)_\xi}{v} + \frac{\bar{v}^S_\xi (\bar{v}^S - v)}{v\bar{v}^S},
\end{equation*}
which yields
\begin{equation*}
\tilde{v}_\xi = v\tilde{u} - v\tilde{h} + \frac{\bar{v}^S_\xi \left( \tilde{v} + (\bar{v}^R-v_m) \right)}{\bar{v}^S} - \bar{v}^R_\xi.
\end{equation*}
Taking the $L^2$-norm, we obtain
\begin{equation*}
\lVert \tilde{v}_\xi \rVert_{L^2}^2 \leq C \left( \lVert \tilde{h} \rVert_{L^2}^2 + \lVert \tilde{u} \rVert_{L^2}^2 + \lVert \tilde{v} \rVert_{L^2}^2 + \delta_R^2 \right),
\end{equation*}
and similarly, using the initial data,
\begin{equation*}
\lVert \tilde{h}(0,\cdot) \rVert_{L^2}^2 \leq C \left( \lVert \tilde{v}_0 \rVert_{H^1}^2 + \lVert \tilde{u}_0 \rVert_{L^2}^2 + \delta_R^2 \right).
\end{equation*}
Substituting these bounds into \eqref{tempu1}, we conclude the proof.

\end{proof}

\subsection{Higher-order estimates and proof of Proposition~\ref{Apriori}}

The higher-order estimates are obtained by closely following the computations in \cite[Sections~6.1--6.2]{KKSh}, which were carried out for a single shock profile, and are adapted here to the present setting with the approximate composite wave. The main difference from the single shock case is the presence of remainder terms associated with the approximate rarefaction wave. For clarity, we state the corresponding bounds below and defer the detailed proofs to Appendix~\ref{High}.

\begin{lemma} \label{uxi}
Under the assumptions of Proposition~\ref{Apriori}, there exists a constant $C>0$ such that
\begin{equation} \label{uxi_ineq}
\begin{split}
\lVert \tilde{u}_\xi (t,\cdot) \rVert_{L^2}^2 + \int_0^t \lVert \tilde{u}_{\xi\xi} \rVert_{L^2}^2 \, d\tau & \leq C \lVert \tilde{u}_{0\xi} \rVert_{L^2}^2 + C (\sqrt{\delta_0} + \varepsilon_1) \int_0^t \left( \delta_S |\dot{X}|^2 + \lVert \tilde{u}_\xi \rVert_{L^2}^2 \right) \, d \tau \\
& \quad + C \int_0^t \left( G^S + G^R + D \right) \, d\tau + C \sqrt{\delta_R}
\end{split}
\end{equation}
for all $t \in [0,T]$, where $G^S$, $G^R$, and $D$ are as defined in Proposition~\ref{Apriori}.
\end{lemma}

\begin{lemma} \label{uxixi}
Under the assumptions of Proposition~\ref{Apriori}, there exists a constant $C>0$ such that
\begin{equation} \label{uxixi_ineq}
\begin{split}
\lVert \tilde{u}_{\xi\xi} (t,\cdot) \rVert_{L^2}^2 + \int_0^t \lVert \tilde{u}_{\xi\xi\xi} \rVert_{L^2}^2 \, d\tau & \leq C \lVert \tilde{u}_{0\xi\xi} \rVert_{L^2}^2 + C (\sqrt{\delta_0} + \varepsilon_1) \int_0^t \left( G^S + G^R + D + \lVert \tilde{u}_\xi \rVert_{H^1}^2 \right) \, d\tau \\
& \quad + C \int_0^t \left( G_2 + \lVert \tilde{v}_{\xi\xi} \rVert_{L^2}^2 \right) \, d\xi + C \sqrt{\delta_R}
\end{split}
\end{equation}
for all $t \in [0,T]$, where $G_2$, $G^S$, $G^R$, and $D$ are as defined in Proposition~\ref{Apriori}.
\end{lemma}

\begin{lemma} \label{vxixi}
Under the assumptions of Proposition~\ref{Apriori}, there exists a constant $C>0$ such that
\begin{equation} \label{vxixi_ineq}
\begin{split}
& \lVert \tilde{v}_{\xi\xi}(t,\cdot) \rVert_{L^2}^2 + \int_0^t \lVert \tilde{v}_{\xi\xi} \rVert_{L^2}^2 \, d\tau \\
&\quad \leq C \left( \lVert \tilde{v}_{0\xi\xi} \rVert_{L^2}^2 + \lVert \tilde{u}_{0\xi} \rVert_{L^2}^2 \right) + C \lVert \tilde{u}_\xi (t,\cdot) \rVert_{L^2}^2 + C \int_0^t \left( G_2 + \lVert \tilde{u}_{\xi\xi} \rVert_{L^2}^2 \right) \, d\tau \\
& \qquad + C (\sqrt{\delta_0} + \varepsilon_1) \int_0^t \left( \delta_S |\dot{X}|^2 + G^S + G^R + D \right) \, d\tau + C\sqrt{\delta_R}
\end{split}
\end{equation}
for all $t \in [0,T]$, where $G_2$, $G^S$, $G^R$, and $D$ are as defined in Proposition~\ref{Apriori}.
\end{lemma}

We now complete the proof of the a priori estimate.

\begin{proof}[Proof of Proposition~\ref{Apriori}]

To prove Proposition~\ref{Apriori}, we combine the results of Lemmas~\ref{vu}--\ref{vxixi}. Summing \eqref{vu_ineq}, \eqref{uxi_ineq}, \eqref{uxixi_ineq}, \eqref{vxixi_ineq}, and \eqref{phiH3}, we obtain
\begin{equation*}
\begin{split}
& \lVert (\tilde{v},\tilde{u},\tilde{\phi})(t,\cdot) \rVert_{H^2}^2 + \int_0^t \left( \delta_S |\dot{X}|^2 + G_2 + G_3 + G^S + G^R + D + \lVert \tilde{u}_\xi \rVert_{H^2}^2 + \lVert \tilde{v}_{\xi\xi} \rVert_{L^2}^2 \right) \, d\tau \\
& \quad \leq C \lVert ( \tilde{v}_0, \tilde{u}_0 ) \rVert_{H^2}^2 + C \delta_R^{1/3} \\
& \qquad  + C \lVert \tilde{v}(t,\cdot) \rVert_{H^2}^2 + C \lVert \tilde{u}_\xi (t,\cdot) \rVert_{L^2}^2 + C \int_0^t \left( G_2 + G^S + G^R + D + \lVert \tilde{v}_{\xi\xi} \rVert_{L^2}^2 + \lVert \tilde{u}_{\xi\xi} \rVert_{L^2}^2 \right) \, d\tau
\end{split}
\end{equation*}
for sufficiently small $\delta_S+\delta_R$ and $\varepsilon_1$. To control the last three terms on the right-hand side, we proceed step by step. First, we apply \eqref{vxixi_ineq} to bound \( \textstyle \lVert \tilde{v}_{\xi\xi} \rVert_{L^2}^2 + \int_0^t \lVert \tilde{v}_{\xi\xi} \rVert_{L^2}^2 \). Next, we use \eqref{uxi_ineq} to control \(  \textstyle \lVert \tilde{u}_\xi \rVert_{L^2}^2 + \int_0^t \lVert \tilde{u}_{\xi\xi} \rVert_{L^2}^2 \), and finally, \eqref{vu_ineq} yields a bound that closes the estimate. This completes the proof of Proposition~\ref{Apriori}.

\end{proof}

\appendix

\section{\texorpdfstring{Basic elliptic estimates for $\phi-\bar{\phi}$}{Basic elliptic estimates for (ϕ - ϕ̄)}}  

This appendix presents an elliptic estimate for $\tilde{\phi}$, as well as a time-integrated bound used in the proof of Lemma~\ref{E_est}.

\begin{lemma} \label{B_E}
Under the assumptions of Proposition~\ref{Apriori}, there exists a constant $C>0$ such that
\begin{equation} \label{phiH3}
\begin{split}
\lVert \tilde{\phi}(t,\cdot) \rVert_{H^3}^2 
&\leq C \lVert \tilde{v}(t,\cdot) \rVert_{H^2}^2 
+ C \delta_R^{-1/2} \Big( 
    \lVert \bar{v}^S_\xi(\bar{v}^R - v_m) \rVert_{L^2}^2 + \lVert \bar{v}^R_\xi(\bar{v}^S - v_m) \rVert_{L^2}^2
    + \lVert \bar{v}^R_\xi \bar{v}^S_\xi \rVert_{L^2}^2 \\
&\qquad 
    + \lVert \bar{v}^R_\xi \rVert_{L^4}^4 
    + \lVert \bar{v}^R_{\xi\xi} \rVert_{L^2}^2 + \lVert \bar{v}^R_{\xi\xi\xi} \rVert_{L^2}^2 
    + \lVert (\bar{v}^S - v_m)(\bar{v}^R - v_m) \rVert_{L^2}^2 
\Big)
\end{split}
\end{equation}
and
\begin{equation} \label{phiH3t}
\begin{split}
\int_0^t \lVert \tilde{\phi}_\xi (\tau,\cdot) \rVert_{H^2}^2 \, d\tau & \leq C \int_0^t \lVert \tilde{v}_\xi (\tau,\cdot) \rVert_{H^1}^2 \, d\tau + C ( \sqrt{\delta_0} + \varepsilon_1) \int_0^t \int_\mathbb{R} \left( \bar{v}_\xi \tilde{v}^2 + \bar{v}_\xi \tilde{\phi}^2 \right) \, d\xi d\tau + C \sqrt{\delta_R}
\end{split}
\end{equation}
for all $t \in [0,T]$.
\end{lemma}

\begin{proof}

Multiplying \eqref{v} by $- \tilde{\phi}$ and integrating the resulting equation with respect to $\xi$, we obtain after integration by parts
\begin{equation*}
\int_\mathbb{R} \bigg( \frac{e^{-\bar{\phi}} \tilde{\phi}_\xi^2}{\bar{v}} + e^{-\bar{\phi}} \tilde{\phi}^2 \bigg) \, d\xi  =  - \int_\mathbb{R} \bigg( \frac{e^{-\bar{\phi}}}{\bar{v}} \bigg)_\xi \tilde{\phi}_\xi \tilde{\phi} \, d\xi - \int_\mathbb{R} \tilde{\phi} \tilde{v} \, d\xi + \int_\mathbb{R} \tilde{\phi} \mathcal{V} \, d\xi.
\end{equation*}
Using the bounds \eqref{shderiv1}--\eqref{raderiv} and Young's inequality, we obtain
\begin{equation} \label{tempphi111}
\begin{split}
& \int_\mathbb{R} \bigg( \frac{e^{-\bar{\phi}} \tilde{\phi}_\xi^2}{\bar{v}} + e^{-\bar{\phi}} \tilde{\phi}^2 \bigg) \, d\xi \\
& \quad \leq C \int_\mathbb{R} |\bar{v}_\xi| | \tilde{\phi}_\xi|  |\tilde{\phi}| \, d\xi + C \int_\mathbb{R} |\tilde{\phi}| |\tilde{v}| \, d\xi + \int_\mathbb{R} |\tilde{\phi}| |\mathcal{V}| \, d\xi \\
& \quad \leq C (\delta_S^2 + \delta_R) \int_\mathbb{R} \left( \tilde{\phi}^2 + \tilde{\phi}_\xi^2 \right) \, d\xi + \theta \int_\mathbb{R} \tilde{\phi}^2 \, d\xi + C\theta^{-1} \int_\mathbb{R} \tilde{v}^2 \, d\xi + \int_\mathbb{R} |\tilde{\phi}| |\mathcal{V}| \, d\xi
\end{split}
\end{equation}
for any constant $0<\theta<1$. For the last term, we use \eqref{V} to obtain
\begin{equation} \label{tempphi222}
\begin{split}
\int_\mathbb{R} |\tilde{\phi}| |\mathcal{V}| \, d\xi & \leq C \left( \delta_S^2 + \delta_R + \varepsilon_1 \right) \int_\mathbb{R} |\tilde{\phi}| \left( |\tilde{v}| + |\tilde{v}_\xi| + |\tilde{\phi}| + |\tilde{\phi}_\xi| + |\tilde{\phi}_{\xi\xi}| \right) \, d\xi \\
& \quad + C \int_\mathbb{R} |\tilde{\phi}| \left( |\bar{v}^S_\xi (\bar{v}^R-v_m)| + |\bar{v}^R_\xi\bar{v}^S_\xi| + |\bar{v}^R_\xi|^2 + |\bar{v}^R_{\xi\xi}| + |\bar{v}^S-v_m| |\bar{v}^R-v_m| \right) \, d\xi \\
& \leq C \left( \delta_S^2 + \delta_R + \varepsilon_1 \right) \int_\mathbb{R} \left( \tilde{v}^2 +\tilde{v}_\xi^2 + \tilde{\phi}^2 + \tilde{\phi}_\xi^2 + \tilde{\phi}_{\xi\xi}^2 \right) \, d\xi + C \delta_R^{1/2} \lVert \tilde{\phi} \rVert_{L^2}^2 \\
& \quad + C \delta_R^{-1/2} \big( \lVert \bar{v}^S_\xi(\bar{v}^R-v_m)\rVert_{L^2}^2 + \lVert \bar{v}^R_\xi \bar{v}^S_\xi \rVert_{L^2}^2 \\
& \qquad \qquad \qquad + \lVert \bar{v}^R_\xi \rVert_{L^4}^4 + \lVert \bar{v}^R_{\xi\xi} \rVert_{L^2}^2 + \lVert (\bar{v}^S-v_m)(\bar{v}^R-v_m) \rVert_{L^2}^2 \big).
\end{split}
\end{equation}
Combining these estimates and using smallness of the parameters, we deduce
\begin{equation} \label{phiH1'}
\begin{split}
\int_\mathbb{R} \left( \tilde{\phi}^2 + \tilde{\phi}_\xi^2 \right) \, d\xi & \leq C \int_\mathbb{R} \tilde{v}^2 \, d\xi + C (\delta_S^2 + \delta_R + \varepsilon_1) \int_\mathbb{R} \left( \tilde{v}_\xi^2 + \tilde{\phi}_{\xi\xi}^2 \right) \, d\xi \\
& \quad + C \delta_R^{-1/2} ( \lVert \bar{v}^S_\xi(\bar{v}^R-v_m)\rVert_{L^2}^2 + \lVert \bar{v}^R_\xi \bar{v}^S_\xi \rVert_{L^2}^2 \\
& \qquad \qquad \qquad + \lVert \bar{v}^R_\xi \rVert_{L^4}^4 + \lVert \bar{v}^R_{\xi\xi} \rVert_{L^2}^2 + \lVert (\bar{v}^S-v_m)(\bar{v}^R-v_m) \rVert_{L^2}^2 ).
\end{split}
\end{equation}

Differentiating \eqref{v} with respect to $\xi$ and taking the $L^2$-inner product against $-\tilde{\phi}_\xi$, we obtain
\begin{equation*}
\begin{split}
\int_\mathbb{R} \bigg( \frac{e^{-\bar{\phi}}\tilde{\phi}_{\xi\xi}^2}{\bar{v}} + e^{-\bar{\phi}}\tilde{\phi}_\xi^2 \bigg) \, d\xi & = - \int_\mathbb{R} ( e^{-\bar{\phi}})_\xi \tilde{\phi} \tilde{\phi}_\xi \, d\xi  -  \int_\mathbb{R} \tilde{\phi}_\xi \tilde{v}_\xi \, d\xi + \int_\mathbb{R} \tilde{\phi}_\xi \mathcal{V}_\xi \, d\xi.
\end{split}
\end{equation*}
By \eqref{shderiv1}, \eqref{raderiv}, and \eqref{Vxi}, the right-hand side is bounded as
\begin{equation*}
\begin{split}
|R.H.S.| & \leq C \int_\mathbb{R} |\bar{v}_\xi| |\tilde{\phi}_\xi| \left( |\tilde{v}| + |\tilde{\phi}| \right) \, d\xi + C \int_\mathbb{R} |\tilde{\phi}_\xi||\tilde{v}_\xi| \, d\xi \\
& \quad + C (\delta_S^2 + \delta_R + \varepsilon_1) \int_\mathbb{R} |\tilde{\phi}_\xi| \left( |\tilde{v}_\xi|  + |\tilde{v}_{\xi\xi}| + |\tilde{\phi}_\xi| + |\tilde{\phi}_{\xi\xi}| + |\tilde{\phi}_{\xi\xi\xi}| \right) \, d\xi \\
& \quad + C \int_\mathbb{R} |\tilde{\phi}_{\xi\xi}| \left( | \bar{v}^S_\xi (\bar{v}^R - v_m) | + |\bar{v}^R_\xi (\bar{v}^S - v_m) | + | \bar{v}^R_\xi \bar{v}^S_\xi | + |\bar{v}^R_\xi|^2 + |\bar{v}^R_{\xi\xi\xi}| \right) \, d\xi \\
& =: I_1 + I_2 + I_3 + I_4.
\end{split}
\end{equation*}
We apply Young's inequality to estimate each term as follows:
\begin{equation} \label{phix11}
\begin{split}
|I_1| & \leq C \int_\mathbb{R} |\bar{v}_\xi|^{1/2} |\tilde{\phi}_\xi|^2 \, d\xi + C \int_\mathbb{R} |\bar{v}_\xi|^{3/2} \left( |\tilde{v}|^2 +  |\tilde{\phi}|^2 \right) \, d\xi \\
& \leq C (\delta_S + \delta_R^{1/2}) \int_\mathbb{R} \left( \bar{v}_\xi \tilde{v}^2 + \bar{v}_\xi \tilde{\phi}^2 + \tilde{\phi}_\xi^2 \right) \, d\xi, \\
|I_2| & \leq \theta \int_\mathbb{R} \tilde{\phi}_\xi^2 \, d\xi + C \theta^{-1} \int_\mathbb{R} \tilde{v}_\xi^2 \, d\xi,
\end{split}
\end{equation}
\begin{equation} \label{phix22}
|I_3| \leq C (\delta_S^2 + \delta_R + \varepsilon_1) \int_\mathbb{R} \left( \tilde{v}_\xi^2 + \tilde{v}_{\xi\xi}^2 + \tilde{\phi}_\xi^2 + \tilde{\phi}_{\xi\xi}^2 + \tilde{\phi}_{\xi\xi\xi}^2 \right) \, d\xi,
\end{equation}
and
\begin{equation} \label{phix33}
\begin{split}
|I_4| & \leq C \delta_R^{1/2} \lVert \tilde{\phi}_{\xi\xi} \rVert_{L^2}^2 + C \delta_R^{-1/2} \big( \| \bar{v}^S_\xi (\bar{v}^R - v_m) \|_{L^2}^2 + \|\bar{v}^R_\xi (\bar{v}^S - v_m) \|_{L^2}^2 \\
& \qquad  + \| \bar{v}^R_\xi \bar{v}^S_\xi \|_{L^2}^2 + \|\bar{v}^R_\xi\|_{L^4}^4 + \|\bar{v}^R_{\xi\xi\xi}\|_{L^2}^2 \big).
\end{split}
\end{equation}
These estimates gives
\begin{equation} \label{phiH2'}
\begin{split}
\int_\mathbb{R} \left( \tilde{\phi}_\xi^2 + \tilde{\phi}_{\xi\xi}^2 \right) \, d\xi & \leq C \int_\mathbb{R} \tilde{v}_\xi^2 \, d\xi + C (\delta_S + \delta_R^{1/2} + \varepsilon_1) \int_\mathbb{R} \left( \bar{v}_\xi \tilde{v}^2 + \tilde{v}_{\xi\xi}^2 + \bar{v}_\xi \tilde{\phi}^2 + \tilde{\phi}_{\xi\xi\xi}^2 \right) \, d\xi \\
& \quad + C \delta_R^{-1/2} \big( \| \bar{v}^S_\xi (\bar{v}^R - v_m) \|_{L^2}^2 + \|\bar{v}^R_\xi (\bar{v}^S - v_m) \|_{L^2}^2 \\
& \qquad \qquad \qquad + \| \bar{v}^R_\xi \bar{v}^S_\xi \|_{L^2}^2 + \|\bar{v}^R_\xi\|_{L^4}^4 + \|\bar{v}^R_{\xi\xi\xi}\|_{L^2}^2 \big)
\end{split}
\end{equation}
for sufficiently small $\delta_S+\delta_R$ and $\varepsilon_1$.

We differentiate \eqref{v} twice and then taking the $L^2$-inner product against $-\tilde{\phi}_{\xi\xi}$ to obtain
\begin{equation*}
\begin{split}
\int_\mathbb{R} \bigg( \frac{e^{-\bar{\phi}}\tilde{\phi}_{\xi\xi\xi}^2}{\bar{v}} + e^{-\bar{\phi}}\tilde{\phi}_{\xi\xi}^2 \bigg) \, d\xi & = \int_\mathbb{R} \bigg( \frac{e^{-\bar{\phi}}}{\bar{v}} \bigg)_{\xi\xi} \tilde{\phi}_{\xi\xi}^2 \, d\xi + \int_\mathbb{R} \bigg( \frac{e^{-\bar{\phi}}}{\bar{v}} \bigg)_\xi \tilde{\phi}_{\xi\xi\xi} \tilde{\phi}_{\xi\xi} \, d\xi - \int_\mathbb{R} (e^{-\bar{\phi}})_{\xi\xi} \tilde{\phi}\tilde{\phi}_{\xi\xi} \, d\xi  \\
& \quad - \int_\mathbb{R} 2(e^{-\bar{\phi}})_\xi \tilde{\phi}_\xi \tilde{\phi}_{\xi\xi} \, d\xi - \int_\mathbb{R} \tilde{\phi}_{\xi\xi} \tilde{v}_{\xi\xi} \, d\xi - \int_\mathbb{R} \tilde{\phi}_{\xi\xi\xi} \mathcal{V}_{\xi} \, d\xi.
\end{split}
\end{equation*}
We then obtain the following estimate by similar computations in \eqref{phix11}--\eqref{phix33}:
\begin{equation} \label{phiH3'}
\begin{split}
\int_\mathbb{R} \left( \tilde{\phi}_{\xi\xi}^2 + \tilde{\phi}_{\xi\xi\xi}^2 \right) \, d\xi & \leq C \int_\mathbb{R} \tilde{v}_{\xi\xi}^2 \, d\xi + C (\delta_S + \delta_R^{1/2} + \varepsilon_1) \int_\mathbb{R} \left( \bar{v}_\xi \tilde{v}^2 + \tilde{v}_{\xi}^2 + \bar{v}_\xi \tilde{\phi}^2 + \tilde{\phi}_{\xi}^2 \right) \, d\xi \\
& \quad + C \delta_R^{-1/2} \big( \| \bar{v}^S_\xi (\bar{v}^R - v_m) \|_{L^2}^2 + \|\bar{v}^R_\xi (\bar{v}^S - v_m) \|_{L^2}^2 \\
& \qquad \qquad \qquad + \| \bar{v}^R_\xi \bar{v}^S_\xi \|_{L^2}^2 + \|\bar{v}^R_\xi\|_{L^4}^4 + \|\bar{v}^R_{\xi\xi\xi}\|_{L^2}^2 \big).
\end{split}
\end{equation}

Combining \eqref{phiH1'}, \eqref{phiH2'}, \eqref{phiH3'}, and using the bound $|\bar{v}_\xi| \leq C \delta_0$, we obtain the bound \eqref{phiH3}. Furthermore, the time-integrated estimate \eqref{phiH3t} also follows from \eqref{phiH2'} and \eqref{phiH3'}, along with  \eqref{inter}--\eqref{Raest}.

\end{proof}

\section{Proof of Lemmas~\ref{uxi}--\ref{vxixi}} \label{High}

Note that the momentum equations \eqref{NSP2''} and \eqref{sh2'} can be rewritten in the original form (see \eqref{NSP}):
\begin{equation*}
u_t - \sigma u_\xi + p(v)_\xi = \bigg( \frac{u_\xi}{v} \bigg)_\xi - \frac{\phi_\xi}{v}
\end{equation*}
and
\begin{equation*}
\bar{u}^S_t - \sigma \bar{u}^S_\xi + p(\bar{v}^S)_\xi + \dot{X}(t) \bar{u}^S_\xi = \bigg( \frac{\bar{u}^S_\xi}{\bar{v}^S} \bigg)_\xi - \frac{\bar{\phi}^S_\xi}{\bar{v}^S}.
\end{equation*}
On the other hand, \eqref{ra2} is equivalent to
\begin{equation*}
\bar{u}^R_t - \sigma \bar{u}^R_\xi + p(\bar{v}^R)_\xi =  \frac{\bar{v}^R_\xi}{(\bar{v}^R)^2},
\end{equation*}
where the term on the right-hand side comes from the difference between $p(\cdot)$ and $\tilde{p}(\cdot)$. Thus, the perturbation equations are given by
\begin{subequations} \label{eq_vu}
\begin{align}
& \label{eq_v} \tilde{v}_t - \sigma \tilde{v}_\xi - \tilde{u}_\xi - \dot{X}(t) \bar{v}^S_\xi = 0, \\
& \label{eq_u} \tilde{u}_t - \sigma \tilde{u}_\xi + \left( p(v) - p(\bar{v}) \right)_\xi - \dot{X}(t) \bar{u}^S_\xi = \left( \frac{u_\xi}{v} - \frac{\bar{u}_\xi}{\bar{v}} \right)_\xi - \left( \frac{\phi_\xi}{v} - \frac{\bar{\phi}_\xi}{\bar{v}} \right) - \frac{\bar{v}^R_\xi}{(\bar{v}^R)^2} - F_1 - \frac{F_2}{2} - F_5,
\end{align}
\end{subequations}
where $F_1$ and $F_2$ are as in \eqref{F1F2_def}, and
\begin{equation*}
F_5 := \frac{\bar{\phi}_\xi}{\bar{v}} - \frac{\bar{\phi}^S_\xi}{\bar{v}^S}.
\end{equation*}
The proofs of Lemmas~\ref{uxi}--\ref{vxixi} are provided below, using the perturbation system \eqref{eq_vu}.

\begin{proof}[Proof of Lemma~\ref{uxi}]
Differentiating \eqref{eq_u} with respect to $\xi$ and multiplying the resulting equation by $\tilde{u}_\xi$, we have after rearrangement
\begin{equation*}
\begin{split}
\bigg( \frac{\tilde{u}_\xi^2}{2} \bigg)_t + \frac{\tilde{u}_{\xi\xi}^2}{v} & = (\cdots)_\xi - \tilde{u}_{\xi\xi} \left( \frac{\tilde{v}}{v\bar{v}} \right)_\xi + \frac{\tilde{v}_\xi \tilde{u}_{\xi\xi}\tilde{u}_\xi}{v^2} + \frac{\bar{v}_\xi \tilde{u}_{\xi\xi}\tilde{u}_\xi}{v^2} + \tilde{u}_{\xi\xi} \left( \frac{( \bar{u}_\xi+ \bar{\phi}_\xi)\tilde{v}}{v \bar{v}} \right)_\xi \\
& \quad + \frac{\tilde{u}_{\xi\xi}\tilde{\phi}_\xi}{v} + \dot{X}(t) \bar{u}^S_{\xi\xi} \tilde{u}_\xi + \tilde{u}_{\xi\xi} \left( F_1 + \frac{F_2}{2} \right) - \tilde{u}_\xi (F_5)_\xi - \tilde{u}_\xi \bigg( \frac{\bar{v}^R_\xi}{(\bar{v}^R)^2} \bigg)_\xi.
\end{split}
\end{equation*}
Integrating this over $\mathbb{R}$, we have 
\begin{equation*}
\frac{d}{dt} \int_\mathbb{R} \frac{\tilde{u}_\xi^2}{2} \, d\xi + \int_\mathbb{R} \frac{\tilde{u}_{\xi\xi}^2}{v} \, d\xi = \sum_{j=1}^7 \mathcal{U}^{(1)}_j,
\end{equation*}
where
\begin{equation*}
\begin{split}
\mathcal{U}^{(1)}_1 & := \int_\mathbb{R} \tilde{u}_{\xi\xi} \bigg( \frac{\tilde{\phi}_\xi}{v} - \frac{\tilde{v}_\xi}{v\bar{v}} \bigg) \, d\xi, \quad \mathcal{U}^{(1)}_2 := \dot{X}(t) \int_\mathbb{R} \bar{u}^S_{\xi\xi} \tilde{u}_\xi \, d\xi, \\
\mathcal{U}^{(1)}_3 & := \int_\mathbb{R} \tilde{u}_{\xi\xi} \left( \frac{\tilde{v}_\xi\tilde{v}}{v^2\bar{v}} + \frac{\tilde{v}_\xi \tilde{u}_\xi}{v^2}  \right) \, d\xi - \int_\mathbb{R} \tilde{u}_{\xi\xi} \left( \frac{(\bar{u}_\xi + \bar{\phi}_\xi)\tilde{v}_\xi \tilde{v}}{v^2 \bar{v}} \right) \, d\xi, \\
\mathcal{U}^{(1)}_4 & = \int_\mathbb{R} \tilde{u}_{\xi\xi} \left( \frac{\bar{v}_\xi \tilde{v}}{v^2 \bar{v}} + \frac{\bar{v}_\xi \tilde{v}}{v\bar{v}^2} + \frac{\bar{v}_\xi \tilde{u}_\xi}{v^2} \right) \, d\xi + \int_\mathbb{R} \tilde{u}_{\xi\xi} \bigg( \bigg( \frac{\bar{u}_\xi + \bar{\phi}_\xi}{\bar{v}} \bigg)_\xi \frac{\tilde{v}}{v} + \frac{( \bar{u}_\xi+ \bar{\phi}_\xi)\tilde{v}_\xi}{v \bar{v}} - \frac{\bar{v}_\xi(\bar{u}_\xi + \bar{\phi}_\xi)\tilde{v}}{v^2\bar{v}} \bigg) \, d\xi \\
\mathcal{U}^{(1)}_5 &:= \int_\mathbb{R} \tilde{u}_{\xi\xi} \left( F_1 + \frac{F_2}{2} \right) \, d\xi, \quad \mathcal{U}^{(1)}_6 := - \int_\mathbb{R} \tilde{u}_{\xi} (  F_5 )_\xi \, d\xi, \quad \mathcal{U}^{(1)}_7 := - \int_\mathbb{R} \tilde{u}_\xi \bigg( \frac{\bar{v}^R_\xi}{(\bar{v}^R)^2} \bigg)_\xi \, d\xi.
\end{split}
\end{equation*}
By applying Young's inequality, it holds that
\begin{equation*}
\begin{split}
|\mathcal{U}^{(1)}_1| & \leq C \int_\mathbb{R} |\tilde{u}_{\xi\xi}| \left( |\tilde{\phi}_\xi| + |\tilde{v}_\xi| \right) \, d\xi \leq \theta \int_\mathbb{R} \tilde{u}_{\xi\xi}^2 \, d\xi + C \theta^{-1} \int_\mathbb{R} \left( \tilde{\phi}_{\xi}^2 + \tilde{v}_\xi^2 \right) \, d\xi
\end{split}
\end{equation*}
for any constant $0 < \theta <1$. We use the bound \eqref{shderiv1}, and apply H\"older's and Young's inequalities to obtain
\begin{equation*}
\begin{split}
|\mathcal{U}^{(1)}_2| & \leq C |\dot{X}| \sqrt{\int_\mathbb{R} \bar{v}^S_\xi \, d\xi} \sqrt{\int_\mathbb{R} \bar{v}^S_\xi \tilde{u}_\xi^2 \, d\xi} \leq C \delta_S^2 |\dot{X}|^2 + \frac{C}{\delta_S^2} \left( \int_\mathbb{R} \bar{v}^S_\xi \, d\xi \right) \left( \int_\mathbb{R} \bar{v}^S_\xi \tilde{u}_\xi^2 \, d\xi \right)\\
& \leq C \delta_S \left( \delta_S |\dot{X}|^2 + \int_\mathbb{R} \tilde{u}_\xi^2 \, d\xi \right).
\end{split}
\end{equation*}
The nonlinear term $\mathcal{U}^{(1)}_3$ is estimated as
\begin{equation*}
\begin{split}
|\mathcal{U}^{(1)}_3| & \leq C \int_\mathbb{R} \left( |\tilde{v}| |\tilde{v}_\xi| |\tilde{u}_{\xi\xi}| + |\tilde{v}_\xi||\tilde{u}_{\xi\xi}||\tilde{u}_\xi| \right) \, d\xi \leq C (\delta_R + \varepsilon_1) \int_\mathbb{R} \left( \tilde{v}_\xi^2 + \tilde{u}_\xi^2 + \tilde{u}_{\xi\xi}^2 \right) \, d\xi,
\end{split}
\end{equation*}
where we have used \eqref{infty}. Using \eqref{shderiv1}--\eqref{raderiv}, we estimate
\begin{equation*}
\begin{split}
|\mathcal{U}^{(1)}_4| & \leq C \int_\mathbb{R} |\bar{v}_\xi| |\tilde{u}_{\xi\xi}| |\tilde{v}| \, d\xi + C \int_\mathbb{R} |\bar{v}_\xi| |\tilde{u}_{\xi\xi}| \left( |\tilde{v}_\xi| + |\tilde{u}_\xi| \right) \, d\xi \\
& \leq C \int_\mathbb{R} |\bar{v}_\xi|^{1/2} |\tilde{u}_{\xi\xi}|^2 \, d\xi + C \int_\mathbb{R} |\bar{v}_\xi|^{3/2} |\tilde{v}|^2 \, d\xi + C \int_\mathbb{R} |\bar{v}_\xi| \left( |\tilde{v}_\xi |^2 + | \tilde{u}_\xi |^2 + |\tilde{u}_{\xi\xi}|^2 \right) \, d\xi \\
& \leq C (\delta_S + \delta_R^{1/2}) \int_\mathbb{R} \left( \bar{v}_\xi \tilde{v}^2 + \tilde{v}_\xi^2 + \tilde{u}_\xi^2 + \tilde{u}_{\xi\xi}^2 \right) \, d\xi.
\end{split}
\end{equation*}
By the definitions of $F_1$ and $F_2$, \eqref{F1F2_def}, it can be checked that
\begin{equation*}
\begin{split}
|F_1| & \leq C \left( |\bar{v}^R_{\xi\xi}| + |\bar{v}^R_\xi|^2 + |\bar{v}^S_\xi(\bar{v}^R-v_m)| + |\bar{v}^R_\xi\bar{v}^S_\xi| \right), \\
|F_2| & \leq C \left( \lvert \bar{v}^R_\xi (\bar{v}^S-v_m) \rvert + |\bar{v}^S_\xi(\bar{v}^R-v_m) | \right).
\end{split}
\end{equation*}
It then follows
\begin{equation*}
\begin{split}
|\mathcal{U}^{(1)}_5| & \leq C \delta_R^{1/2} \lVert \tilde{u}_{\xi\xi} \rVert_{L^2}^2 + \frac{C}{\delta_R^{1/2}} \big( \lVert \bar{v}^R_\xi (\bar{v}^S-v_m) \rVert_{L^2}^2 + \|\bar{v}^S_\xi(\bar{v}^R-v_m) \|_{L^2}^2 \\
& \qquad + \lVert \bar{v}^R_\xi\bar{v}^S_\xi \rVert_{L^2}^2 + \lVert \bar{v}^R_\xi \rVert_{L^4}^4 + \lVert \bar{v}^R_{\xi\xi} \rVert_{L^2}^2 \big).
\end{split}
\end{equation*}
Similarly, since
\begin{equation*}
\begin{split}
|(F_5)_\xi| & = \bigg| \bigg( \frac{ \bar{\phi}_{\xi\xi}}{\bar{v}} - \frac{ \bar{\phi}^S_{\xi\xi}}{\bar{v}^S} - \frac{\bar{v}_\xi \bar{\phi}_\xi}{\bar{v}^2} + \frac{\bar{v}^S_\xi \bar{\phi}^S_\xi}{(\bar{v}^S)^2}  \bigg) \bigg| \leq \bigg| \frac{\bar{v}^S\bar{\phi}_{\xi\xi} - \bar{\phi}^S_{\xi\xi}\bar{v}}{\bar{v}\bar{v}^S} \bigg| + \bigg| \frac{(\bar{v}^S)^2 \bar{v}_\xi \bar{\phi}_\xi - \bar{v}^S_\xi \bar{\phi}^S_\xi \bar{v}^2}{\bar{v}^2 (\bar{v}^S)^2} \bigg| \\
& \leq C \left( |\bar{v}^R_{\xi\xi}| + |\bar{v}^S_\xi (\bar{v}^R-v_m)| + |\bar{v}^R_\xi (\bar{v}^S-v_m)| + |\bar{v}^R_\xi|^2 + |\bar{v}^R_\xi\bar{v}^S_\xi| \right),
\end{split}
\end{equation*}
we obtain
\begin{equation*}
\begin{split}
|\mathcal{U}^{(1)}_6| & \leq C \delta_R^{1/2} \lVert \tilde{u}_{\xi} \rVert_{L^2}^2 + \frac{C}{\delta_R^{1/2}} \left( \lVert \bar{v}^R_\xi (\bar{v}^S-v_m) \rVert_{L^2}^2 + \|\bar{v}^S_\xi(\bar{v}^R-v_m) \|_{L^2}^2 + \lVert \bar{v}^R_\xi\bar{v}^S_\xi \rVert_{L^2}^2 + \lVert \bar{v}^R_\xi \rVert_{L^4}^4 + \lVert \bar{v}^R_{\xi\xi} \rVert_{L^2}^2 \right).
\end{split}
\end{equation*}
Lastly, by Young's inequality,
\begin{equation*}
|\mathcal{U}^{(1)}_7| \leq C \int_\mathbb{R} |\tilde{u}_\xi| \left( |\bar{v}^R_{\xi\xi}| + |\bar{v}^R_\xi|^2 \right) \, d\xi \leq C \delta_R^{1/2} \lVert \tilde{u}_\xi \rVert_{L^2}^2 + \frac{C}{\delta_R^{1/2}} \left( \lVert \bar{v}^R_{\xi\xi} \rVert_{L^2}^2 + \lVert \bar{v}^R_\xi \rVert_{L^4}^4 \right)
\end{equation*}

Collecting all the estimates and using smallness of the parameters $\delta_S$, $\delta_R$, and $\varepsilon_1$, we have
\begin{equation*}
\begin{split}
& \frac{d}{dt} \int_\mathbb{R} \frac{\tilde{u}_\xi^2}{2} \, d\xi + \int_\mathbb{R} \frac{\tilde{u}_{\xi\xi}^2}{v} \, d\xi \\
& \quad \leq C (\delta_S + \delta_R^{1/2} + \varepsilon_1) \left( \delta_S |\dot{X}|^2 + \int_\mathbb{R} \left( \bar{v}_\xi \tilde{v}^2 + \tilde{v}_\xi^2 + \tilde{u}_\xi^2 \right) \, d\xi \right) + C \int_\mathbb{R} \left( \tilde{v}_\xi^2 + \tilde{\phi}_\xi^2 \right) \, d\xi \\
& \qquad + C \delta_R^{-1/2} \left( \lVert \bar{v}^R_\xi (\bar{v}^S-v_m) \rVert_{L^2}^2 + \|\bar{v}^S_\xi(\bar{v}^R-v_m) \|_{L^2}^2 + \lVert \bar{v}^R_\xi\bar{v}^S_\xi \rVert_{L^2}^2 + \lVert \bar{v}^R_\xi \rVert_{L^4}^4 + \lVert \bar{v}^R_{\xi\xi} \rVert_{L^2}^2 \right). 
\end{split}
\end{equation*}
Integrating this with respect in time in over $[0,t]$, and using \eqref{inter}--\eqref{Raest}, we obtain the result.
\end{proof}

\begin{proof}[Proof of Lemma~\ref{uxixi}]

Differentiating \eqref{eq_u} twice with respect to $\xi$, and multiplying the resulting equation by $\tilde{u}_{\xi\xi}$, we have
\begin{equation*}
\begin{split}
\bigg( \frac{\tilde{u}_{\xi\xi}^2}{2} \bigg)_t + \frac{\tilde{u}_{\xi\xi\xi}^2}{v} & = \left( \cdots \right)_\xi - \tilde{u}_{\xi\xi\xi} \left( \frac{\tilde{v}}{v\bar{v}} \right)_{\xi\xi} + \tilde{u}_{\xi\xi\xi} \left( \frac{\tilde{v}_\xi \tilde{u}_{\xi\xi}}{v^2} + \frac{\bar{v}_\xi \tilde{u}_{\xi\xi}}{v^2} \right) \\
& \quad + \tilde{u}_{\xi\xi\xi} \left( \frac{\tilde{v}_\xi \tilde{u}_\xi}{v^2} + \frac{\bar{v}_\xi \tilde{u}_\xi}{v^2} \right)_\xi + \tilde{u}_{\xi\xi\xi} \left( \frac{\bar{u}_\xi \tilde{v}}{v\bar{v}} \right)_{\xi\xi} + \tilde{u}_{\xi\xi\xi} \bigg( \frac{\tilde{\phi}_\xi}{v} - \frac{\bar{\phi}_\xi \tilde{v}}{v\bar{v}} \bigg)_{\xi} \\
& \quad + \dot{X}(t) \bar{u}^S_{\xi\xi\xi} \tilde{u}_{\xi\xi} + \tilde{u}_{\xi\xi\xi} \left( F_1 + \frac{F_2}{2} + F_5 \right)_{\xi} + \tilde{u}_{\xi\xi\xi} \bigg( \frac{\bar{v}^R_\xi}{(\bar{v}^R)^2} \bigg)_\xi.
\end{split}
\end{equation*}
Integrating this with respect to $\xi$, we have
\begin{equation*}
\frac{d}{dt} \int_\mathbb{R} \frac{\tilde{u}_{\xi\xi}^2}{2} \, d\xi + \int_\mathbb{R} \frac{\tilde{u}_{\xi\xi\xi}^2}{v} \, d\xi = \sum_{j=1}^6 \mathcal{U}^{(2)}_j,
\end{equation*}
where
\begin{equation*}
\begin{split}
\mathcal{U}^{(2)}_1 & := \int_\mathbb{R} \bigg( - \frac{\tilde{u}_{\xi\xi\xi}\tilde{v}_{\xi\xi}}{v\bar{v}} + \frac{\tilde{u}_{\xi\xi\xi}\tilde{\phi}_{\xi\xi}}{v} \bigg) \, d\xi , \quad \mathcal{U}^{(2)}_2 := \dot{X}(t) \int_\mathbb{R} \bar{u}^S_{\xi\xi\xi} \tilde{u}_{\xi\xi} \, d\xi, \\
\mathcal{U}^{(2)}_3 & := \int_\mathbb{R} \tilde{u}_{\xi\xi\xi} \bigg( \frac{2\tilde{v}_\xi^2}{v^2 \bar{v}} + \frac{\tilde{v}_{\xi\xi}\tilde{v}}{v^2 \bar{v}} - \frac{2\tilde{v}_\xi^2 \tilde{v}}{v^3 \bar{v}} - \frac{4\bar{v}_\xi \tilde{v}_\xi \tilde{v}}{v^3 \bar{v}} - \frac{2\bar{v}_\xi \tilde{v}_\xi \tilde{v}}{v^2(\bar{v})^2} \bigg) \, d\xi \\
& \quad + \int_\mathbb{R} \tilde{u}_{\xi\xi\xi} \bigg( \frac{2\tilde{v}_\xi \tilde{u}_{\xi\xi}}{v^2} + \frac{\tilde{v}_{\xi\xi}\tilde{u}_\xi}{v^2} - \frac{2\tilde{v}_\xi^2 \tilde{u}_\xi}{v^3} - \frac{4\bar{v}_\xi \tilde{v}_\xi\tilde{u}_\xi}{v^3}  \bigg) \, d\xi \\
& \quad + \int_\mathbb{R} \tilde{u}_{\xi\xi\xi} \bigg[ \frac{\bar{u}_\xi}{\bar{v}} \bigg( \frac{2\tilde{v}_\xi^2 \tilde{v}}{v^3} + \frac{4\bar{v}_\xi \tilde{v}_\xi\tilde{v}}{v^3} - \frac{2\tilde{v}_\xi^2}{v^2} - \frac{\tilde{v}_{\xi\xi}\tilde{v}}{v^2} \bigg) - \left( \frac{\bar{u}_\xi}{\bar{v}} \right)_\xi \frac{\tilde{v}_\xi\tilde{v}}{v^2} \bigg] \, d\xi \\
& \quad + \int_\mathbb{R} \tilde{u}_{\xi\xi\xi} \bigg( \frac{\bar{\phi}_\xi \tilde{v}_\xi \tilde{v}}{v^2\bar{v}} - \frac{\tilde{v}_\xi \tilde{\phi}_\xi}{v^2} \bigg) \, d\xi, \\
\mathcal{U}^{(2)}_4 & := \int_\mathbb{R} \tilde{u}_{\xi\xi\xi} \bigg( \frac{2\bar{v}_\xi \tilde{v}_\xi}{v^2\bar{v}} + \frac{2\bar{v}_\xi\tilde{v}_\xi}{v\bar{v}^2} - \frac{2\bar{v}_\xi^2 \tilde{v}}{v^3\bar{v}} - \frac{\bar{v}_{\xi\xi}\tilde{v}}{v^2 \bar{v}} + \frac{2\bar{v}_\xi^2 \tilde{v}}{v^2\bar{v}^2} - \left( \frac{\bar{v}_\xi}{\bar{v}^2} \right)_\xi \frac{\tilde{v}}{v} \bigg) \, d\xi \\
& \quad + \int_\mathbb{R} \tilde{u}_{\xi\xi\xi} \bigg( \frac{\bar{v}_\xi \tilde{u}_{\xi\xi}}{v^2} + \frac{\bar{v}_\xi \tilde{u}_{\xi\xi}}{v^2} - \frac{2\bar{v}_\xi^2 \tilde{u}_\xi}{v^3} + \frac{\bar{v}_{\xi\xi}\tilde{u}_\xi}{v^2} \bigg) \, d\xi \\
& \quad + \int_\mathbb{R} \tilde{u}_{\xi\xi\xi} \bigg[ \frac{\bar{u}_\xi}{\bar{v}} \bigg( \frac{\tilde{v}_{\xi\xi}}{v} - \frac{2\bar{v}_\xi\tilde{v}_\xi}{v^2} + \frac{2\bar{v}_\xi^2 \tilde{v}}{v^3} - \frac{\bar{v}_{\xi\xi}\tilde{v}}{v^2} \bigg) - \left( \frac{\bar{u}_\xi}{\bar{v}} \right)_\xi \frac{2\bar{v}_\xi \tilde{v}}{v^2} + \left( \frac{\bar{u}_\xi}{\bar{v}} \right)_{\xi\xi} \frac{\tilde{v}}{v} \bigg] \, d\xi \\
& \quad + \int_\mathbb{R} \tilde{u}_{\xi\xi\xi} \bigg( - \frac{\bar{v}_\xi \tilde{\phi}_\xi}{v^2}- \frac{\bar{\phi}_\xi\tilde{v}_\xi}{v\bar{v}} + \frac{\bar{v}_\xi \bar{\phi}_\xi \tilde{v}}{v^2\bar{v}} - \left( \frac{\bar{\phi}_\xi}{\bar{v}} \right)_\xi \frac{\tilde{v}}{v} \bigg) \, d\xi,
\end{split}
\end{equation*}
and
\begin{equation*}
\mathcal{U}^{(2)}_5 := \int_\mathbb{R} \tilde{u}_{\xi\xi\xi} \left( F_1 + \frac{F_2}{2} + F_5 \right)_{\xi} \, d\xi, \quad \mathcal{U}^{(2)}_6 := \int_\mathbb{R} \tilde{u}_{\xi\xi\xi} \bigg( \frac{\bar{v}^R_\xi}{(\bar{v}^R)^2} \bigg)_{\xi} \, d\xi.
\end{equation*}
The term $\mathcal{U}^{(2)}_1$ is estimated by applying Young's inequality as
\begin{equation*}
|\mathcal{U}^{(2)}_1| \leq C \int_\mathbb{R} |\tilde{u}_{\xi\xi\xi}| \left( |\tilde{v}_{\xi\xi}| + |\tilde{\phi}_{\xi\xi}| \right) \, d\xi \leq \theta \int_\mathbb{R} \tilde{u}_{\xi\xi\xi}^2 \, d\xi + C \theta^{-1} \int_\mathbb{R} \left( \tilde{v}_{\xi\xi}^2 + \tilde{\phi}_{\xi\xi}^2 \right) \, d\xi
\end{equation*}
for arbitrarily small $0 < \eta < 1$. We obtain the bound on $\mathcal{U}^{(2)}_2$ by using H\"older's and Young's inequality:
\begin{equation*}
\begin{split}
|\mathcal{U}^{(2)}_2| & \leq C |\dot{X}| \sqrt{\int_\mathbb{R} \bar{v}^S_\xi \, d\xi} \sqrt{\int_\mathbb{R} \bar{v}^S_\xi \tilde{u}_{\xi\xi} \, d\xi }  \leq C \delta_S^2 |\dot{X}|^2 + \frac{C}{\delta_S^2} \left( \int_\mathbb{R} \bar{v}^S_\xi \, d\xi \right) \left( \int_\mathbb{R} \bar{v}^S_\xi \tilde{u}_{\xi\xi}^2 \, d\xi \right) \\
& \leq C \delta_S \left( \delta_S |\dot{X}|^2 + \int_\mathbb{R} \tilde{u}_{\xi\xi}^2 \, d\xi \right).
\end{split}
\end{equation*}
The nonlinear term $\mathcal{U}^{(2)}_3$ is estimated as
\begin{equation*}
\begin{split}
|\mathcal{U}^{(2)}_3| & \leq C \int_\mathbb{R} |\tilde{v}| |\tilde{u}_{\xi\xi\xi}| \left( |\tilde{v}_{\xi\xi}| + |\tilde{v}_\xi| + |\tilde{v}_\xi|^2 \right) \, d\xi + C \int_\mathbb{R} |\tilde{u}_\xi| |\tilde{u}_{\xi\xi\xi} || \tilde{v}_{\xi\xi}| \, d\xi \\
& \quad + C \int_\mathbb{R} |\tilde{v}_{\xi} ||\tilde{u}_{\xi\xi\xi}| \left( |\tilde{v}_\xi| + |\tilde{u}_{\xi\xi}| + |\tilde{u}_\xi| + |\tilde{v}_\xi||\tilde{u}_\xi| + |\tilde{\phi}_\xi| \right) \, d\xi \\
& \leq C \left( \delta_R + \varepsilon_1 \right) \int_\mathbb{R} \left( \tilde{v}_\xi^2 + \tilde{v}_{\xi\xi}^2 + \tilde{u}_\xi^2 + \tilde{u}_{\xi\xi}^2 + \tilde{u}_{\xi\xi\xi}^2 + \tilde{\phi}_\xi^2 \right) \, d\xi.
\end{split}
\end{equation*}
For $\mathcal{U}^{(2)}_4$, we use \eqref{shderiv1}--\eqref{raderiv} to obtain
\begin{equation*}
\begin{split}
|\mathcal{U}^{(2)}_4| & \leq C \int_\mathbb{R} |\bar{v}_\xi| |\tilde{u}_{\xi\xi\xi}||\tilde{v}| \, d\xi + C \int_\mathbb{R} | \bar{v}_\xi| |\tilde{u}_{\xi\xi\xi}| \left( |\tilde{v}_{\xi\xi}| + |\tilde{v}_\xi| + |\tilde{u}_{\xi\xi}| + |\tilde{u}_\xi| + |\tilde{\phi}_\xi| \right) \, d\xi \\
& \leq C \int_\mathbb{R} |\bar{v}_\xi|^{1/2} |\tilde{u}_{\xi\xi\xi}|^2 \, d\xi + C \int_\mathbb{R} |\bar{v}_\xi|^{3/2} |\tilde{v}|^2 \, d\xi \\
& \quad + C \int_\mathbb{R} |\bar{v}_\xi| \left( |\tilde{u}_{\xi\xi\xi}|^2 + |\tilde{u}_{\xi\xi}|^2 + |\tilde{u}_\xi|^2 + |\tilde{v}_{\xi\xi}|^2 + |\tilde{v}_\xi|^2 + |\tilde{\phi}_\xi|^2 \right) \, d\xi \\
& \leq C (\delta_S + \delta_R^{1/2}) \int_\mathbb{R} \left( \bar{v}_\xi \tilde{v}^2 + \tilde{v}_\xi^2 + \tilde{v}_{\xi\xi}^2 + \tilde{u}_\xi^2 + \tilde{u}_{\xi\xi}^2 + \tilde{u}_{\xi\xi\xi}^2 + \tilde{\phi}_\xi^2 \right) \, d\xi.
\end{split}
\end{equation*}
We observe that, by \eqref{shderiv1}--\eqref{raderiv},
\begin{equation*}
\begin{split}
& |(F_1)_\xi| \leq C \left( |\bar{v}^R_{\xi\xi}| + |\bar{v}^R_\xi|^2 + |\bar{v}^S_\xi(\bar{v}^R-v_m)| + |\bar{v}^R_\xi\bar{v}^S_\xi| \right), \\
& |(F_2)_\xi| \leq C \left( |\bar{v}^R_\xi(\bar{v}^S-v_m)| + |\bar{v}^S_\xi(\bar{v}^R-v_m)| + |\bar{v}^R_\xi\bar{v}^S_\xi| \right),
\end{split}
\end{equation*}
and, as shown in the proof of Lemma~\ref{uxi},
\begin{equation*}
\begin{split}
|(F_5)_\xi| & \leq C \left( |\bar{v}^R_{\xi\xi}| + |\bar{v}^S_\xi (\bar{v}^R-v_m)| + |\bar{v}^R_\xi (\bar{v}^S-v_m)| + |\bar{v}^R_\xi|^2 + |\bar{v}^R_\xi\bar{v}^S_\xi| \right).
\end{split}
\end{equation*}
Thus, applying Young's inequality, we obtain
\begin{equation*}
\begin{split}
|\mathcal{U}^{(2)}_5| & \leq C \delta_R^{1/2} \lVert \tilde{u}_{\xi\xi\xi} \rVert_{L^2}^2 + \frac{C}{\delta_R^{1/2}} \big( \lVert \bar{v}^R_\xi (\bar{v}^S-v_m) \rVert_{L^2}^2 + \|\bar{v}^S_\xi(\bar{v}^R-v_m) \|_{L^2}^2 \\
& \qquad + \lVert \bar{v}^R_\xi\bar{v}^S_\xi \rVert_{L^2}^2 + \lVert \bar{v}^R_\xi \rVert_{L^4}^4 + \lVert \bar{v}^R_{\xi\xi} \rVert_{L^2}^2 \big).
\end{split}
\end{equation*}
Similarly, we have
\begin{equation*}
|\mathcal{U}^{(2)}_6| \leq  C \delta_R^{1/2} \lVert \tilde{u}_{\xi\xi\xi} \rVert_{L^2}^2 + C \delta_R^{-1/2} \left( \lVert \bar{v}^R_\xi \rVert_{L^4}^4 + \lVert \bar{v}^R_{\xi\xi} \rVert_{L^2}^2 \right).
\end{equation*}
Collecting all the estimates and using smallness of the parameters $\delta_0$ and $\varepsilon_1$, we have
\begin{equation*}
\begin{split}
& \frac{d}{dt} \int_\mathbb{R} \frac{\tilde{u}_{\xi\xi}^2}{2} \, d\xi + \int_\mathbb{R} \frac{\tilde{u}_{\xi\xi\xi}^2}{v} \, d\xi \\
& \quad \leq C (\delta_S + \delta_R^{1/2} + \varepsilon_1) \left( \delta_S |\dot{X}|^2 + \int_\mathbb{R} \left( \bar{v}_\xi \tilde{v}^2 + \tilde{v}_\xi^2 + \tilde{v}_{\xi\xi}^2 + \tilde{u}_\xi^2 + \tilde{u}_{\xi\xi}^2 + \tilde{\phi}_\xi^2 \right) \, d\xi \right) + C \int_\mathbb{R} \left( \tilde{v}_{\xi\xi}^2 + \tilde{\phi}_{\xi\xi}^2 \right) \, d\xi \\
& \qquad + C \delta_R^{-1/2} \left( \lVert \bar{v}^R_\xi (\bar{v}^S-v_m) \rVert_{L^2}^2 + \|\bar{v}^S_\xi(\bar{v}^R-v_m) \|_{L^2}^2 + \lVert \bar{v}^R_\xi\bar{v}^S_\xi \rVert_{L^2}^2 + \lVert \bar{v}^R_\xi \rVert_{L^4}^4 + \lVert \bar{v}^R_{\xi\xi} \rVert_{L^2}^2 \right). 
\end{split}
\end{equation*}
Integrating this in time over $[0,t]$, and using \eqref{inter}--\eqref{Raest}, we obtain the desired bound.

\end{proof}

\begin{proof}[Proof of Lemma~\ref{vxixi}]

Differentiating \eqref{eq_v} twice with respect to $\xi$ and multiplying the resulting equation by $\tilde{v}_{\xi\xi}$, we have
\begin{equation} \label{vxxvxx}
\bigg( \frac{\tilde{v}_{\xi\xi}^2}{2} \bigg)_t - \tilde{v}_{\xi\xi} \tilde{u}_{\xi\xi\xi} = (\cdots)_{\xi} + \dot{X}(t) \bar{v}^S_{\xi\xi\xi} \tilde{v}_{\xi\xi}.
\end{equation} 
Differentiating \eqref{eq_u} with respect to $\xi$ and multiplying the resulting equation by $-v\tilde{v}_{\xi\xi}$, we have
\begin{equation} \label{vxxux}
\begin{split}
& -v\tilde{v}_{\xi\xi}\tilde{u}_{\xi t} - v\tilde{v}_{\xi\xi}\bigg( \frac{1}{v}-\frac{1}{\bar{v}} \bigg)_{\xi\xi} + v \tilde{v}_{\xi\xi} \left( \frac{u_\xi}{v} - \frac{\bar{u}_\xi}{v} \right)_{\xi\xi} - v\tilde{v}_{\xi\xi}\bigg( \frac{\phi_\xi}{v} - \frac{\bar{\phi}_\xi}{\bar{v}} \bigg)_\xi\\
& \quad = - \dot{X}(t) v\bar{u}^S_{\xi\xi} \tilde{v}_{\xi\xi} + v \tilde{v}_{\xi\xi} \left( F_1 + \frac{F_2}{2} + F_5 \right)_\xi + v \tilde{v}_{\xi\xi} \bigg( \frac{\bar{v}^R_\xi}{(\bar{v}^R)^2} \bigg)_\xi.
\end{split}
\end{equation}
Notice that each term on the left-hand side can be rewritten as follows:
\begin{equation*}
\begin{split}
-v\tilde{v}_{\xi\xi}\tilde{u}_{\xi t} & = (-v\tilde{v}_{\xi\xi}\tilde{u}_\xi)_t + v_t \tilde{v}_{\xi\xi}\tilde{u}_\xi +  v\tilde{v}_{\xi\xi t} \tilde{u}_\xi \\
& = (-v\tilde{v}_{\xi\xi}\tilde{u}_\xi)_t + \sigma v_\xi \tilde{v}_{\xi\xi} \tilde{u}_\xi + u_\xi \tilde{v}_{\xi\xi}\tilde{u}_\xi + \sigma v \tilde{v}_{\xi\xi\xi}\tilde{u}_\xi + v \tilde{u}_{\xi\xi\xi} \tilde{u}_\xi + \dot{X}(t) v\bar{v}^S_{\xi\xi\xi} \tilde{u}_\xi \\
& = (-v\tilde{v}_{\xi\xi}\tilde{u}_\xi)_t  + (\cdots)_\xi + u_\xi \tilde{v}_{\xi\xi}\tilde{u}_\xi - \sigma v \tilde{v}_{\xi\xi}\tilde{u}_{\xi\xi} - v_\xi \tilde{u}_{\xi\xi}\tilde{u}_\xi - v \tilde{u}_{\xi\xi}^2 + \dot{X}(t) v\bar{v}^S_{\xi\xi\xi} \tilde{u}_\xi, \\
- v \tilde{v}_{\xi\xi} \left( \frac{1}{v}-\frac{1}{\bar{v}} \right) & = \frac{\tilde{v}_{\xi\xi}^2}{v} - \frac{2\tilde{v}_{\xi}^2 \tilde{v}_{\xi\xi}}{v^2} - v \tilde{v}_{\xi\xi} \left( \frac{\bar{v}_\xi\tilde{v}^2}{v^2\bar{v}^2} + \frac{2\bar{v}_\xi\tilde{v}}{v^2\bar{v}} \right)_\xi, \\
v \tilde{v}_{\xi\xi} \left( \frac{u_\xi}{v} - \frac{\bar{u}_\xi}{v} \right)_{\xi\xi} & = \tilde{v}_{\xi\xi} \tilde{u}_{\xi\xi\xi} - \frac{v_\xi \tilde{v}_{\xi\xi}\tilde{u}_{\xi\xi}}{v} - v \tilde{v}_{\xi\xi} \left( \frac{\bar{u}_\xi \tilde{v}}{v\bar{v}} - \left( \frac{v_\xi u_\xi}{v^2} - \frac{\bar{v}_\xi \bar{u}_\xi}{\bar{v}^2} \right) \right)_\xi,
\end{split}
\end{equation*}
and
\begin{equation*}
- v\tilde{v}_{\xi\xi}\left( \frac{\phi_\xi}{v} - \frac{\bar{\phi}_\xi}{\bar{v}} \right)_\xi  = - \tilde{v}_{\xi\xi}\tilde{\phi}_{\xi\xi} + \frac{\bar{\phi}_{\xi\xi}\tilde{v}\tilde{v}_{\xi\xi}}{\bar{v}} + v\tilde{v}_{\xi\xi} \left( \frac{v_\xi \phi_\xi}{v^2} - \frac{\bar{v}_\xi \bar{\phi}_\xi}{\bar{v}^2} \right).
\end{equation*}
With this, summing \eqref{vxxvxx}--\eqref{vxxux} and integrating the resulting equation with respect to $\xi$, we have
\begin{equation*}
\frac{d}{dt} \int_\mathbb{R} \bigg( \frac{\tilde{v}_{\xi\xi}^2}{2} - v\tilde{v}_{\xi\xi} \tilde{u}_\xi \bigg) \, d\xi + \int_\mathbb{R} \frac{\tilde{v}_{\xi\xi}^2}{v} \, d\xi = \sum_{j=1}^6 \mathcal{V}^{(2)}_j,
\end{equation*}
where
\begin{equation*}
\begin{split}
\mathcal{V}^{(2)}_1 & := \int_\mathbb{R} \left( \sigma v \tilde{v}_{\xi\xi}\tilde{u}_{\xi\xi} + v \tilde{u}_{\xi\xi}^2 + \tilde{v}_{\xi\xi}\tilde{\phi}_{\xi\xi} \right) \, d\xi, \quad \mathcal{V}^{(2)}_2 := \dot{X}(t) \int_\mathbb{R} \left( \bar{v}^S_{\xi\xi\xi} \tilde{v}_{\xi\xi} - v\bar{u}^S_{\xi\xi}\tilde{v}_{\xi\xi} - v\bar{v}^S_{\xi\xi\xi}\tilde{u}_\xi \right) \, d\xi, \\
\mathcal{V}^{(2)}_3 & := \int_\mathbb{R} \tilde{u}_\xi \left( - \tilde{v}_{\xi\xi} \tilde{u}_\xi + \tilde{v}_\xi \tilde{u}_{\xi\xi} \right) \, d\xi + \int_\mathbb{R} \tilde{v}_{\xi\xi} \bigg( \frac{2\tilde{v}_\xi^2 }{v^2} + \left( \frac{\bar{v}_\xi}{\bar{v}^2} \right)_\xi \frac{\tilde{v}^2}{v} + \frac{2\bar{v}_\xi \tilde{v}\tilde{v}_\xi}{v\bar{v}^2} - \frac{2\bar{v}_\xi v_\xi \tilde{v}^2}{v^2\bar{v}^2} - \frac{2\bar{v}_\xi\tilde{v}_\xi\tilde{v}}{v^2\bar{v}} \bigg) \, d\xi \\
& \quad + \int_\mathbb{R} \tilde{v}_{\xi\xi} \left( - \frac{\tilde{v}_{\xi\xi}\tilde{u}_\xi}{v} + \frac{2v_\xi \tilde{v}_\xi\tilde{u}_\xi}{v^2} \right) \, d\xi  + \int_\mathbb{R} \tilde{v}_{\xi\xi} \bigg( - \frac{\bar{\phi}_{\xi\xi}\tilde{v}\tilde{v}_{\xi\xi}}{\bar{v}} - \frac{\tilde{v}_\xi\tilde{\phi}_\xi}{v} + \frac{\bar{v}_\xi\bar{\phi}_\xi \tilde{v}^2}{v\bar{v}^2} \bigg) \, d\xi, \\
\mathcal{V}^{(2)}_4 & := \int_\mathbb{R} \tilde{u}_\xi \left( - \bar{u}_\xi \tilde{v}_{\xi\xi} + \bar{v}_\xi \tilde{u}_{\xi\xi} \right) \, d\xi +  \int_\mathbb{R} \tilde{v}_{\xi\xi} \bigg( - \bar{u}_\xi \tilde{u}_\xi + \left( \frac{2\bar{v}_\xi}{\bar{v}} \right)_\xi \frac{\tilde{v}}{v} + \frac{2\bar{v}_\xi \tilde{v}_\xi}{v\bar{v}} - \frac{2\bar{v}_\xi^2\tilde{v}}{v^2\bar{v}} \bigg) \, d\xi \\
& \quad + \int_\mathbb{R} v \tilde{v}_{\xi\xi} \bigg(  \frac{\bar{v}_{\xi}\tilde{u}_{\xi\xi}}{v^2} + \left( \frac{\bar{u}_\xi \tilde{v}}{v\bar{v}} - \frac{ \bar{v}_{\xi}\tilde{u}_{\xi} + \bar{u}_\xi \tilde{v}_{\xi}}{v^2} +  \frac{\bar{v}_\xi\bar{u}_\xi (\tilde{v}^2 + 2\bar{v}\tilde{v})}{v^2\bar{v}^2} \right)_\xi \bigg) \, d\xi \\
& \quad + \int_\mathbb{R} \tilde{v}_{\xi\xi} \bigg( \frac{\bar{v}_\xi \tilde{\phi}_\xi + \bar{\phi}_\xi \tilde{v}_\xi}{v} + \frac{2\bar{v}_\xi \bar{\phi}_\xi \tilde{v}}{v\bar{v}} \bigg) \, d\xi, \\
\mathcal{V}^{(2)}_5 & := \int_\mathbb{R} v \tilde{v}_{\xi\xi} \left( F_1 + \frac{F_2}{2} + F_5 \right)_\xi \, d\xi, \quad \mathcal{V}^{(2)}_6 := \int_\mathbb{R} v \tilde{v}_{\xi\xi} \bigg( \frac{\bar{v}^R_\xi}{(\bar{v}^R)^2} \bigg)_\xi \, d\xi.
\end{split}
\end{equation*}
Applying Young's inequality, we have
\begin{equation*}
\begin{split}
|\mathcal{V}^{(2)}_1| & \leq C \int_\mathbb{R} |\tilde{v}_{\xi\xi}| \left( |\tilde{u}_{\xi\xi}| + |\tilde{\phi}_{\xi\xi}| \right) \, d\xi + C \int_\mathbb{R} |\tilde{u}_{\xi\xi}|^2 \, d\xi \\
& \leq \theta \int_\mathbb{R} \tilde{v}_{\xi\xi}^2 \, d\xi + C \theta^{-1} \int_\mathbb{R} \left( \tilde{u}_{\xi\xi}^2 + \tilde{\phi}_{\xi\xi}^2 \right) \, d\xi + C \int_\mathbb{R} \tilde{u}_{\xi\xi}^2 \, d\xi
\end{split}
\end{equation*}
for any constant $0 < \theta <1$. Next we use H\"older's inequality together with the bound $|\bar{v}^S_{\xi\xi\xi}|,|\bar{u}^S_{\xi\xi}| \leq C |\bar{v}^S_\xi|$ to obtain
\begin{equation*}
\begin{split}
|\mathcal{V}^{(2)}_2| & \leq C |\dot{X}| \sqrt{\int_\mathbb{R} \bar{v}^S_\xi \, d\xi} \left( \sqrt{\int_\mathbb{R} \bar{v}^S_\xi \tilde{v}_{\xi\xi}^2 \, d\xi} + \sqrt{\int_\mathbb{R} \bar{v}^S_\xi \tilde{u}_\xi^2 \, d\xi} \right) \\
& \leq C \delta_S^2 |\dot{X}|^2 + \frac{C}{\delta_S} \left( \int_\mathbb{R} \bar{v}^S_\xi \, d\xi \right) \left( \int_\mathbb{R} \bar{v}^S_\xi \tilde{v}_{\xi\xi}^2 \, d\xi + \int_\mathbb{R} \bar{v}^S_\xi \tilde{u}_\xi^2 \, d\xi \right) \\
& \leq C \delta_S \left( \delta_S |\dot{X}|^2 + \int_\mathbb{R} \left( \tilde{v}_{\xi\xi}^2 + \tilde{u}_\xi^2 \right) \, d\xi \right).
\end{split}
\end{equation*}
The nonlinear term $\mathcal{V}^{(2)}_3$ is estimated as
\begin{equation*}
\begin{split}
|\mathcal{V}^{(2)}_3| & \leq C \int_\mathbb{R} |\tilde{v}_{\xi\xi}| \left( |\tilde{u}_\xi|^2 + |\tilde{v}_\xi|^2 + |\bar{v}_\xi||\tilde{v}|^2 + |\tilde{v}||\tilde{v}_\xi| + |\tilde{v}_\xi||\tilde{u}_\xi| \right) \, d\xi \\
& \quad + C \int_\mathbb{R} |\tilde{v}_{\xi\xi}| \left( |\tilde{v}_{\xi\xi}||\tilde{u}_\xi| + |\tilde{v}||\tilde{v}_{\xi\xi}| + |\tilde{v}_\xi||\tilde{\phi}_\xi| \right) \, d\xi \\
& \leq C (\delta_R + \varepsilon_1) \int_\mathbb{R} \left( \bar{v}_\xi\tilde{v}^2 + \tilde{v}_\xi^2 + \tilde{v}_{\xi\xi}^2 + \tilde{u}_\xi^2 + \tilde{\phi}_\xi^2 \right) \, d\xi.
\end{split}
\end{equation*}
Using \eqref{shderiv1}--\eqref{raderiv}, we estimate
\begin{equation*}
\begin{split}
|\mathcal{V}^{(2)}_4| & \leq C \int_\mathbb{R} |\bar{v}_\xi| |\tilde{v}_{\xi\xi}| \left(  |\tilde{v}| + |\tilde{v}_\xi| + |\tilde{v}_{\xi\xi}| +  |\tilde{u}_\xi| + |\tilde{u}_{\xi\xi}| + |\tilde{u}_{\xi\xi\xi}| + |\tilde{\phi}_\xi| \right)  \, d\xi \\
& \leq C \int_\mathbb{R} |\bar{v}_\xi|^{1/2} |\tilde{v}_{\xi\xi}|^2 \, d\xi + C \int_\mathbb{R} |\bar{v}_\xi|^{3/2} |\tilde{v}|^2 \, d\xi \\
& \quad + C \int_\mathbb{R}|\bar{v}_\xi| \left( |\tilde{v}_\xi|^2 + |\tilde{v}_{\xi\xi}|^2 + |\tilde{u}_\xi|^2 + |\tilde{u}_{\xi\xi}|^2 + |\tilde{u}_{\xi\xi\xi}|^2 + |\tilde{\phi}_\xi|^2 \right) \, d\xi \\
& \leq C (\delta_S + \delta_R^{1/2} ) \int_\mathbb{R} \left( \bar{v}_\xi \tilde{v}^2 + \tilde{v}_\xi^2 + \tilde{v}_{\xi\xi}^2 + \tilde{u}_\xi^2 + \tilde{u}_{\xi\xi}^2 + \tilde{u}_{\xi\xi\xi}^2 + \tilde{\phi}_\xi^2  \right) \, d\xi.
\end{split}
\end{equation*}
Finally, as in the proof of Lemma~\ref{uxixi}, we obtain
\begin{equation*}
\begin{split}
|\mathcal{V}^{(2)}_5| & \leq C \int_\mathbb{R} |\tilde{v}_{\xi\xi}| \left( |(F_1)_\xi| + |(F_2)_\xi| +|(F_5)_\xi| \right) \, d\xi \\
& \leq C \delta_R^{1/2} \lVert \tilde{v}_{\xi\xi} \rVert_{L^2}^2 + C \delta_R^{-1/2} \left( \lVert \bar{v}^R_\xi (\bar{v}^S-v_m) \rVert_{L^2}^2 + \|\bar{v}^S_\xi(\bar{v}^R-v_m) \|_{L^2}^2 + \lVert \bar{v}^R_\xi\bar{v}^S_\xi \rVert_{L^2}^2 + \lVert \bar{v}^R_\xi \rVert_{L^4}^4 + \lVert \bar{v}^R_{\xi\xi} \rVert_{L^2}^2 \right), \\
|\mathcal{V}^{(2)}_6| & \leq  C \delta_R^{1/2} \lVert \tilde{v}_{\xi\xi} \rVert_{L^2}^2 + C \delta_R^{-1/2} \left( \lVert \bar{v}^R_\xi \rVert_{L^4}^4 + \lVert \bar{v}^R_{\xi\xi} \rVert_{L^2}^2 \right).
\end{split}
\end{equation*}

Combining the above estimates and taking $\theta$ sufficiently small, we obtain
\begin{equation*}
\begin{split}
& \frac{d}{dt} \int_\mathbb{R} \bigg( \frac{\tilde{v}_{\xi\xi}^2}{2} - v\tilde{v}_{\xi\xi}\tilde{u}_\xi \bigg) \, d\xi + \int_\mathbb{R} \tilde{v}_{\xi\xi}^2 \, d\xi \\
& \quad \leq C \int_\mathbb{R} \left( \tilde{u}_{\xi\xi}^2 + \tilde{\phi}_{\xi\xi}^2 \right) \, d\xi + C ( \delta_0^{1/2} + \varepsilon_1 ) \int_\mathbb{R} \left( \bar{v}_\xi \tilde{v}^2 + \tilde{v}_\xi^2 + \tilde{\phi}_\xi^2 \right) \, d\xi \\
& \qquad + C \delta_R^{-1/2} \left( \lVert \bar{v}^R_\xi (\bar{v}^S-v_m) \rVert_{L^2}^2 + \|\bar{v}^S_\xi(\bar{v}^R-v_m) \|_{L^2}^2 + \lVert \bar{v}^R_\xi\bar{v}^S_\xi \rVert_{L^2}^2 + \lVert \bar{v}^R_\xi \rVert_{L^4}^4 + \lVert \bar{v}^R_{\xi\xi} \rVert_{L^2}^2 \right)
\end{split}
\end{equation*}
for sufficiently small $\delta_S+\delta_R$ and $\varepsilon_1$. Integrating this in time over $[0,t]$, and applying \eqref{inter}--\eqref{Raest}, we have
\begin{equation*}
\begin{split}
& \lVert \tilde{v}_{\xi\xi}(t,\cdot) \rVert_{L^2}^2 + \int_0^t \lVert \tilde{v}_{\xi\xi} \rVert_{L^2}^2 \, d\tau \\
& \quad \leq C \lVert \tilde{v}_{0\xi\xi} \rVert_{L^2}^2 + \bigg[ \int_\mathbb{R} v\tilde{v}_{\xi\xi}\tilde{u}_\xi \, d\xi \bigg]_{t=0}^{t=t} + C \int_0^t \left( \lVert \tilde{u}_{\xi\xi} \rVert_{L^2}^2 + \lVert \tilde{\phi}_{\xi\xi} \rVert_{L^2}^2 \right) \, d\tau \\
& \qquad + C(\delta_0^{1/2} + \varepsilon_1) \int_0^t \left( \bar{v}_\xi \tilde{v}^2 + \tilde{v}_\xi^2 + \tilde{\phi}_\xi^2 \right) \, d\tau + C \delta_R.
\end{split}
\end{equation*}
Therefore, the application of Young's inequality to the second term on the right-hand side, together with \eqref{vL2}, \eqref{vxiL2}, and \eqref{phi_est}, yields \eqref{vxixi_ineq}.

\end{proof}

\section{Global-in-time existence of perturbations} \label{Sec_Ge}

We choose smooth functions $\underline{v}$, $\underline{u}$, and $\underline{\phi}$ satisfying
\begin{equation*}
\begin{split}
\sum_{\pm} & \lVert (\underline{v} - v_\pm, \underline{u} - u_\pm, \underline{\phi} - \phi_\pm )\rVert_{L^2(\mathbb{R}_\pm)} \\
& + \lVert \underline{v}_\xi \rVert_{H^1(\mathbb{R})} + \lVert \underline{u}_\xi \rVert_{H^1(\mathbb{R})} + \lVert \underline{\phi}_\xi \rVert_{H^2(\mathbb{R})} < \delta_0.
\end{split}
\end{equation*}
Then, by Lemma~\ref{rarefaction} and Lemma~\ref{Prop.1.1}, there exists a constant $C_*>0$ such that
\begin{equation*}
\begin{split}
& \lVert \underline{v} - \bar{v}(0,\cdot) \rVert_{H^2(\mathbb{R})} + \lVert \underline{u} - \bar{u}(0,\cdot) \rVert_{H^2(\mathbb{R})} + \lVert \underline{\phi} - \bar{\phi}(0,\cdot) \rVert_{H^2(\mathbb{R})} \\
& \quad \leq \sum_{\pm} \lVert (\underline{v} - v_\pm, \underline{u} - u_\pm, \underline{\phi} - \phi_\pm )\rVert_{L^2(\mathbb{R}_\pm)} + \lVert \underline{v}_\xi \rVert_{H^1(\mathbb{R})} + \lVert \underline{u}_\xi \rVert_{H^1(\mathbb{R})} + \lVert \underline{\phi}_\xi \rVert_{H^2(\mathbb{R})} \\
& \qquad + \lVert (\bar{v}^R-v_m, \bar{u}^R-u_m, \bar{\phi}^R - \phi_m) \rVert_{L^2(\mathbb{R}_+)} + \lVert (\bar{v}^S-v_+, \bar{u}^S-u_+, \bar{\phi}^S - \phi_+) \rVert_{L^2(\mathbb{R}_+)} \\
& \qquad + \lVert (\bar{v}^R-v_-, \bar{u}^R-u_-, \bar{\phi}^R - \phi_-) \rVert_{L^2(\mathbb{R}_-)} + \lVert (\bar{v}^S-v_m, \bar{u}^S-u_m, \bar{\phi}^S - \phi_m) \rVert_{L^2(\mathbb{R}_-)} \\
& \qquad + \lVert \bar{v}_\xi \rVert_{H^1(\mathbb{R})} + \lVert \bar{u}_\xi \rVert_{H^1(\mathbb{R})} + \lVert \bar{\phi}_\xi \rVert_{H^2(\mathbb{R})} \\
& \quad \leq \delta_0 + C (\delta_R + \sqrt{\delta_S}) \leq C_* \sqrt{\delta_0}
\end{split}
\end{equation*}
at $t=0$. Given any $\varepsilon_1>0$ and sufficiently small $\delta_0>0$, we choose $\varepsilon_0 > 0$ as
\begin{equation*}
\varepsilon_0 < \frac{\varepsilon_1}{3} - \delta_0 - C_* \sqrt{\delta_0}.
\end{equation*}
Consider the initial data $(v_0,u_0)$ satisfying
\begin{equation} \label{initcond}
\sum_{\pm} \lVert (v_0 - v_\pm, u_0 - u_\pm) \rVert_{L^2(\mathbb{R}_\pm)}  + \lVert (v_{0\xi},u_{0\xi}) \rVert_{H^1(\mathbb{R})} < \varepsilon_0,
\end{equation}
where we write the initial data in terms of \(\xi\), noting that \(x = \xi\) at \(t = 0\). Then we obtain
\begin{equation*}
\begin{split}
& \lVert v_0 - \underline{v} \rVert_{H^2(\mathbb{R})} + \lVert u_0 - \underline{u} \rVert_{H^2(\mathbb{R})}, \\
& \quad \leq \sum_{\pm} \left( \lVert (v_0-v_\pm, u_0-u_\pm) \rVert_{L^2(\mathbb{R}_\pm)} + \lVert (\underline{v} - v_\pm, \underline{u} - u_\pm) \rVert_{L^2(\mathbb{R}_\pm)} \right) \\
& \qquad + \lVert \underline{v}_\xi \rVert_{H^1(\mathbb{R})} + \lVert \underline{u}_\xi \rVert_{H^1(\mathbb{R})}  + \lVert v_{0\xi} \rVert_{H^1(\mathbb{R})}  + \lVert u_{0\xi}\rVert_{H^1(\mathbb{R})} \\
& \quad \leq \varepsilon_0 + \delta_0 < \frac{\varepsilon_1}{3}.
\end{split}
\end{equation*}
By the result of Proposition~\ref{LocalE}, there exists $T_0>0$ such that
\begin{equation} \label{Dref}
\lVert (v - \underline{v}, u - \underline{u} ) \rVert_{L^\infty(0,T_0;H^2(\mathbb{R}))} + \lVert \phi - \underline{\phi} \rVert_{L^\infty(0,T_0;H^3(\mathbb{R}))} \leq \frac{\varepsilon_1}{2},
\end{equation}
and so, $\tfrac{v_-}{3} < v < 3v_+$ for sufficiently small $\varepsilon_1$. For $t \in (0,T_0]$, we use the results of Lemmas~\ref{rarefaction} and \ref{Prop.1.1} to compute
\begin{equation*}
\begin{split}
& \lVert \underline{v} - \bar{v}(t,\cdot) \rVert_{H^2(\mathbb{R})} + \lVert \underline{u} - \bar{u}(t,\cdot)  \rVert_{H^2(\mathbb{R})} + \lVert \underline{\phi} - \bar{\phi}(t,\cdot)  \rVert_{H^3(\mathbb{R})} \\
& \quad \leq \sum_{\pm} \left( \lVert \underline{v} - v_\pm \rVert_{L^2(\mathbb{R}_\pm)} + \lVert \underline{u} - u_\pm \rVert_{L^2(\mathbb{R}_\pm)} + \lVert \underline{\phi} - \phi_\pm \rVert_{L^2(\mathbb{R}_\pm)} \right) \\
& \qquad + \lVert (\bar{v}^R-v_m, \bar{u}^R-u_m, \bar{\phi}^R - \phi_m)(t,\cdot + \sigma t) \rVert_{L^2(\mathbb{R}_+)} \\
& \qquad + \lVert (\bar{v}^S-v_+, \bar{u}^S-u_+, \bar{\phi}^S - \phi_+)(\cdot - X(t)) \rVert_{L^2(\mathbb{R}_+)} \\
& \qquad + \lVert (\bar{v}^R-v_-, \bar{u}^R-u_-, \bar{\phi}^R - \phi_-)(t,\cdot + \sigma t) \rVert_{L^2(\mathbb{R}_-)} \\
& \qquad + \lVert (\bar{v}^S-v_m, \bar{u}^S-u_m, \bar{\phi}^S - \phi_m)(\cdot - X(t)) \rVert_{L^2(\mathbb{R}_-)} \\
& \qquad + \lVert \underline{v}_\xi \rVert_{H^1(\mathbb{R})} + \lVert \underline{u}_\xi \rVert_{H^1(\mathbb{R})} + \lVert \underline{\phi}_\xi \rVert_{H^2(\mathbb{R})} + \lVert \bar{v}_\xi \rVert_{H^1(\mathbb{R})} + \lVert \bar{u}_\xi \rVert_{H^1(\mathbb{R})} + \lVert \bar{\phi}_\xi \rVert_{H^2(\mathbb{R})} \\
& \quad \leq C \delta_R \sqrt{1 + (\sigma - \lambda_1(v_-))t} + C \sqrt{\delta_S} \sqrt{1 + |X(t)| }  \leq C \sqrt{\delta_0}(1+ \sqrt{t}).
\end{split}
\end{equation*}
For related computations, we refer to \cite[Section~3.5]{KVW2}. Taking $T \in (0,T_0)$ small enough so that $\textstyle C \sqrt{\delta_0}(1+\sqrt{T}) \leq \frac{\varepsilon_1}{2}$, we have
\begin{equation} \label{refC}
\lVert \underline{v} - \bar{v} \rVert_{L^\infty(0,T;H^2(\mathbb{R}))} + \lVert \underline{u} - \bar{u} \rVert_{L^\infty(0,T;H^2(\mathbb{R}))} + \lVert \underline{\phi} - \bar{\phi} \rVert_{L^\infty(0,T;H^3(\mathbb{R}))} \leq \frac{\varepsilon_1}{2}.
\end{equation}
Combining \eqref{Dref} and \eqref{refC}, we have
\begin{equation*}
\lVert v-\bar{v} \rVert_{L^\infty(0,T;H^2(\mathbb{R}))} + \lVert u-\bar{u} \rVert_{L^\infty(0,T;H^2(\mathbb{R}))} + \lVert \phi-\bar{\phi} \rVert_{L^\infty(0,T;H^3(\mathbb{R}))} \leq \varepsilon_1
\end{equation*}
for the initial data satisfying \eqref{initcond}. The a priori estimate~\eqref{apriori} implies that \( T \) can be extended to \( +\infty \), yielding global existence of solutions near the approximate composite wave in the \( \xi \)-coordinate. The local solution $(v,u,\phi)(t,\xi)$ in this moving coordinate corresponds, via the change of variables $x=\xi+\sigma t$, to the solution in the original $(t,x)$ coordinate. This establishes the global existence assertion of Theorem~\ref{Main}.

\section{Time-asymptotic behavior}

Now we investigate the time-asymptotic behavior of the global solution obtained in Appendix~\ref{Sec_Ge}. For this purpose, we define the function $g(t)$ as
\begin{equation*}
g(t) := \lVert (\tilde{v}_\xi,\tilde{u}_\xi,\tilde{\phi}_\xi) \rVert_{L^2}^2,
\end{equation*}
where
\begin{equation*}
(\tilde{v},\tilde{u},\tilde{\phi}) (t,\xi) := (v,u,\phi)(t,\xi) - (\bar{v},\bar{u},\bar{\phi})(t,\xi).
\end{equation*}
Our goal is to verify that $g(t)$ goes to zero as $t \to +\infty$, by showing $g \in W^{1,1}(\mathbb{R}_+)$. We recall from Appendix~\ref{Sec_Ge} that \eqref{apriori} holds for all $t>0$. Thanks to \eqref{apriori} and \eqref{vxiL2}, extended to $T \to +\infty$, we have
\begin{equation} \label{gvu}
\int_0^\infty \lVert (\tilde{v}_\xi,\tilde{u}_\xi) \rVert_{L^2}^2 \, dt \leq C \lVert (v-\bar{v},u-\bar{u})(0,\cdot) \rVert_{H^2}^2 + C \delta_R^{1/3}.
\end{equation}
By \eqref{phi_est}, together with \eqref{apriori}, we also have
\begin{equation} \label{gphi}
\int_0^\infty \lVert \tilde{\phi}_\xi (t,\cdot) \rVert_{L^2}^2 \, dt \leq C \lVert (v-\bar{v},u-\bar{u})(0,\cdot) \rVert_{H^2}^2 + C \delta_R^{1/3}.
\end{equation}
Thus, we obtain that $g(t) \in L^1$.

Next we consider the first derivative $g'(t)$. First, using \eqref{eq_v}, we have
\begin{equation} \label{gvt}
\left| \frac{d}{dt} \int_\mathbb{R} \tilde{v}_\xi^2 \, d\xi \right| = \left| 2\int_\mathbb{R} \tilde{v}_\xi\tilde{v}_{\xi t} \, d\xi \right| \leq \left| 2 \int_\mathbb{R} \tilde{v}_\xi \tilde{u}_{\xi\xi} \, d\xi \right| + \left| 2 \dot{X}(t) \int_\mathbb{R} \bar{v}^S_{\xi\xi}\tilde{v}_\xi \, d\xi \right|.
\end{equation}
We apply Young's and H\"olders inequalities on the right-hand side to obtain
\begin{equation} \label{gvt1}
\begin{split}
RHS & \leq C \int_\mathbb{R} \left( \tilde{v}_\xi^2 + \tilde{u}_{\xi\xi}^2 \right) \, d\xi + C \delta_S |\dot{X}|^2 + \frac{C}{\delta_S} \left( \int_\mathbb{R} |\bar{v}^S_{\xi\xi}| |\tilde{v}_\xi| \, d\xi \right)^2 \\
& \leq C \int_\mathbb{R} \left( \tilde{v}_\xi^2 + \tilde{u}_{\xi\xi}^2 \right) \, d\xi + C \delta_S |\dot{X}|^2 + \frac{C}{\delta_S} \left( \int_\mathbb{R} |\bar{v}^S_{\xi\xi}| \, d\xi \right) \left( \int_\mathbb{R} |\bar{v}^S_{\xi\xi}| |\tilde{v}_\xi|^2 \, d\xi \right) \\
& \leq C \int_\mathbb{R} \left( \tilde{v}_\xi^2 + \tilde{u}_{\xi\xi}^2 \right) \, d\xi + C \delta_S |\dot{X}|^2,
\end{split}
\end{equation}
where in the last inequality, we used the bound $|\bar{v}^S_{\xi\xi}| \leq C|\bar{v}^S_\xi|$. Also, we obtain
\begin{equation} \label{gphit}
\begin{split}
\left| \frac{d}{dt} \int_\mathbb{R} \tilde{\phi}_{\xi}^2 \, d\xi \right| & = \left| 2 \int_\mathbb{R} \tilde{\phi}_\xi \tilde{\phi}_{\xi t} \, d\xi \right|  \leq C \left( \int_\mathbb{R} \tilde{\phi}_\xi^2 \, d\xi + \int_\mathbb{R} \tilde{\phi}_{\xi t}^2 \, d\xi \right).
\end{split}
\end{equation}
To obtain the estimate for $\tilde{u}_\xi$, we borrow the result of the proof of Lemma~\ref{uxi}
\begin{equation} \label{gut}
\left| \frac{d}{dt}\int_\mathbb{R} \tilde{u}_\xi^2 \, d\xi \right| \leq \left|2 \int_\mathbb{R} \tilde{u}_\xi \tilde{u}_{\xi t} \, d\xi \right| \leq C \int_\mathbb{R} \left( \bar{v}_\xi \tilde{v}^2 + \tilde{v}_\xi^2 + \tilde{u}_\xi^2 + \tilde{u}_{\xi\xi}^2 + \tilde{\phi}_{\xi}^2 \right) \, d\xi + r(t).
\end{equation}
where $\textstyle \int_0^\infty r(t)\, dt \leq C \delta_R^{1/2}$. Hence, applying \eqref{vL2}, \eqref{vxiL2}, \eqref{phi_est}, and \eqref{phixt_est} to \eqref{gvt}--\eqref{gut}, we have
\begin{equation} \label{gt}
\int_0^\infty |g'(t)| \, dt = \int_0^\infty \left| \frac{d}{dt} \int_\mathbb{R} \left( \tilde{v}_\xi^2 + \tilde{u}_\xi^2 + \tilde{\phi}_\xi^2 \right) \, d\xi \right| \, dt \leq C \lVert (v-\bar{v},u-\bar{u})(0,\cdot) \rVert_{H^2}^2 + \delta_R^{1/3}.
\end{equation}
by \eqref{apriori}. Collecting \eqref{gvu}, \eqref{gphi}, and \eqref{gt}, we have that $g \in W^{1,1}(\mathbb{R}_+)$. Therefore, by the interpolation inequality, we obtain
\begin{equation*}
\begin{split}
& \lVert \tilde{v} (t,\cdot) \rVert_{L^\infty} + \lVert \tilde{u} (t,\cdot) \rVert_{L^\infty} + \lVert \tilde{\phi} (t,\cdot) \rVert_{L^\infty} \\
& \quad \leq C \left( \lVert \tilde{v} \rVert_{L^2}^{1/2}\lVert \tilde{v}_\xi \rVert_{L^2}^{1/2} + \lVert \tilde{u} \rVert_{L^2}^{1/2}\lVert \tilde{u}_\xi \rVert_{L^2}^{1/2} + \lVert \tilde{\phi} \rVert_{L^2}^{1/2}\lVert \tilde{\phi}_\xi \rVert_{L^2}^{1/2} \right) \to 0 \quad \text{as } t \to +\infty,
\end{split}
\end{equation*}
which together with Lemma~\ref{rarefaction} implies \eqref{asympt}. Furthermore, by the estimate \eqref{Xbd}, we have
\begin{equation*}
\lim_{t \to +\infty}{|\dot{X}(t)|} \leq C \lim_{ t \to + \infty}{\lVert \tilde{v}(t,\cdot) \rVert_{L^\infty}} =0,
\end{equation*}
which completes the proof of Theorem~\ref{Main}.

\end{document}